\documentclass[10pt]{amsart}
\usepackage{amsmath,amssymb,latexsym,cancel,rotating}
\usepackage{graphicx,amssymb,mathrsfs,amsmath,color,fancyhdr,amsthm}
\usepackage[all]{xy}

\newtheorem{theorem}{Theorem}[section]
\newtheorem{lemma}[theorem]{Lemma}
\newtheorem{corollary}[theorem]{Corollary}
\newtheorem{definition}[theorem]{Definition}
\newtheorem{proposition}[theorem]{Proposition}

\newtheorem{remark}[theorem]{Remark}

\def\Hom{\mbox{\rm Hom}}  
       
   \def\Ker{\mbox{\rm Ker}\,}
\def\dim{\mbox{\rm dim}\,}

\def\mod{\mbox{\rm mod}\,}  
  
\begin{document}

\title[The canonical and semicanonical bases]
{The correspondence between the canonical and semicanonical bases}
\author[Fang,Lan,Xiao]{Jiepeng Fang,Yixin Lan,Jie Xiao}
\address{School of mathematical secience, Peking University, Beijing 100871, P. R. China}
\email{fangjp@math.pku.edu.cn (J.Fang)}

\address{Department of mathematical secience, Tsinghua University, Beijing 100084, P. R. China}
\email{lanyx18@mails.tsinghua.edu.cn (Y.Lan)}

\address{Department of mathematical secience, Tsinghua University, Beijing 100084, P. R. China}
\email{jxiao@tsinghua.edu.cn (J.Xiao)}


\date{\today}
\thanks{Jiepeng Fang and Yixin Lan were supported by Tsinghua University Initiative Scientific Research
	Program (No. 2019Z07L01006), and Jie Xiao was supported by NSF of China (No. 12031007).} 

\bibliographystyle{abbrv}

\begin{abstract}
	Given any symmetric Cartan datum, Lusztig has provided in \cite{MR1758244} and \cite{MR1088333} a pair of key lemmas to construct the perverse sheaves over the corresponding quiver and the functions of irreducible components over the corresponding preprojective algebra respectively. In the present article, we prove that these two inductive algorithms of Lusztig coincide. Consequently we can define two $\mathbb{Z}_{2} \times I$-colored graphs and prove that they are isomorhic. This result finishes the proof of the statement that Lusztig's functions $f_{Z}$ of irreducible components form a basis of the enveloping algebra. And we can also deduces its crystal structure (in the sense of Kashiwara-Saito \cite{MR1458969}) from those Lusztig's functions. As an application, we prove that the transition matrix between the canonical basis (at $v=1$) and the semicanonical basis is upper triangular with diagonal entries $1$.
\end{abstract}

\maketitle
\setcounter{tocdepth}{1}\tableofcontents

\section{Introduction}
\subsection{}  A symmetric Cartan datum $(I,(-,-))$ is given by a finite set $I$ and a symmetric bilinear form $(-,-):\mathbb{Z}^{I}\times \mathbb{Z}^{I} \rightarrow \mathbb{Z}$. The associated generalized Cartan matrix  $(a_{i,j})_{i,j \in I}$ satisfies $a_{i,i}=2$ and $a_{i,j}=(i,j)$ for $i \neq j$. Let  $\mathfrak{g}$ be the Kac-Moody Lie algebra associated to the generalized Cartan matrix $(a_{i,j})_{i,j \in I}$, let $\mathbf{U}(\mathfrak{g})$ be the enveloping algebra of $\mathfrak{g}$ and let $\mathbf{U}_{v}(\mathfrak{g})$ be the quantized enveloping algebra. The (quantized) enveloping algebra admits triangular decomposition.

 Given a symmetric Cartan datum, we can associate a finite graph $\mathbf{\Gamma}$ without loops such that the set of vertices is $I$ and there are exactly $|a_{i,j}|$ edges connecting  $i$ and $j$ for $i \neq j$. Choose an orientation $\Omega$, we get a quiver $Q=(I,\Omega)$. Lusztig considered the moduli space $\mathbf{E}_{\mathbf{V},\Omega}$ of representations of the quiver $Q$ and categorified the positive (or negative) part $\mathbf{U}^{+}_{v}(\mathfrak{g})$ of the quantized enveloping algebra via perverse sheaves on $\mathbf{E}_{\mathbf{V},\Omega}$ in \cite{MR1088333}. 
 
 More precisely, in \cite{MR1088333} Lusztig considered the flag variety $\tilde{\mathcal{F}}_{\underline{v},\Omega}$ for each flag type $\underline{v}$ of $\mathbf{V}$, which is smooth, and defined a proper map $\pi_{\underline{v},\Omega}: \tilde{\mathcal{F}}_{\underline{v},\Omega} \rightarrow \mathbf{E}_{\mathbf{V},\Omega}$. Let $\bar{\mathbb{Q}}_{l}|_{\tilde{\mathcal{F}}_{\underline{v},\Omega} }$ be the constant sheaf on $\tilde{\mathcal{F}}_{\underline{v},\Omega}$, then $L_{\underline{v}}= (\pi_{\underline{v},\Omega})_{!} \bar{\mathbb{Q}}_{l}$ is semisimple by \cite{MR751966}. Let $\mathcal{Q}_{\mathbf{V},\Omega}$ be the category consisting of direct sums of shifts of direct summands of such $L_{\underline{v}}$ and  let $\mathcal{K}_{\mathbf{V},\Omega}$ be the $\mathbb{Z}[v,v^{-1}]$-module spanned by $\{[L]| L \in \mathcal{Q}_{\mathbf{V},\Omega}\}$ modulo the following relations $[X \oplus Y]=[X]+[Y],[X[1]]=v[X]$. The induction and restriction functors introduced by Lusztig makes $\mathcal{K}_{\Omega}= \bigoplus\limits_{\mathbf{V}}\mathcal{K}_{\mathbf{V},\Omega}$ a bialgebra, which is isomorphic to the integeral form $_{\mathbb{Z}}\mathbf{U}^{+}_{v}(\mathfrak{g})$. The main result of Lusztig's categorification is that the set $\mathcal{P}_{\mathbf{V},\Omega}$ of simple obejects in $\mathcal{Q}_{\mathbf{V},\Omega}$ form a basis of $\mathbf{U}^{+}_{v}(\mathfrak{g})_{|\mathbf{V}|}$, which is called the canonical basis. The canonical basis has some remarable properties, like integrality and positivity.
 
 When constructing the canonical basis, Lusztig provided a key lemma (See details in \cite{MR1088333} or Lemma 3.4 in the present article) to prove that the image of each direct summand of $L_{\underline{v}}$ in $\mathcal{Q}_{\mathbf{V},\Omega}$ is contained in $\mathbf{U}^{+}_{v}(\mathfrak{g})$. After a more detailed analysis in \cite{MR1227098}, the key lemma for Lusztig's sheaves induces the Kashiwara's operators (in the sense of \cite{MR1115118}) in a categorification level and implies that Lusztig's sheaves has a crystal structure isomorphic to $B(\infty)$. Actually, with the key lemma, we can define a $\mathbb{Z}_{2} \times I$-colored graph $\mathcal{G}_{1}$ for Lusztig's sheaves, which essentially realizes the crystal graph of $B(\infty)$.
 
 \subsection{}
 Let $Q=(I,\Omega)$ be the quiver associated with the given Cartan datum, we let $\Pi(\bar{Q})=\mathbb{C}\bar{Q}/\mathcal{I}$ be the preprojective algebra of $Q$. More precisely, let $\bar{Q}=(I,\Omega \cup \bar{\Omega})$ be the double quiver of $Q$, then $\Pi(\bar{Q})$ is the path algbera of $\bar{Q}$ modulo the addmissible relation $\mathcal{I}$ generated by $$\sum\limits_{h \in \Omega \cup \bar{\Omega},h''=i} x_{h}x_{\bar{h}}-\sum\limits_{h \in \Omega \cup \bar{\Omega},h'=i}x_{\bar{h}}x_{h} ,i \in I.$$
 
 Lusztig \cite{MR1758244} introduced the moduli space $\Lambda_{\mathbf{V}}$ of nilpotent representations of $\Pi(\bar{Q})$. By considering the  convolution product of the constructible functions on $\Lambda_{\mathbf{V}}$, he
  defined an algebra $\mathcal{M}$, which is isomorphic to the positive part $\mathbf{U}^{+}(\mathfrak{g})$. For each irreducible component $Z$ of $\Lambda_{\mathbf{V}}$, Lusztig has constructed some function $f_{Z}$. With the result $|Irr \Lambda_{\mathbf{V}}|=\dim_{\mathbb{Q}} \mathbf{U}^{+}(\mathfrak{g})_{|\mathbf{V}|} $ in \cite{MR1458969}, Lusztig's functions $f_{Z}$ form a basis of $\mathbf{U}^{+}(\mathfrak{g})$, which is called the semicanonical basis by Lusztig.
  
  When constructing the functions $f_{Z}$, Lusztig provided a key lemma for constructible functions. (See details in \cite{MR1758244} or Lemma 5.4 in the present article) We have an important observation that the inductive construction of Lusztig's functions $f_{Z}$ coincides with the indcutive algorithm of Lusztig's perverse sheaves. Indeed, we can define a $\mathbb{Z}_{2}\times I$-colored graph $\mathcal{G}_{2}$ for certain equivalent classes of Lusztig's functions $f_{Z}$ and prove the following main theorem (See details in Section 7): There is a natural isomorphism of $\mathbb{Z}_{2}\times I$-colored graph $\Phi: \mathcal{G}_{1} \rightarrow \mathcal{G}_{2}$.

  Notice that the set of equivalent classes of Lusztig's functions $f_{Z} \in \mathcal{M}_{\mathbf{V}}$ can be identified with the set of irreducible components of $\Lambda_{\mathbf{V}}$, so we can deduce the crystal structure of $\bigcup \limits_{\mathbf{V}} Irr \Lambda_{\mathbf{V}}$ directly inside Lusztig's construction.   As a corollary, using the fact that  $|\mathcal{P}_{\mathbf{V},\Omega}|=\dim_{\mathbb{Q}(v)}\mathbf{U}^{+}_{v}(\mathfrak{g})=\dim_{\mathbb{Q}}\mathbf{U}^{+}(\mathfrak{g})$, we can deduce that Lusztig's functions $f_{Z}$ form a basis of $\mathbf{U}^{+}(\mathfrak{g})$ without assuming the main theorem in \cite{MR1458969} holds.
 
 \subsection{} Taking the classical limit $q \rightarrow 1$, the canonical basis becomes a basis of $\mathbf{U}^{+}(\mathfrak{g})$, still called  the canonical basis. Then the isomorphism $\Phi$ of $\mathbb{Z}_{2}\times I$-colored graphs gives the correspondence between the canonical and semicanonical bases of $\mathbf{U}^{+}(\mathfrak{g})$. Notice that there are natural constructible vector bundles $\psi_{\Omega}$ and $\tilde{\psi}_{\Omega}$, under the correspondence given by $\Phi$, our second main result (Theorem 8.6) is that the transition matrix between the canonical basis  and the semicanonical basis is upper triangular with all diagonal entries equal to $1$. Notice that  Geiss, Leclerc and Schroer  have provided an example in \cite{MR2144987} to show that these two bases are not equal. 
 
 As a byproduct, we also obtain a construction of monomial basis of $\mathbf{U}^{+}(\mathfrak{g})$ depending on an order of $I$.
 
 \subsection{} In Section 2, we define the convolution algebras $\mathcal{M}_{\Omega}$ and $\mathcal{M}$ of the quiver and its preprojective algebra respectively and prove that the pushforward of the natural projection $\psi_{\Omega}$ induces an isomorphism of algebras from $\mathcal{M}$ to $\mathcal{M}_{\Omega}$. In Section 3, we recall the geometric realization of $\mathbf{U}^{+}_{v}(\mathfrak{g})$ by Lusztig \cite{MR1088333} and the key inductive lemmas for perverse sheaves. In Section 4, we define the $\mathbb{Z}_{2} \times I$-colored graph for Lusztig's perverse sheaves and study the commutative relations of arrows in the graph. In Section 5, we recall the construction of Lusztig's functions $f_{Z}$ in \cite{MR1758244} and the key inductive lemmas for those $f_{Z}$. In Section 6, we define the $\mathbb{Z}_{2} \times I$-colored graph for equivalent classes of Lusztig's functions and study the commutative relations of arrows in the graph.  In Section 7, we prove our first main result that  the $\mathbb{Z}_{2} \times I$-colored graphs defined in Section 4 and Section 6 are isomorphic and deduce the results in \cite{MR1458969}. In Section 8, with an order of $I$, we define an order of vertices of $\mathbb{Z}_{2} \times I$-colored graphs and  construct a monomial basis. We also prove the second main theorem: The transition matrix between the canonical basis and the semicanonical basis is upper triangular with all diagonal entries equal to $1$.
\subsection*{Acknowledgement}
 We are grateful to George Lusztig. After he saw our first version of the present article on arxiv, he informed us the preprint \cite{baumann2011canonical} by Pierre Baumann, in which similar results have been obtained. (See Remark 7.4 and Remark 8.10.)

\section{Convolution algebras on $\mathbf{E}_{\mathbf{V},\Omega}$ and $\Lambda_{\mathbf{V}}$} 
\subsection{The variety $\mathbf{E}_{\mathbf{V},\Omega}$ and $\Lambda_{\mathbf{V}}$}
Given a symmetric Cartan datum $(I,(-,-))$, let $\mathbf{\Gamma}$ be the finite graph without loops associated to $(I,(-,-))$. Let $I$ be its set of vertices and $H$ be the set of pairs consisting of egdes with orientation. More precisely, to give an egde with orientation is equivalent to give $h',h'' \in I$ and we adapt the notation $h' \xrightarrow{h} h''$. Let $-:h \mapsto \bar{h}$ be the involution of $H$ such that $\bar{h}'=h'',\bar{h}''=h'$ and $\bar{h} \neq h$. An orientation of the graph $\Gamma$ is a subset $\Omega \subset H$ such that $\Omega \cap \bar{\Omega} =\emptyset$ and $\Omega \cup \bar{\Omega} = H$.

 Let $k =\bar{k}$ be an algebraic closed field. Given $\upsilon \in \mathbb{N}^{I}$, a subset $\tilde{H} \subseteq H$ and a graded $k$-vector space $\mathbf{V}$ with dimension vector $|\mathbf{V}|=\upsilon$,  we let
\begin{center}
	$\mathbf{E}_{\mathbf{V}}= \bigoplus\limits_{h \in H} \mathbf{Hom}(\mathbf{V}_{h'},\mathbf{V}_{h''})$ \\
	 $\mathbf{E}_{\mathbf{V}, \tilde{H}}= \bigoplus\limits_{h \in \tilde{H}} \mathbf{Hom}(\mathbf{V}_{h'},\mathbf{V}_{h''})$
\end{center} 
In particular, for an orientation $\Omega$, we let 
\begin{center}
	$\mathbf{E}_{\mathbf{V}, \Omega}= \bigoplus\limits_{h \in \Omega} \mathbf{Hom}(\mathbf{V}_{h'},\mathbf{V}_{h''})$
\end{center}

The algebraic group $G_{\mathbf{V}}= \prod\limits_{i \in I} \mathbf{GL}(\mathbf{V}_{i})$ acts on $\mathbf{E}_{\mathbf{V}},\mathbf{E}_{\mathbf{V},\tilde{H}}$ and $\mathbf{E}_{\mathbf{V}, \Omega}$ by $(g \cdot x)_{h} =g_{h''} x_{h} g_{h'}^{-1} $.

 We fix a function $\epsilon:H \rightarrow k^{\ast}$ such that $\epsilon(h) + \epsilon(\bar{h})=0$ for all $h$. The moment map  $\mu: \mathbf{E}_{\mathbf{V}} \rightarrow \mathbf{gl}_{\mathbf{V}}= \bigoplus\limits_{i \in I} \mathbf{End}(\mathbf{V}_{i})$ is defined by $(\mu(x))_{i}= \sum\limits_{h \in H, h''=i} \epsilon(h)x_{h}x_{\bar{h}}$. We let $\Lambda_{\mathbf{V}}$ be the closed subvariety of $\mathbf{E}_{\mathbf{V}}$ consisting of nilponent elements $x$ such that $\mu(x)=0$. It is easy to see that $\Lambda_{\mathbf{V}}$ is $G_{\mathbf{V}}$-stable.

 \subsection{Convolution algebra arising from $\mathbf{E}_{\mathbf{V},\Omega}$}

 If $f:X \rightarrow Y$ is a morphism of varieties and $g$ is a constructible function on $Y$, then following \cite{MR361141}, we define $f^{\ast}(g)$ to be the constructible function such that $f^{\ast}(g)(x)=g(f(x))$.  If $h$ is a constructible function on $X$, then we define $f_{!}(h)$ to be the constructible function such that $f_{!}(h)(y)=\sum \limits_{c \in \mathbb{Q}}\chi(h^{-1}(c) \cap f^{-1}(y))$. Here $\chi(X)$ means the Euler characteristic for the constructible set $X$.

 Assume that $k=\mathbb{C}$ and we fix an orientation $\Omega$. Given $\upsilon'+\upsilon''=\upsilon \in \mathbb{N}^{I} $ and graded vector spaces $\mathbf{V},\mathbf{V}',\mathbf{V}''$ with dimension vectors $\upsilon, \upsilon', \upsilon''$ respectively, let $\mathbf{E}'_{\Omega}$ be the set consisting of $(x,\tilde{\mathbf{W}}, \rho_{1}, \rho_{2})$ where $x \in \mathbf{E}_{\mathbf{V},\Omega}$,$\tilde{\mathbf{W}}$ is a $x$-stable  subspace of $\mathbf{V}$ with dimension vector $\upsilon''$ and $ \rho_{1}: \mathbf{V}/\tilde{\mathbf{W}} \simeq \mathbf{V}',\rho_{2}:\tilde{\mathbf{W}} \simeq \mathbf{V}''$ are linear isomorphisms. We also let $\mathbf{E}''_{\Omega}$ be the set consisting of $(x,\tilde{\mathbf{W}})$ as above.  Here we say $\tilde{\mathbf{W}}$ is $x$-stable if and only if $x_{h}(\tilde{\mathbf{W}}_{h'}) \subset \tilde{\mathbf{W}}_{h''}$ for any $h \in \Omega$. 
 
 Lusztig has introduced the following diagram in \cite{MR1088333}:
 \begin{center}
 	$\mathbf{E}_{\mathbf{V}',\Omega} \times \mathbf{E}_{\mathbf{V}'',\Omega} \xleftarrow{p_{1}} \mathbf{E}'_{\Omega} \xrightarrow{p_{2}} \mathbf{E}''_{\Omega} \xrightarrow{p_{3}} \mathbf{E}_{\mathbf{V},\Omega}$
 \end{center}
 where $p_{1}(x,\tilde{\mathbf{W}},\rho_{1},\rho_{2})=(\rho_{1,\ast}(\bar{x}|_{\mathbf{V}/\tilde{\mathbf{W}}}),\rho_{2,\ast}(x|_{\tilde{\mathbf{W}}})  )$, $p_{2}(x,\tilde{\mathbf{W}},\rho_{1},\rho_{2}) =(x, \tilde{\mathbf{W}}) $ and $p_{3}(x,\tilde{\mathbf{W}})=x$. Here $x|_{\tilde{\mathbf{W}}}$ is the restriction of $x$ on the subspace $\tilde{\mathbf{W}}$ and $\bar{x}|_{\mathbf{V}/\tilde{\mathbf{W}}}$ is the natural induced linear map of $x$ on the quotient space $\mathbf{V}/\tilde{\mathbf{W}}$. $ \rho_{1,\ast}(\bar{x}|_{\mathbf{V}/\tilde{\mathbf{W}}})= \rho_{1}\circ (\bar{x}|_{\mathbf{V}/\tilde{\mathbf{W}}}) \circ \rho_{1}^{-1}$ is the element in  $\mathbf{E}_{\mathbf{V}'}$ induced by $\bar{x}|_{\mathbf{V}/\tilde{\mathbf{W}}}$ and $\rho_{2,\ast}(x|_{\tilde{\mathbf{W}}})=\rho_{2} \circ (x|_{\tilde{\mathbf{W}}})  \circ \rho_{2}^{-1}$ is the element in $\mathbf{E}_{\mathbf{V}''} $ induced by $x|_{\tilde{\mathbf{W}}}$.
  
  Notice that  $p_{1}$ is smooth with connected fibers, $p_{2}$ is a principle $G_{\mathbf{V}'} \times G_{\mathbf{V}''}$-bundle and $p_{3}$ is proper. (See details in \cite{MR1088333}.)

 Now let $\tilde{M}(\mathbf{E}_{\mathbf{V},\Omega})$ be the $\mathbb{Q}$-space consisting of all $G_{\mathbf{V}}$-invariant constructible functions $f: \mathbf{E}_{\mathbf{V},\Omega} \rightarrow \mathbb{Q}$. We define a bilinear operator $\ast:\tilde{M}(\mathbf{E}_{\mathbf{V}',\Omega}) \times \tilde{M}(\mathbf{E}_{\mathbf{V}'',\Omega}) \rightarrow \tilde{M}(\mathbf{E}_{\mathbf{V},\Omega}) $ as the following: Given $f_{1} \in \tilde{M}(\mathbf{E}_{\mathbf{V}',\Omega}), f_{2} \in \tilde{M}(\mathbf{E}_{\mathbf{V}'',\Omega})$, $f_{1} \ast f_{2}= (p_{3})_{!} f''$  where $f''$ is the unique constructible function on $\mathbf{E}''_{\Omega}$ such that $p_{2}^{\ast}(f'') =p_{1}^{\ast}(f_{1} \otimes f_{2})$. Lusztig has proved in \cite{MR1088333} that $\ast$ makes $\tilde{\mathcal{M}}_{\Omega} =\bigoplus\limits_{\mathbf{V}} \tilde{M}(\mathbf{E}_{\mathbf{V},\Omega})$ be an associative algebra. This is a constructible function version of Ringle-Hall algebra in \cite{MR1062796}.

 Denote the constant function (taking value $1$) on $\mathbf{E}_{\mathbf{V},\Omega}$ with $|\mathbf{V}|=mi$ by $\mathbf{1}_{i^{m}}=\mathbf{1}_{i^{(m)}}$ and let $\mathcal{M}_{\Omega}$ be the subalgebra of $\tilde{\mathcal{M}}_{\Omega}$ generated by those  $\mathbf{1}_{i^{m}}$. We denote $\mathcal{M}_{\Omega} \cap \tilde{M}(\mathbf{E}_{\mathbf{V},\Omega})$ by $\mathcal{M}_{\mathbf{V},\Omega}$. Then following Section 10.19 of \cite{MR1088333} or \cite{MR1062796}, we can see that $\mathcal{M}_{\Omega}$ is isomorphic to the envoloping algebra $\mathbf{U}^{+}=\mathbf{U}^{+}(\mathfrak{g})$ of the Kac-Moody Lie algebra $\mathfrak{g}$ associated with the Cartan datum which corresponds to the given graph $\mathbf{\Gamma}$.

 \subsection{Convolution algebra arising from $\Lambda_{\mathbf{V}}$}
 Given $\upsilon'+\upsilon''=\upsilon \in \mathbb{N}^{I} $ and graded vector spaces $\mathbf{V},\mathbf{V}',\mathbf{V}''$ with dimension vectors $\upsilon, \upsilon', \upsilon''$ respectively, we let  $\Lambda'$ be the set consisting of $(x,\tilde{\mathbf{W}},\rho_{1},\rho_{2})$ where $x \in \Lambda_{\mathbf{V}}$, $\tilde{\mathbf{W}}$ is a $x$-stable subspace of $\mathbf{V}$ with dimension $\upsilon''$ and $\rho_{1}:\mathbf{V}/\tilde{\mathbf{W}} \simeq \mathbf{V}', \rho_{2}: \tilde{\mathbf{W}} \simeq \mathbf{V}''$ are linear isomorphisms. We also let $\Lambda''$ be the set consisting of $(x,\tilde{\mathbf{W}})$ as above. Then there is a diagram in \cite{MR1088333}:
 \begin{center}
 	$\Lambda_{\mathbf{V}'} \times \Lambda_{\mathbf{V}''} \xleftarrow{p} \Lambda' \xrightarrow{r} \Lambda'' \xrightarrow{q} \Lambda_{\mathbf{V}}$
 \end{center}
 $p(x,\tilde{\mathbf{W}},\rho_{1},\rho_{2})=(\rho_{1,\ast}(\bar{x}|_{\mathbf{V}/\tilde{\mathbf{W}}}),\rho_{2,\ast}(x|_{\tilde{\mathbf{W}}})  )$, $r(x,\tilde{\mathbf{W}},\rho_{1},\rho_{2}) =(x, \tilde{\mathbf{W}}) $ and $q(x,\tilde{\mathbf{W}})=x$. Notice that $r$ is a principle $G_{\mathbf{V}'} \times G_{\mathbf{V}''}$-bundle and $q$ is proper.

 Now let $\tilde{M}(\Lambda_{\mathbf{V}})$ be the $\mathbb{Q}$-space consisting of all $G_{\mathbf{V}}$-invariant constructible functions $f: \Lambda_{\mathbf{V}} \rightarrow \mathbb{Q}$. We define a bilinear operator $\ast:\tilde{M}(\Lambda_{\mathbf{V}'}) \times \tilde{M}(\Lambda_{\mathbf{V}''}) \rightarrow \tilde{M}(\Lambda_{\mathbf{V}}) $ as the following: Given $f_{1} \in \tilde{M}(\Lambda_{\mathbf{V}'}), f_{2} \in \tilde{M}(\Lambda_{\mathbf{V}''})$, $f_{1} \ast f_{2}= q_{!} f''$  where $f''$ is the unique constructible function on $\Lambda''$ such that $r^{\ast}(f'') =p^{\ast}(f_{1} \otimes f_{2})$. Then $\ast$ makes $\tilde{\mathcal{M}} =\bigoplus\limits_{\mathbf{V}} \tilde{M}(\Lambda_{\mathbf{V}})$ be an associative algebra.
 Denote the constant function (taking value $1$) on $\Lambda_{\mathbf{V}}$ with $|\mathbf{V}|=mi$ by $\mathbf{1}_{i^{m}}=\mathbf{1}_{i^{(m)}}$ and let $\mathcal{M}$ be the $\mathbb{Q}$-subalgebra of $\tilde{\mathcal{M}}$ generated by those  $\mathbf{1}_{i^{m}}$. We  denote $\mathcal{M} \cap \tilde{M}(\Lambda_{\mathbf{V}})$ by $\mathcal{M}_{\mathbf{V}}$ and denote the integral form  of $\mathcal{M}$ by $_{\mathbb{Z}}\mathcal{M}$. More precisely, $_{\mathbb{Z}}\mathcal{M}=span_{\mathbb{Z}}\{ \mathbf{1}_{i_{1}^{m_{1}}} \ast \mathbf{1}_{i_{2}^{m_{2}}} \cdots \ast \mathbf{1}_{i_{l}^{m_{l}}}|l \in \mathbb{N}, i_{j} \in I, m_{j} \in \mathbb{N}$ for any  $1 \leq j \leq l \} $.

 \subsection{The flag varieties and the map $\psi_{\Omega}$} 
 
 For any $\underline{v}=(v^{1},v^{2},\cdots, v^{m})$ with each $v^{l} \in \mathbb{N}^{I}$, let $\tilde{\mathcal{F}}_{\underline{v}}$ be the variety consisting of $(x,f)$ where $x \in \Lambda_{\mathbf{V}}$ and $f=(0=\mathbf{V}^{m} \subset \mathbf{V}^{m-1} \cdots \subset \mathbf{V}^{1} \subset \mathbf{V}^{0}=\mathbf{V})$ is an $x$-stable flag of type $\underline{v}$. More precisely, $\mathbf{dim} (\mathbf{V}^{l-1}/\mathbf{V}^{l}) =v^{l}$ for any $l$. There is a proper morphism  $\pi_{\underline{v}}:\tilde{\mathcal{F}}_{\underline{v}} \rightarrow \Lambda_{\mathbf{V}}$. Denote the constant function on $\tilde{\mathcal{F}}_{\underline{v}}$ by $\mathbf{1}_{\tilde{\mathcal{F}}_{\underline{v}}}$.
 
 Then we have the following result:
 \begin{lemma}
 Let $\underline{v}=(v^{1},v^{2},\cdots,v^{l})$ be a flag type such that $v^{k}=m_{k}i_{k}$ for $1 \leq k \leq l$, then we have  $(\pi_{\underline{v}}) _{!}(\mathbf{1}_{\tilde{\mathcal{F}}_{\underline{v}}})=	\mathbf{1}_{i_{1}^{m_{1}}} \ast \mathbf{1}_{i_{2}^{m_{2}}} \cdots \ast \mathbf{1}_{i_{l}^{m_{l}}}$.
 \end{lemma}
 \begin{proof}
 	We prove the lemma by induction on $l$. When $l=1$, the lemma holds trivially.
 	Now we assume $\mathbf{1}_{i_{1}^{m_{1}}} \ast \mathbf{1}_{i_{2}^{m_{2}}} \cdots \ast \mathbf{1}_{i_{l-1}^{m_{l-1}}}= (\pi_{\underline{v}'}) _{!}(\mathbf{1}_{\tilde{\mathcal{F}}_{\underline{v}'}})$ where $\underline{v'}=(v^{1},v^{2},\cdots v^{l-1})$. Now we fix a decomposition $\mathbf{V}=\mathbf{V}'\oplus \mathbf{V}''$  with  $|\mathbf{V}''|=v^{l}$ and $|\mathbf{V}'|+|\mathbf{V}''|=|\mathbf{V}|$. Let $F$ be the subvariety of $\Lambda_{\mathbf{V}}$ consisting of $x\in \Lambda_{\mathbf{V}}$ such that $\mathbf{V'}$ is $x$-stable  and $\tilde{\mathcal{F}}_{\underline{v},0}$ be the subvariety of $\tilde{\mathcal{F}}_{\underline{v}}$ consisting of $(x,f)$ such that the $(l-1)-th$ component $\mathbf{V}^{l-1}$ is exactly $\mathbf{V}'$. Then the restriction of $\pi_{\underline{v}}$ on $\tilde{\mathcal{F}}_{\underline{v},0}$ gives a morphism $u: \tilde{\mathcal{F}}_{\underline{v},0} \rightarrow F$ and it induces a proper map $\tilde{u}:G_{\mathbf{V}} \times_{Q_{\mathbf{V}}} \tilde{\mathcal{F}}_{\underline{v},0}\rightarrow G_{\mathbf{V}} \times_{Q_{\mathbf{V}}} F = \Lambda'' $.  Here $Q_{\mathbf{V}}$ is the stablizer of $\mathbf{V'}$ in $G_{\mathbf{V}}$. Then we can check by definition that 
 	\begin{center}
 		$ r^{\ast} (\tilde{u})_{!} (\mathbf{1}_{G_{\mathbf{V}} \times_{Q_{\mathbf{V}}} \tilde{\mathcal{F}}_{\underline{v},0} })= p^{\ast}((\pi_{\underline{v}'}) _{!}(\mathbf{1}_{\tilde{\mathcal{F}}_{\underline{v}'}}) \otimes \mathbf{1}_{i_{l}^{m_{l}}}).$
 	\end{center}
 	 Notice that $(r\tilde{u})_{!}( (\mathbf{1}_{G_{\mathbf{V}} \times_{Q_{\mathbf{V}}} \tilde{\mathcal{F}}_{\underline{v},0} } )= (\pi_{\underline{v}}) _{!}(\mathbf{1}_{\tilde{\mathcal{F}}_{\underline{v}}})$, we are done.
 \end{proof} 

\begin{remark}
	By the above lemma, it is easy to see that $(\pi_{\underline{v}'}) _{!}(\mathbf{1}_{\tilde{\mathcal{F}}_{\underline{v}'}}) \ast (\pi_{\underline{v}''}) _{!}(\mathbf{1}_{\tilde{\mathcal{F}}_{\underline{v}''}})=(\pi_{\underline{v}'\underline{v}''}) _{!}(\mathbf{1}_{\tilde{\mathcal{F}}_{\underline{v}'\underline{v}''}})$. This is an analogy of Lemma 3.2 in \cite{MR1088333}.
\end{remark}

Similarly, we can define $\tilde{\mathcal{F}}_{\underline{v},{\Omega}}$ to be the smooth variety consisting of $(x,f)$ where $x \in \mathbf{E}_{\mathbf{V},\Omega}$ and $f=(0=\mathbf{V}^{m} \subset \mathbf{V}^{m-1} \cdots \subset \mathbf{V}^{1} \subset \mathbf{V}^{0}=\mathbf{V})$ is an $x$-satble flag of type $\underline{v}$. More precisely, $\mathbf{dim} (\mathbf{V}^{l-1}/\mathbf{V}^{l}) =v^{l}$ for any $l$. There is a proper morphism  $\pi_{\underline{v},\Omega}:\tilde{\mathcal{F}}_{\underline{v}} \rightarrow E_{\mathbf{V},\Omega}$. Denote the constant function on $\tilde{\mathcal{F}}_{\underline{v},\Omega}$ by $\mathbf{1}_{\tilde{\mathcal{F}}_{\underline{v},\Omega}}$. Then by the same argument in Lemma 2.1, the following equation holds in in $\mathcal{M}_{\Omega}$ for a flag type $\underline{v}=(v^{1},v^{2},\cdots,v^{l})$ satisfying $v^{k}=m_{k}i_{k}$ for $1 \leq k \leq l$:
\begin{center}
	$(\pi_{\underline{v},\Omega})_{!}(\mathbf{1}_{\tilde{\mathcal{F}}_{\underline{v},\Omega}} ) =\mathbf{1}_{i_{1}^{m_{1}}} \ast \mathbf{1}_{i_{2}^{m_{2}}} \cdots \ast \mathbf{1}_{i_{l}^{m_{l}}}.$
\end{center} 
\
\

 Fix an orientation $\Omega$, there is a natural map $\psi_{\Omega}: \mu^{-1}(0) \rightarrow \mathbf{E}_{\mathbf{V},\Omega}$ defined by forgetting the components in $\bar{\Omega}$:
 \begin{center}
 	$\psi_{\Omega}((x_{h})_{h\in H}) =((x_{h})_{h \in \Omega}). $
 \end{center}
  Then $\psi_{\Omega}$ induces a natural morphism $\tilde{\psi}_{\Omega}: \tilde{\mathcal{F}}_{\underline{v}} \rightarrow \tilde{\mathcal{F}}_{\underline{v},\Omega}$, which is defined by forgetting the components of $x \in \Lambda_{\mathbf{V}}$  in $\bar{\Omega}$:
 \begin{center}
 	$\tilde{\psi}_{\Omega}((x_{h})_{h\in H},f) =((x_{h})_{h \in \Omega},f). $
 \end{center}

 \begin{proposition}
 	 Each fiber of the map $\psi_{\Omega}$ and $\tilde{\psi}_{\Omega}$ is isomorphic to an affine space. Moreover, we have the following commutative diagram:
 	 
 	 \[
 	 \xymatrix{
 	 	\tilde{\mathcal{F}}_{\underline{v}}\ar[d]^{\pi_{\underline{v}}} \ar[r]^{\tilde{\psi}_{\Omega}} & \tilde{\mathcal{F}}_{\underline{v},\Omega}\ar[d]^{\pi_{\underline{v},\Omega}}\\
 	 	\mu^{-1}(0) \ar[r]^{\psi_{\Omega}} & \mathbf{E}_{\mathbf{V},\Omega}
 	 }
 	 \]
 	 
 \end{proposition}

\begin{proof}
	To give a fiber of $\psi_{\Omega}$ at $(x_{h})_{h \in \Omega}$ is equivalent to give the solution $(x_{h})_{h \in \bar{\Omega}}$ of $\mu(x)=0$, which is a family of homogenous linear equations for fixed $(x_{h})_{h \in \Omega}$, hence each fiber of $\psi_{\Omega}$ is an affine space. 
	
	Notice that if $x \in \mu^{-1}(0)$ fixes a flag, then $x$ is automatically nilpotent. We can see that for a fixed $(x_{h})_{h \in \Omega} \in \mathbf{E}_{\mathbf{V}}$ and an $(x_{h})_{h \in \Omega}$-stable flag $f=(0=\mathbf{V}^{m} \subset \mathbf{V}^{m-1} \cdots \subset \mathbf{V}^{1} \subset \mathbf{V}^{0}=\mathbf{V})$,  to give a fiber of $\tilde{\psi}_{\Omega}$ at $(x,f)$ is equivalent to give the solution of $\mu(x)=0$ and $x_{h}(\mathbf{V}^{l}_{h'}) \subseteq \mathbf{V}^{l}_{h''}$ for any $0 \leq l \leq m$ and $h \in \bar{\Omega}$, which is also a family of homogenous linear  equations. Hence each fiber of $\tilde{\psi}_{\Omega}$ is an affine space. 
	
	The diagram commutes by definition.
	
\end{proof}
 
 After extensions by zero, the constructible functions on $\Lambda_{\mathbf{V}}$ can be naturally regarded as constructible functions on $\mu^{-1}(0)$. Then $(\psi_{\Omega})_{!}: \mathcal{M} \rightarrow \mathcal{M}_{\Omega}$ is well-defined.
\begin{corollary}
	 The map $\psi_{\Omega}$ induces a surjective homomorphism of algebras  $(\psi_{\Omega})_{!}: \mathcal{M} \rightarrow \mathcal{M}_{\Omega}$.
\end{corollary}
\begin{proof}
	We only need to prove that $(\psi_{\Omega})_{!}((\pi_{\underline{v}})_{!}(\mathbf{1}_{\tilde{\mathcal{F}}_{\underline{v}}} ))=(\pi_{\underline{v},\Omega})_{!}(\mathbf{1}_{\tilde{\mathcal{F}}_{\underline{v},\Omega}} )$.
	
	Since each fiber of $\tilde{\psi}_{\Omega}$ is isomorphic to an affine space , whose Euler characteristic equals to $1$, $(\tilde{\psi}_{\Omega})_{!}(\mathbf{1}_{\tilde{\mathcal{F}}_{\underline{v}}})= \mathbf{1}_{\tilde{\mathcal{F}}_{\underline{v},\Omega}}$. Since $\psi_{\Omega} \pi_{\underline{v}}= \pi_{\underline{v},\Omega} \tilde{\psi}_{\Omega}$, we are done.
\end{proof}

Lusztig has proved that the generators $\mathbf{1}_{i^{m}}$ satisfy the Serre relation in $\mathcal{M}$ (see details in\cite{MR1088333} Lemma 12.11), hence there is a surjective algebra morphism $\Psi: \mathbf{U}^{+} \rightarrow \mathcal{M}$ given by $\Psi(e_{i}^{(m)})=\mathbf{1}_{i^{m}}$. Here $e_{i}^{(m)}$ is the $m$-th divided power of the Chevalley generator $e_{i}$. Combining this result with the corollary above, we can see that the following theorem holds:
\begin{theorem}
	The map $(\psi_{\Omega})_{!}: \mathcal{M} \rightarrow \mathcal{M}_{\Omega}$ is indeed an isomorphism of algebras. 
\end{theorem}
\begin{proof}
	The composition $\mathbf{U}^{+} \xrightarrow{\Psi} \mathcal{M} \xrightarrow{(\psi_{\Omega})_{!}} \mathcal{M}_{\Omega} \cong \mathbf{U}^{+}$ is surjective, sending $e_{i}^{(p)}$ to itself. Hence it is an isomorphism and so does  $(\psi_{\Omega})_{!}$.
\end{proof}

Given $\upsilon'+\upsilon''=\upsilon \in \mathbb{N}^{I} $ and graded vector spaces $\mathbf{V},\mathbf{V}',\mathbf{V}''$ with dimension vectors $\upsilon, \upsilon', \upsilon''$ respectively, assume $\underline{v}'=((v')^{1},(v')^{2},\cdots,(v')^{m'})$ is a flag type of $\mathbf{V}'$ and  $\underline{v}''=((v'')^{1},(v'')^{2},\cdots,(v'')^{m''})$ is a flag type of $\mathbf{V}''$. Let $\tilde{\mathcal{F}}_{\underline{v}'\underline{v}''}'$ be the set consisting of $(x,f,\rho_{1},\rho_{2})$ where $x \in \Lambda_{\mathbf{V}}$, $f=(0=\mathbf{V}^{m'+m''} \subset \mathbf{V}^{m'+m''-1} \cdots \subset \mathbf{V}^{1} \subset \mathbf{V}^{0}=\mathbf{V})$ is an $x$-stable flag of type $\underline{v}'\underline{v}''$ and $\rho_{1}:\mathbf{V}/\mathbf{V}^{m''} \cong \mathbf{V}', \rho_{2}:\mathbf{V}^{m''} \cong \mathbf{V}''$ are linear isomorphisms. Similarly, let $\tilde{\mathcal{F}}_{\underline{v}'\underline{v}'',\Omega}'$ be the set consisting of $(x,f,\rho_{1},\rho_{2})$ where $x \in \mathbf{E}_{\mathbf{V},\Omega}$, $f=(0=\mathbf{V}^{m'+m''} \subset \mathbf{V}^{m'+m''-1} \cdots \subset \mathbf{V}^{1} \subset \mathbf{V}^{0}=\mathbf{V})$ is an $x$-stable flag of type $\underline{v}'\underline{v}''$ and $\rho_{1}:\mathbf{V}/\mathbf{V}^{m''} \cong \mathbf{V}', \rho_{2}:\mathbf{V}^{m''} \cong \mathbf{V}''$ are linear isomorphisms. 
There are natural morphisms defined as the following:
$$\tilde{\mathcal{F}}_{\underline{v}'\underline{v}''}' \xrightarrow{\tilde{r}} \tilde{\mathcal{F}}_{\underline{v}'\underline{v}''}: (x,f,\rho_{1},\rho_{2}) \mapsto (x,f),$$	
$$ \tilde{\mathcal{F}}_{\underline{v}'\underline{v}''} \rightarrow \Lambda'': (x,f) \mapsto (x, \mathbf{V}^{m''}),$$
$$ \tilde{\mathcal{F}}_{\underline{v}'\underline{v}''}' \rightarrow \Lambda': (x,f,\rho_{1},\rho_{2}) \mapsto (x, \mathbf{V}^{m''},\rho_{1},\rho_{2}),$$
and
$$\tilde{\mathcal{F}}_{\underline{v}'\underline{v}''}' \xrightarrow{\tilde{p}} \tilde{\mathcal{F}}_{\underline{v}'}\times \tilde{\mathcal{F}}_{\underline{v}''} : (x,f,\rho_{1},\rho_{2}) \mapsto ((\rho_{1,\ast}(\bar{x}|_{\mathbf{V}/\mathbf{V}^{m''} }),f_{1}  )( ,\rho_{2,\ast}(x|_{\mathbf{V}^{m''}}),f_{2}  )),$$	
where $f_{1}$ is the flag of $\mathbf{V}' \cong\mathbf{V}/\mathbf{V}^{m''}$ induced from $f$ and $f_{2}$ is the flag of $\mathbf{V}'' \cong \mathbf{V}^{m''} $ induced from $f$. 

Similarly, we can set:
$$\tilde{\mathcal{F}}_{\underline{v}'\underline{v}'',\Omega}' \xrightarrow {\tilde{p}_{2}}\tilde{\mathcal{F}}_{\underline{v}'\underline{v}'',\Omega}: (x,f,\rho_{1},\rho_{2}) \mapsto (x,f),$$
$$ \tilde{\mathcal{F}}_{\underline{v}'\underline{v}'',\Omega} \rightarrow \mathbf{E}_{\Omega}'': (x,f) \mapsto (x, \mathbf{V}^{m''}),$$
$$ \tilde{\mathcal{F}}_{\underline{v}'\underline{v}'',\Omega}' \rightarrow \mathbf{E}_{\Omega}': (x,f,\rho_{1},\rho_{2}) \mapsto (x, \mathbf{V}^{m''},\rho_{1},\rho_{2}),$$
and
$$\tilde{\mathcal{F}}_{\underline{v}'\underline{v}'',\Omega}' \xrightarrow{\tilde{p}_{1}} \tilde{\mathcal{F}}_{\underline{v}',\Omega}\times \tilde{\mathcal{F}}_{\underline{v}'',\Omega} : (x,f,\rho_{1},\rho_{2}) \mapsto (\rho_{1,\ast}(\bar{x}|_{\mathbf{V}/\mathbf{V}^{m''} }),f_{1}  )( ,\rho_{2,\ast}(x|_{\mathbf{V}^{m''}}),f_{2}  ).$$

Then we have the following commutative diagram:

\[
\xymatrix@C=1em@R=3ex{
	& \tilde{\mathcal{F}}_{\underline{v}'} \times \tilde{\mathcal{F}}_{\underline{v}''} \ar[dd] \ar[dl]
	& & \tilde{\mathcal{F}}_{\underline{v}'\underline{v}''}' \ar[dd] \ar[rr]^{\tilde{r}} \ar[dl] \ar[ll]_{\tilde{p}}
	& & \tilde{\mathcal{F}}_{\underline{v}'\underline{v}''} \ar[dd] \ar@{=}[rr] \ar[dl]
	& & \tilde{\mathcal{F}}_{\underline{v}'\underline{v}''}\ar[dd] \ar[dl]
	\\
	\Lambda_{\mathbf{V}'} \times \Lambda_{\mathbf{V}''} \ar[dd]
	& & \Lambda' \ar[dd] \ar[rr]^{r} \ar[ll]_{p}
	& & \Lambda'' \ar[dd] \ar[rr]^{q}
	& & \Lambda_{\mathbf{V}} \ar[dd]
	\\
	& \tilde{\mathcal{F}}_{\underline{v}',
		\Omega}\times \tilde{\mathcal{F}}_{\underline{v}'',\Omega} \ar[dl]
	& & \tilde{\mathcal{F}}_{\underline{v}'\underline{v}'',\Omega}' \ar[rr]^{\tilde{p}_{2}} \ar[dl] \ar[ll]_{\tilde{p}_{1}}
	& & \tilde{\mathcal{F}}_{\underline{v}'\underline{v}'',\Omega} \ar@{=}[rr] \ar[dl]
	& & \tilde{\mathcal{F}}_{\underline{v}'\underline{v}'',\Omega}  \ar[dl]
	\\
	\mathbf{E}_{\mathbf{V}',\Omega} \times \mathbf{E}_{\mathbf{V}'',\Omega}
	& & \mathbf{E}'_{\Omega}  \ar[rr]^{p_{2}}  \ar[ll]_{p_{1}}
	& & \mathbf{E}''_{\Omega} \ar[rr]^{p_{3}}
	& & \mathbf{E}_{\mathbf{V},\Omega}
}
\]
where all morphisms from up to down are naturally induced by $\psi_{\Omega}$, whose fibers are affine spaces, and the other outward morphisms are the naturally projection given by forgetting flags. 
The commuative diagram above gives another explanation of Corollary 2.4 and suggests that the two different construction of $\mathbf{U}^{+}(\mathfrak{g})$ have more detailed relations.

 \section{Lusztig's sheaves and the quantized envoloping algebra}
 \subsection{Induction functor and restriction functor}
Consider the following diagram defined in Section 2.2: 
 \begin{center}
 	$\mathbf{E}_{\mathbf{V}',\Omega} \times \mathbf{E}_{\mathbf{V}'',\Omega} \xleftarrow{p_{1}} \mathbf{E}'_{\Omega} \xrightarrow{p_{2}} \mathbf{E}''_{\Omega} \xrightarrow{p_{3}} \mathbf{E}_{\mathbf{V},\Omega}$
 \end{center}
Let $d_{1}$ be the dimension of the fibers of $p_{1}$ and $d_{2}$ be the dimension of the fibers of $p_{2}$. \\

 We fix a decomposition $\mathbf{T} \oplus \mathbf{W} =\mathbf{V}$ of graded vector space, let $F_{\Omega}$ be the closed subvariety of $E_{\mathbf{V}}$ consisting of $x$ such that $\mathbf{W}$ is $x$-stable. Let $Q_{\mathbf{V}}$ be the stablizer of $\mathbf{W}$ in $G_{\mathbf{V}}$ and let $d_{3}$ be the dimension of $G_{\mathbf{V}}/Q_{\mathbf{V}}$. Consider the following diagram
  \begin{center}
 	$\mathbf{E}_{\mathbf{T},\Omega} \times \mathbf{E}_{\mathbf{W},\Omega} \xleftarrow{\kappa_{\Omega} } F_{\Omega} \xrightarrow{\iota_{\Omega}} \mathbf{E}_{\mathbf{V},\Omega}$
 \end{center} 
where $\iota_{\Omega}$ is the natural embedding and $\kappa_{\Omega}(x)=(x',x'') \in \mathbf{E}_{\mathbf{T},\Omega} \times \mathbf{E}_{\mathbf{W},\Omega} $ is induced by $x \in F$. Notice that $\kappa_{\Omega}$ is a vector bundle.
 
Since for any $\underline{v}=(v^{1},v^{2},\cdots, v^{m})$ with each $v^{l} \in \mathbb{N}^{I}$, the flag variety $\tilde{\mathcal{F}}_{\underline{v},\Omega}$ is smooth and the map $\pi_{\underline{v},\Omega}$ is proper, we can see that $L_{\underline{v}}= (\pi_{\underline{v},\Omega})_{!} \bar{\mathbb{Q}}_{l}$ is a semisimple perverse sheaf (complex) on $\mathbf{E}_{\mathbf{V},\Omega}$ by \cite{MR751966}. Here  $\bar{\mathbb{Q}}_{l}$ is the constant sheaf on $\tilde{\mathcal{F}}_{\underline{v},\Omega}$. 

Let $\mathcal{D}^{b}_{G_{\mathbf{V}}}(\mathbf{E}_{\mathbf{V},\Omega})$ be the  $G_{\mathbf{V}}$-equivariant derived category of constructible sheaves on $E_{\mathbf{V},\Omega}$, and $\mathcal{Q}_{\mathbf{V},\Omega}$ be the subcategory consisting of direct sums of shifts of direct summands of such $L_{\underline{v}}$. We say $L$ is a Lusztig's sheaf if it is a simple object in $\mathcal{Q}_{\mathbf{V},\Omega}$.  \\

Let $\mathcal{K}_{\mathbf{V},\Omega}$ be the Grothendieck group of
$\mathcal{Q}_{\mathbf{V},\Omega}$ and $\mathcal{K}_{\Omega}=\bigoplus\limits_{\mathbf{V}} \mathcal{K}_{\mathbf{V},\Omega}$. More precisely, $\mathcal{K}_{\mathbf{V},\Omega}$ is a $\mathbb{Z}[v,v^{-1}]$-module spanned by $\{[L]| L \in \mathcal{Q}_{\mathbf{V},\Omega}\}$ modulo the following relations:
\begin{center}
	$[X \oplus Y]=[X]+[Y],$\\
	$[X[1]]=v[X].$ 
\end{center}

 Let $\mathcal{P}_{\mathbf{V},\Omega}$ be the set of simple perverse sheaves in $\mathcal{Q}_{\mathbf{V},\Omega}$ and denote the union of $\mathcal{P}_{\mathbf{V},\Omega}$ by $\mathcal{P}_{\Omega}$, then the images $\{[L]|L \in \mathcal{P}_{\Omega} \}$ of $\mathcal{P}_{\Omega}$ in $\mathcal{K}_{\Omega}$ indeed form a $\mathbb{Z}[v,v^{-1}]$-basis of  $\mathcal{K}_{\Omega}$, which is called the canonical basis. (See details in \cite{MR1088333}.) In particular, We denote the image in $\mathcal{K}_{\Omega}$ of constant sheaf $\bar{\mathbb{Q}}_{l}$ on $\mathbf{E}_{\mathbf{V},\Omega}$ with $|\mathbf{V}|=pi$ by $E_{i}^{(p)}$.

  The Verdier duality functor $\mathbf{D}:\mathcal{D}^{b}_{G_{\mathbf{V}}}(\mathbf{E}_{\mathbf{V},\Omega}) \rightarrow \mathcal{D}^{b}_{G_{\mathbf{V}}}(\mathbf{E}_{\mathbf{V},\Omega}) $ induces a $\mathbb{Z}$-linear bar involution $\bar{~}:\mathcal{K}_{\Omega} \rightarrow \mathcal{K}_{\Omega}$ such that $\overline{[L]}=[L]$ for any $L \in \mathcal{P}_{\Omega}$. (See details in \cite{MR1227098}.) \\

Lusztig has introduced the induction functor in \cite{MR1088333} 
\begin{center}
	$\mathbf{Ind}^{\mathbf{V}}_{\mathbf{V'},\mathbf{V''}}= (p_{3})_{!}(p_{2})_{b}(p_{1})^{\ast}: \mathcal{D}^{b}_{G_{\mathbf{V}'}}(\mathbf{E}_{\mathbf{V}',\Omega}) \times \mathcal{D}^{b}_{G_{\mathbf{V}''}}(\mathbf{E}_{\mathbf{V}'',\Omega}) \rightarrow \mathcal{D}^{b}_{G_{\mathbf{V}}}(\mathbf{E}_{\mathbf{V},\Omega})$ 
\end{center}
and the restriction functor
\begin{center}
	$\mathbf{Res}^{\mathbf{V}}_{\mathbf{T},\mathbf{W}}=(\kappa_{\Omega})_{!} (\iota_{\Omega})^{\ast}: \mathcal{D}^{b}_{G_{\mathbf{V}}}(\mathbf{E}_{\mathbf{V},\Omega}) \rightarrow \mathcal{D}^{b}_{G_{\mathbf{V}'} \times G_{\mathbf{V}''}}(\mathbf{E}_{\mathbf{V}',\Omega}\times \mathbf{E}_{\mathbf{V}'',\Omega}) $ 
\end{center} 
 Lusztig has also proved the following results (See details in \cite{MR1088333} Section 3 and Section 4.): 
\begin{center}
	$\mathbf{Ind}^{\mathbf{V}}_{\mathbf{V}',\mathbf{V}''}(L_{\underline{v'}} \boxtimes L_{\underline{v}''})= L_{\underline{v}' \underline{v}''},$\\
	$\mathbf{Res}^{\mathbf{V}}_{\mathbf{T},\mathbf{W}}L_{\underline{v}} =\bigoplus \limits_{\underline{\tau}+\underline{\omega}=\underline{v}} L_{\underline{\tau}} \boxtimes L_{\underline{\omega}}[-2M(\underline{\tau},\underline{\omega})],$
\end{center}
here $M(\underline{\tau},\underline{\omega})= \sum\limits_{h \in \Omega, l' < l} \tau^{l'}_{h'}\omega^{l}_{h''}+\sum\limits_{i \in I,l < l'} \tau_{i}^{l'}\omega_{i}^{l}$.\\
 
 Set $\mathbf{ind}^{\mathbf{V}}_{\mathbf{V}',\mathbf{V}''}=\mathbf{Ind}^{\mathbf{V}}_{\mathbf{V}',\mathbf{V}''}[d_{1}-d_{2}]$, then we can define a $\mathbb{Z}[v,v^{-1}]$-bilinear operator $\ast:\mathcal{K}_{\Omega}\otimes \mathcal{K}_{\Omega} \rightarrow \mathcal{K}_{\Omega}$ as the following: 	Given $A \in \mathcal{Q}_{\mathbf{V}',\Omega}, B \in \mathcal{Q}_{\mathbf{V}'',\Omega}$, we set
 \begin{center}
 	  $[A]\ast [B]= [\mathbf{ind}^{\mathbf{V}}_{\mathbf{V}',\mathbf{V}''}(A \boxtimes B)].$
 \end{center}
 Then $(\mathcal{K}_{\Omega},\ast)$ is an associative algebra.  Lusztig has proved that there is an isomorphism of algebras from $(\mathcal{K}_{\Omega},\ast)$ to $ {_{\mathbb{Z}}U}^{+}_{q}(\mathfrak{g})$, sending  $E_{i}^{(p)}$ to the $p$-th divided power $e_{i}^{(p)}$ of the Chevalley generator $e_{i}$ in ${_{\mathbb{Z}}U}^{+}_{q}(\mathfrak{g}) $. (See detials in \cite{MR1088333} Section 10.)
 
 Set $\mathbf{res}^{\mathbf{V}}_{\mathbf{T},\mathbf{W}}=\mathbf{Res}^{\mathbf{V}}_{\mathbf{T},\mathbf{W}}[d_{1}-d_{2}-2d_{3}]$, then we can define a $\mathbb{Z}[v,v^{-1}]$-linear operator $\bar{\Delta}:\mathcal{K}_{\Omega} \rightarrow \mathcal{K}_{\Omega} \otimes \mathcal{K}_{\Omega}$ as the following: Given $A \in \mathcal{Q}_{\mathbf{V},\Omega}$, we set
 \begin{center}
 	$\bar{\Delta}([A])= [\bigoplus\limits_{\mathbf{T},\mathbf{W}}\mathbf{res}^{\mathbf{V}}_{\mathbf{T},\mathbf{W}}(A)],$ 
 \end{center}
  then  $\bar{\Delta}$ gives a coalgebra structure of $\mathcal{K}_{\Omega}$.
  
  Similarly, we can also define  $\Delta:\mathcal{K}_{\Omega} \rightarrow \mathcal{K}_{\Omega} \otimes \mathcal{K}_{\Omega}$ as the following: Given $A \in \mathcal{Q}_{\mathbf{V},\Omega}$, we set
  \begin{center}
  	$\Delta([A])= [\bigoplus\limits_{\mathbf{T},\mathbf{W}}\mathbf{D}(\mathbf{res}^{\mathbf{V}}_{\mathbf{T},\mathbf{W}}(\mathbf{D}A))],$ 
  \end{center}
then $\bar{\Delta}=(\bar{~}\otimes\bar{~})\circ \Delta \circ \bar{~}$ and $\Delta$ also gives a coalgebra structure of $\mathcal{K}_{\Omega}$. $(\mathcal{K},\Delta)$ The coproduct $\Delta$ indeed realizes the coproduct of $_{\mathbb{Z}}U^{+}_{q}(\mathfrak{g})$. (See detials in \cite{MR1088333} Section 10.)

\subsection{Fourier-Deligne transformation}
Assume $k$ is an algebraic closure of a finite field $\mathbb{F}_{q}$ and fix a nontrivial character $\mathbb{F}_{q} \rightarrow \bar{\mathbb{Q}}_{l}^{\ast}$. This character defines an Artin-Schreier local system of rank $1$ on $k$. 

For two orientations $\Omega,\Omega'$, we define $T: \mathbf{E}_{\mathbf{V},\Omega \cup \Omega'} \rightarrow k$ by $T(x)=\sum \limits_{h \in \Omega \backslash \Omega'}tr(x_{h}x_{\bar{h}})$. Then the inverse image of the Artin-Schreier local system under $T$ is a well-defined $G_{\mathbf{V}}$-equivariant local system of rank $1$ on $\mathbf{E}_{\mathbf{V},\Omega \cup \Omega'}$ and we denote it by $\mathcal{L}_{T}$.

Consider the following diagram:
\begin{center}
	$\mathbf{E}_{\mathbf{V},\Omega} \xleftarrow{\delta} \mathbf{E}_{\mathbf{V},\Omega \cup \Omega'} \xrightarrow{\delta'} \mathbf{E}_{\mathbf{V},\Omega'}$
\end{center}
here $\delta,\delta'$ are the forgetting maps defined by:
\begin{center}
	$\delta((x_{h})_{h \in \Omega \cup \Omega'} )= ((x_{h})_{h \in \Omega}),$\\
	$\delta'((x_{h})_{h \in \Omega \cup \Omega'} )= ((x_{h})_{h \in \Omega'}).$
\end{center}
Then following \cite{MR1088333} we can define a functor
\begin{center}
	$\mathcal{F}_{\Omega,\Omega'}:\mathcal{D}^{b}_{G_{\mathbf{V}}}(E_{\mathbf{V},\Omega}) \rightarrow \mathcal{D}^{b}_{G_{\mathbf{V}}}(E_{\mathbf{V},\Omega'}), \mathcal{F}_{\Omega,\Omega'}(L)=\delta'_{!}(\delta^{\ast}(L)\otimes \mathcal{L}_{T})[D],$
\end{center}  
here $D=\sum\limits_{h \in \Omega \backslash \Omega'}\dim \mathbf{V}_{h'}\dim\mathbf{V}_{h''}$.

The functor $\mathcal{F}_{\Omega,\Omega'}$ is called the Fourier-Deligne transformation for quivers by Lusztig. Moreover, he has proved the following proposition:
\begin{proposition} [ \cite{MR1088333} Theorem 5.4]
	With the notations above, we  have:
\begin{center}
	$\mathcal{F}_{\Omega,\Omega'}(\mathbf{Ind}^{\mathbf{V}} _{\mathbf{V}',\mathbf{V}''}L_{1} \boxtimes L_{2}) \cong  \mathbf{Ind}^{\mathbf{V}} _{\mathbf{V}',\mathbf{V}''}(\mathcal{F}_{\Omega,\Omega'}(L_{1}) \boxtimes \mathcal{F}_{\Omega,\Omega'}(L_{2}))[C],$
\end{center}
where $C=\sum\limits_{h \in \Omega \backslash \Omega'} (\upsilon''_{h'} \upsilon'_{h''} -\upsilon'_{h'}\upsilon''_{h''})$.
\end{proposition}

With the proposition above, we can easily get the following two results:

\begin{corollary}[\cite{MR1088333}]
	The functor $\mathcal{F}_{\Omega,\Omega'}$ induces an isomorphism between algebras $\mathcal{K}_{\Omega}$ and $\mathcal{K}_{\Omega'}$. 
\end{corollary}

\begin{corollary}[\cite{MR1088333}]
	The functor $\mathcal{F}_{\Omega,\Omega'}$ indeuces a bijection $\mathcal{P}_{\Omega}$ and $\mathcal{P}_{\Omega'}$ and defiens a equivalence of categories $\mathcal{Q}_{\mathbf{V},\Omega} \cong \mathcal{Q}_{\mathbf{V},\Omega'}$. Moreover, if we denote the bijection between the sets $\mathcal{P}_{\Omega}$ and $\mathcal{P}_{\Omega'}$ by $\eta_{\Omega,\Omega'}$, then $\eta_{\Omega',\Omega''}\eta_{\Omega,\Omega'}=\eta_{\Omega,\Omega''}$. 
\end{corollary}

With the two corollaries above, we can denote $\mathcal{K}_{\Omega}$ by $\mathcal{K}$ and $\mathcal{P}_{\mathbf{V},\Omega}$ by $\mathcal{P}_{\mathbf{V}}$ respectively if there is no ambiguity.

\subsection{Bilinear pairing}
Following \cite{MR1247495} and \cite{MR1227098}, we choose  a smooth irreducible variety $\Gamma$ with free $G_{\mathbf{V}}$-action such that the $\bar{\mathbb{Q}}_{l}$-cohomology of $\Gamma$ vanishes in degree $1,2 \cdots m$ for a large enough $m$, then $G_{\mathbf{V}}$  acts freely on $\Gamma \times \mathbf{E}_{\mathbf{V},\Omega}$.  Let us consider the diagram 
\begin{center}
	$\mathbf{E}_{\mathbf{V},\Omega} \xleftarrow{s} \Gamma \times \mathbf{E}_{\mathbf{V},\Omega} \xrightarrow{t} G_{\mathbf{V}} \backslash \Gamma \times \mathbf{E}_{\mathbf{V},\Omega}$
\end{center}
here $G_{\mathbf{V}} \backslash \Gamma \times \mathbf{E}_{\mathbf{V},\Omega}$ is the quotient space.

Let $u: G_{\mathbf{V}} \backslash \Gamma \times \mathbf{E}_{\mathbf{V},\Omega} \rightarrow \{ pt\}$ and $d= dim (G_{\mathbf{V}} \backslash \Gamma) $, then
for any $A,B \in \mathcal{P}_{\mathbf{V},\Omega}$ and $j \in \mathbb{Z}$, we denote the $\bar{\mathbb{Q}}_{l}$-vector space $\mathcal{H}^{j+2d}(u_{!}(t_{\flat}s^{\ast}A \otimes t_{\flat}s^{\ast}B)   )$ by $\mathbf{D}_{j}(\mathbf{V},A,B)$. This vector space does not depend on the choice of $\Gamma$ and $m$. (See details in \cite{MR972345} or \cite{MR1247495}.).

Lusztig introduced a symmetric $\mathbb{Z}[v,v^{-1}]$-bilinear pairing $(-,-): \mathcal{K}_{\mathbf{V},\Omega} \times \mathcal{K}_{\mathbf{V},\Omega} \rightarrow \mathbb{Z}((v))$ by:
\begin{align*}
	([L],[K])=\sum\limits_{j} \dim \mathbf{D}_{j}(\mathbf{V},L,K)v^{j}
\end{align*}
Here $L,K \in \mathcal{P}_{\mathbf{V},\Omega}$ are simple perverse sheaves. In particular, we have $(E_{i},E_{i})=(1-v^{-2})^{-1}$.

Set $(\mathcal{K}_{\mathbf{V},\Omega},\mathcal{K}_{\mathbf{V}',\Omega})=0$ for $|\mathbf{V}|\neq |\mathbf{V}'|$, then  $(-,-): \mathcal{K}_{\Omega} \times \mathcal{K}_{\Omega} \rightarrow \mathbb{Z}((v))$ is a well-defined symmetric $\mathbb{Z}[v,v^{-1}]$-bilinear pairing. Moreover, the pairing $(-,-)$ satisfies the following property: (See details in \cite{MR1227098} Chapter 10 and Chapter 13.)

(1)$(-,-)$ doesn't depned on the choice of $\Omega$;

(2) $([L],[L']) \in v^{-1} \mathbb{Z}[[v^{-1}]] \cap \mathbb{Q}(v)$ if $[L] \neq [L']$ for $L,L' \in \mathcal{P}_{\mathbf{V},\Omega}$;

(3) $([L],[L]) \in 1+ v^{-1}\mathbb{Z}[[v^{-1}]]\cap \mathbb{Q}(v)$ for $L \in \mathcal{P}_{\mathbf{V},\Omega}$.

\subsection{Analysis at sink}
Fix $i \in I$ and an orientation $\Omega$ such that $i$ is a sink, we define $\mathbf{E}_{\mathbf{V},i,p}$ to be the subset of $\mathbf{E}_{\mathbf{V},\Omega}$ consisting of $x$ such that ${\rm{codim}}_{\mathbf{V}_{i}} ( {\rm{Im}} \sum\limits_{h \in \Omega, h''=i} x_{h}) =p$. Then $\mathbf{E}_{\mathbf{V}}$ has a partition $\mathbf{E}_{\mathbf{V},\Omega}= \bigcup \limits_{p} \mathbf{E}_{\mathbf{V},i,p}$. and the union $\mathbf{E}_{\mathbf{V},i, \geq p}= \bigcup\limits_{p' \geq p} \mathbf{E}_{\mathbf{V},i, p'}$ is a closed subset of $\mathbf{E}_{\mathbf{V},\Omega}$.

Given $L \in \mathcal{P}_{\mathbf{V},\Omega}$, there exists a unique integer $t$ such that $supp(L) \subseteq \mathbf{E}_{\mathbf{V},i, \geq t}$ but $supp(L) \nsubseteq \mathbf{E}_{\mathbf{V},i, \geq t+1}$ and we write $t_{i}(L)=t$. Notice that $t_{i}(L) \leq \upsilon_{i}$.
 
 Then we have the following lemma proved by Lusztig, which is called the key lemma in \cite{MR1227098}:
 \begin{lemma} [\cite{MR1088333} Lemma 6.4]
 	With the notation above, fix $0 \leq t \leq \upsilon_{i}$. Assume $|\mathbf{T}|=|\mathbf{V}'|=ti$ and let $d=t(\upsilon_{i}-t)$.
 	
 	(1) Let $L\in\mathcal{P}_{\mathbf{V},\Omega}$ be such that $t_{i}(L)=t$, then $\mathbf{Res}^{\mathbf{V}}_{\mathbf{T},\mathbf{W}}L \in \mathcal{Q}_{\mathbf{W},\Omega}$ is a direct sum of finitely many summands of the form $K'[f']$ for various $K' \in \mathcal{P}_{\mathbf{W},\Omega}$ and $f' \in \mathbb{Z}$. Moreover, exactly one of these summands, denoted by $K[f]$, satisfies $t_{i}(K)=0$ and $f=d$ and the others satisfy $t_{i}(K')> 0$.
 	
 	(2) Let $K \in \mathcal{P}_{\mathbf{V}'',\Omega}$ be such that $t_{i}(K)=0$, then $\mathbf{Ind}^{\mathbf{V}}_{\mathbf{V}',\mathbf{V}''}(\bar{\mathbb{Q}}_{l} \boxtimes K)$ is a direct sum of finitely many summands of the form $L'[g']$ for variuos $L' \in \mathcal{P}_{\mathbf{V},\Omega}$ and $g' \in \mathbb{Z}$. Moreover, exactly one of these summands, denoted by $L[g]$, satisfies $t_{i}(L)=t$ and $g=d$ and the others satisfy $t_{i}(L')> t$.
 	
 	(3) There is a bijectiion $$\pi_{i,t}:\{K \in \mathcal{P}_{\mathbf{V}'',\Omega}|t_{i}(K)=0 \} \rightarrow \{L \in \mathcal{P}_{\mathbf{V},\Omega}|t_{i}(L)=t \}$$   induced by the decompositions of the direct sums above.
 \end{lemma}

If $|\mathbf{V}'|=ri$, we denote $\mathbf{V'}$ by $\mathbf{V}'_{r}$ and $\mathbf{V}''$ by $\mathbf{V}''_{r}$. For an orientation $\Omega'$ and $L \in \mathcal{P}_{\mathbf{V},\Omega'}$, we define $s_{i}(L)$ to be the largest integer $r$ satisfying that there exists $L' \in \mathcal{P}_{\mathbf{V}''_{r},\Omega'}$ such that $L$ is isomorphic to a shift of a direct summand of $\mathbf{Ind}^{\mathbf{V}}_{\mathbf{V}'_{r},\mathbf{V}''_{r}}(\bar{\mathbb{Q}}_{l} \boxtimes L')$. Notice that the definition of $s_{i}(L)$ does not depend on the choice of $\Omega'$ by Proposition 3.1.

\begin{proposition} [\cite{MR1088333} Proposition 6.6]
	With the notations above, we have:
	
	(1)  There exist $L'_{r'} \in \mathcal{P}_{\mathbf{V}''_{r'},\Omega'}$ for $r' > s_{i}(L)$ and $L''_{r'} \in  \mathcal{P}_{\mathbf{V}''_{r'},\Omega'}$ for $r' \geq s_{i}(L)$ such that 
	\begin{center}
		$L \oplus \bigoplus \limits_{r' >s_{i}(L)}\mathbf{Ind}^{\mathbf{V}}_{\mathbf{V}'_{r'},\mathbf{V}''_{r'}}(\bar{\mathbb{Q}}_{l} \boxtimes L'_{r'}) \cong \bigoplus \limits_{r' \geq s_{i}(L)}\mathbf{Ind}^{\mathbf{V}}_{\mathbf{V}'_{r'},\mathbf{V}''_{r'}}(\bar{\mathbb{Q}}_{l} \boxtimes L''_{r'}) .$
	\end{center}

   (2) $s_{i}(L)=t_{i}(L)$ if $i$ is a sink in $\Omega'$.

\end{proposition}
We  sketch a proof which is given by Lusztig  in \cite{MR1088333} as follows:
\begin{proof}
	 Without loss of generality, we assume that $i$ is a sink in $\Omega'$. By Lemma 3.4, we can see that $s_{i}(L) \geq t_{i}(L)$. Assume that $L' \in \mathcal{P}_{\mathbf{V}''_{r},\Omega'}$ is a shift of a direct summand of $L_{\underline{v}}$ such that $L$ is isomorphic to a shift of a direct summand of $\mathbf{Ind}^{\mathbf{V}}_{\mathbf{V}'_{r},\mathbf{V}''_{r}}(\bar{\mathbb{Q}}_{l} \boxtimes L')$, then $L$ is a shift of a direct summand of $L_{\underline{v}'}$ for $\underline{v}'=(ri,\underline{v})$. Notice that $supp(L_{\underline{v}'}) \subset E_{\mathbf{V},\Omega',\geq r}$, we can see that $t_{i}(L) \geq r$. Hence $t_{i}(L) \geq s_{i}(L)$. (2) is proved. By desending induction on $s_{i}(L)=t_{i}(L)$, (1) can be proved easily.
\end{proof}

\subsection{Analysis at source}
 
 Fix $i \in I$ and an orientation $\Omega$ such that $i$ is a source, we define $\mathbf{E}_{\mathbf{V},i}^{p}$ to be the subset of $\mathbf{E}_{\mathbf{V},\Omega}$ consisting of $x$ such that ${\rm{dim}}( {\rm{Ker}} \bigoplus\limits_{h \in \Omega, h'=i} x_{h}) =p$. Then $\mathbf{E}_{\mathbf{V}}$ has a partition $\mathbf{E}_{\mathbf{V},\Omega}= \bigcup \limits_{p} \mathbf{E}_{\mathbf{V},i}^{p}$. and the union $\mathbf{E}_{\mathbf{V},i}^{\geq p}= \bigcup\limits_{p' \geq p} \mathbf{E}_{\mathbf{V},i}^{p'}$ is a closed subset.

 Given $L \in \mathcal{P}_{\mathbf{V},\Omega}$, there exists a unique integer $t$ such that $supp(L) \subseteq \mathbf{E}_{\mathbf{V},i}^{ \geq t}$ but $supp(L) \nsubseteq \mathbf{E}_{\mathbf{V},i}^{\geq t+1}$ and we write $t_{i}^{\ast}(L)=t$. Notice that $t_{i}^{\ast}(L) \leq \upsilon_{i}$.
 
 By a similar arguments, we can prove the following result, which is dual to Lemma 3.4:
 \begin{lemma}
 	With the notation above, fix $0 \leq t \leq \upsilon_{i}$. Assume $|\mathbf{W}|=|\mathbf{V}''|=ti$ and let $d=t(\upsilon_{i}-t)$.
 	
 	(1) Let $L\in\mathcal{P}_{\mathbf{V},\Omega}$ be such that $t_{i}^{\ast}(L)=t$, then $\mathbf{Res}^{\mathbf{V}}_{\mathbf{T},\mathbf{W}}L \in \mathcal{Q}_{\mathbf{T},\Omega}$ is a direct sum of finitely many summands of the form $K'[f']$ for various $K' \in \mathcal{P}_{\mathbf{T},\Omega}$ and $f' \in \mathbb{Z}$. Moreover, exactly one of these summands, denoted by $K[f]$, satisfies $t_{i}^{\ast}(K)=0$ and $f=d$ and the others satisfy $t_{i}^{\ast}(K')> 0$.
 	
 	(2) Let $K \in \mathcal{P}_{\mathbf{V}',\Omega}$ be such that $t_{i}^{\ast}(K)=0$, then $\mathbf{Ind}^{\mathbf{V}}_{\mathbf{V}',\mathbf{V}''}(K \boxtimes \bar{\mathbb{Q}}_{l} )$ is a direct sum of finitely many summands of the form $L'[g']$ for variuos $L' \in \mathcal{P}_{\mathbf{V},\Omega}$ and $g' \in \mathbb{Z}$. Moreover, exactly one of these summands, denoted by $L[g]$, satisfies $t_{i}^{\ast}(L)=t$ and $g=d$ and the others satisfy $t_{i}^{\ast}(L')> t$.
 	
 	(3) There is a bijectiion $$\pi^{\ast}_{i,t}:\{K \in \mathcal{P}_{\mathbf{V}',\Omega}|t_{i}^{\ast}(K)=0 \} \rightarrow  \{L \in \mathcal{P}_{\mathbf{V},\Omega}|t_{i}^{\ast}(L)=t \} $$  induced by the decompositions of the direct sums above.
 \end{lemma}

  In this section, if $|\mathbf{V}''|=ri$, we denote $\mathbf{V}'$ by $\mathbf{V}'_{r}$ and $\mathbf{V}''$ by $\mathbf{V}''_{r}$. For an orientation $\Omega'$ and $L \in \mathcal{P}_{\mathbf{V},\Omega' }$, we define $s_{i}^{\ast}(L)$ to be the largest integer $r$ satisfying that there exists $L' \in \mathcal{P}_{\mathbf{V}'_{r},\Omega'}$ such that $L$ is isomorphic to a direct summand of $\mathbf{Ind}^{\mathbf{V}}_{\mathbf{V}'_{r},\mathbf{V}''_{r}}(L' \boxtimes \bar{\mathbb{Q}}_{l})$.

 With the notations above, we have the following proposition, which is dual to Proposition 3.5:
 \begin{proposition}
 	
 	(1)  There exist $L'_{r'} \in \mathcal{P}_{\mathbf{V''}_{r'},\Omega'}$ for $r' > s_{i}^{\ast}(L)$ and $L''_{r'} \in  \mathcal{P}_{\mathbf{V''}_{r'},\Omega'}$ for $r' \geq s_{i}^{\ast}(L)$ such that 
 	\begin{center}
 		$L \oplus \bigoplus \limits_{r' >s_{i}^{\ast}(L)}\mathbf{Ind}^{\mathbf{V}}_{\mathbf{V}'_{r'},\mathbf{V}''_{r'}}(L'_{r'} \boxtimes \bar{\mathbb{Q}}_{l}) \cong \bigoplus \limits_{r' \geq s_{i}^{\ast}(L)}\mathbf{Ind}^{\mathbf{V}}_{\mathbf{V}'_{r'},\mathbf{V}''_{r'}}(L''_{r'} \boxtimes \bar{\mathbb{Q}}_{l}). $
 	\end{center}
 	
 	(2) $s_{i}^{\ast}(L)=t_{i}^{\ast}(L)$ if $i$ is a source in $\Omega'$.
 	
 \end{proposition}

\section{An $I \times \mathbb{Z}_{2}$-colored graph $\mathcal{G}_{1}$ of $\mathcal{K}$}
\subsection{Kashiwara's operator in rank one case}
We fix an $i \in I$. For any $\mathbf{V}$, we choose a  decompositions $\mathbf{V}=\mathbf{T}\oplus {\mathbf{V}''}$ such that $|\mathbf{T}|=i$. Then $\mathbf{E}_{\mathbf{T},\Omega} \times \mathbf{E}_{\mathbf{V}'',\Omega} \cong \mathbf{E}_{\mathbf{V}'',\Omega}$, hence $\mathbf{res}^{\mathbf{V}}_{\mathbf{T},{\mathbf{V}''}}: \mathcal{D}^{b}_{G_{\mathbf{V}}}(\mathbf{E}_{\mathbf{V},\Omega}) \rightarrow \mathcal{D}^{b}_{G_{ \mathbf{V}''}}(\mathbf{E}_{\mathbf{V}'',\Omega})$. Similarly, we choose $\mathbf{V}=\mathbf{V}'\oplus \mathbf{W}$ such that $|\mathbf{W}|=i$, then we have $\mathbf{res}^{\mathbf{V}}_{\mathbf{V}',\mathbf{W}}: \mathcal{D}^{b}_{G_{\mathbf{V}}}(\mathbf{E}_{\mathbf{V},\Omega}) \rightarrow \mathcal{D}^{b}_{G_{ \mathbf{V}'}}(\mathbf{E}_{\mathbf{V}',\Omega})$

\begin{definition}
	We define $_{i} \mathcal{R}$ to be the functor $\bigoplus\limits_{\mathbf{V}} \mathbf{D}\circ \mathbf{res}^{\mathbf{V}}_{\mathbf{T},\mathbf{V}''} \circ \mathbf{D} $ and $\mathcal{R}_{i}$ to be the functor $\bigoplus\limits_{\mathbf{V}} \mathbf{D}\circ \mathbf{res}^{\mathbf{V}}_{\mathbf{V}',\mathbf{W}} \circ \mathbf{D} $, then $_{i} \mathcal{R}$ and $\mathcal{R}_{i}$ induces  operators $_{i}r$ and  $r_{i}: \mathcal{K}\rightarrow \mathcal{K}$ defined by:
	\begin{center}
		$_{i}r([A])=[_{i} \mathcal{R}(A) ],$\\
		$r_{i}([A])=[ \mathcal{R}_{i}(A) ]$
	\end{center}
for any $A \in \mathcal{Q}_{\mathbf{V},\Omega}$.

For $n \in \mathbb{N}$, we define operators $\theta_{i}^{(n)}$ and $ \theta_{i}^{\ast(n)}:  \mathcal{K}\rightarrow \mathcal{K}$ by setting:
\begin{center}
	$\theta_{i}^{(n)}([A])=E_{i}^{(n)}  \ast [A],$\\
	$\theta_{i}^{\ast (n)}([A])=[A] \ast E_{i}^{(n)} $ 
\end{center}
for any $A \in \mathcal{Q}_{\mathbf{V},\Omega}$. In particular, we denote $\theta_{i}^{(n)}$ and $ \theta_{i}^{\ast(n)}$ by $\theta_{i}$ and $ \theta_{i}^{\ast}$ respectively if $n=1$.
\end{definition}

 The following relation is easy to know and it implies that the operators $_{i}r,r_{i}$ coincide with the derivative operators of $U^{+}_{q}(\mathfrak{g})$(See details in \cite{zhao2022derivation}.):
\begin{lemma}
	The operators satisfies the following relation:
	\begin{center}
		$_{i}r \theta_{i}^{(n)}=v^{2n}\theta_{i}^{(n)}{_{i}r}+v^{n-1}\theta_{i}^{(n-1)}$ for $n \geq 1$.
	\end{center}
\end{lemma}

Notice that for $N> \dim \mathbf{V}_{i}$, $(_{i}r)^{N}([A])=0$ for any $A \in \mathcal{Q}_{\mathbf{V},\Omega}$, hence the operator $\Pi_{t} =\sum\limits_{s \geq 0} (-1)^{s}v^{(s-1)s/2} \theta_{i}^{(s)}(_{i}r)^{(s+t)}: \mathcal{K} \rightarrow \mathcal{K} $ is well-defined and satisfies the following properties:

\begin{lemma}[\cite{MR1227098} Lemma 16.1.2]
	(1) We have $_{i}r \Pi_{t}=0$ for any $t \geq 0$;\\
	(2) Notice that for any $x \in \mathcal{K}$,  $\sum\limits_{t \geq 0} v^{-(t-1)t/2} \theta_{i}^{(t)} \Pi_{t}(x)$ is a finite sum. We have $\sum\limits_{t \geq 0} v^{-(t-1)t/2} \theta_{i}^{(t)} \Pi_{t}=\mathbf{Id}_{\mathcal{K}}$;\\
	(3) Given $x \in \mathcal{K}$, there is a unique collection of $\{x_{N}\in \mathcal{K}| N \geq 0\}$ such that  $x=\sum \limits_{N \geq 0} \theta_{i}^{(N)}x_{N}$ and  $_{i}r(x_{N})=0$. Indeed, $x_{N}=v^{-(N-1)N/2}\Pi_{N}(x)$. More precisely, if we denote $\Ker {_{i}r}$ by $\mathcal{K}(0)$ and set $\mathcal{K}(n)= \theta_{i}^{(n)} \mathcal{K}(0)$, we have $\mathcal{K}=\bigoplus \limits_{n \geq 0}\mathcal{K}(n)$.
\end{lemma}

With the lemma above, we can introduce the Kashiwara's operators of $\mathcal{K}$ following \cite{MR1227098}.
 \begin{definition}
 	  We define linear operators $\tilde{\phi}_{i}, \tilde{\epsilon}_{i}: \mathcal{K} \rightarrow \mathcal{K}$ as the following:
 	 \begin{align*}
 	 	\tilde{\phi}_{i}(\theta_{i}^{(n)}y)=&\theta_{i}^{(n+1)}y,\\
 	 \tilde{\epsilon}_{i}(\theta_{i}^{(n)}y)=&\theta_{i}^{(n-1)}y~ if~ n>0,\\
 	 	\tilde{\epsilon}_{i}(\theta_{i}^{(n)}y)=&0~ if~ n=0,
 	 \end{align*}
  for $n \in \mathbb{N}$ and $y \in \mathcal{K}(0)$.
 \end{definition}
\begin{remark}
	If we identify $\mathcal{K}$ with $_{\mathbb{Z}}U^{+}_{q}(\mathfrak{g})$, then the operators $\tilde{\phi}_{i}, \tilde{\epsilon}_{i}$ coincide with the Kashiwara's operators $\tilde{f}_{i},\tilde{e}_{i}$ on $_{\mathbb{Z}}U^{+}_{q}(\mathfrak{g})$ in \cite{MR1115118}.
\end{remark}

Assume $i$ is a sink in $\Omega$, let $\mathcal{B}=\{[L]| L \in \mathcal{P}_{\Omega}\}$ and $\mathcal{B}_{i,n}= \{[L] \in \mathcal{P}_{\Omega}| t_{i}(L)=n \}, n \geq 0, i \in I$, then $\pi_{i,n}$ induces a bijection between the sets $\mathcal{B}_{i,0}$ and $\mathcal{B}_{i,n}$. Moreover, by Lemma 3.4 and Proposition 3.5 we can see that the basis $\mathcal{B}$ is apated in the sense of Lusztig in \cite{MR1227098} Section 16.3.1. More precisely, we have:

\begin{proposition} [\cite{MR1227098} Section 17.3.2]  The following results hold:\\ 
	(1) The set $\bigcup \limits_{N \geq n}\mathcal{B}_{i,n} $ form a basis of $\theta_{i}^{(n)}\mathcal{K}$;\\
	(2) For $[L_{0}] \in \mathcal{B}_{i,0}$, we have $\theta_{i}^{(n)}[L_{0}]-\pi_{i,n}([L_{0}]) \in \theta_{i}^{(n+1)}\mathcal{K}$.
\end{proposition}

Let $\mathcal{L}=\{x \in \mathcal{K}|(x,x) \in \mathbb{Z}[[v^{-1}]] \cap \mathbb{Q}(v)\}$ and set $T_{i,0}=\{x \in \mathcal{K}(0) |(x,x)\in 1 +v^{-1}\mathbb{Z}[v,v^{-1}]  \}$, $T_{i,n}=\theta_{i}^{(n)}T_{i,0}$ for $n>0$. Let $\mathcal{B}_{i}(n)= \mathcal{B} \cap (T_{i,n}+v^{-1} \mathcal{L})$, then $\mathcal{B}= \bigcup\limits_{N \geq 0} \mathcal{B}_{i}(N)$. The following proposition shows that the operators $\tilde{\phi}_{i}, \tilde{\epsilon}_{i}$ can be described by the bijections $\pi_{i,n}$ modulo $v^{-1}\mathcal{L}$:

\begin{proposition}[\cite{MR1227098} Proposition16.2.5, Proposition 16.2.8 and Proposition 16.3.5]
	We have the following results:\\
	(1) $\mathcal{L}$ is a $\mathbb{Z}[v^{-1}]$-module with the basis $\mathcal{B}$.\\ 
	(2) $\mathcal{L}$ is preserved by $\tilde{\phi}_{i}, \tilde{\epsilon}_{i}$.\\
	(3) $\mathcal{B}_{i,n}=\mathcal{B}_{i}(n)$.\\
	(4) Let $[L_{0}] \in \mathcal{B}_{i,0}$ and $[L]=\pi_{i,n}([L_{0}])$, then $$\tilde{\phi}_{i}([L])= \pi_{i,n+1}([L_{0}])~\mod~ v^{-1} \mathcal{L}$$ and  
	\begin{center}
		$\tilde{\epsilon}_{i}([L])= \left\{
		\begin{aligned}
			&	\pi_{i,n-1}([L_{0}])~ \mod ~ v^{-1} \mathcal{L}  & n>0 \\
			&	0~ \mod ~ v^{-1} \mathcal{L} & n=0
		\end{aligned}
		\right. $
	\end{center} 
\end{proposition}

 If we replace $\theta_{i}^{(n)}$ and $_{i}r$ by $\theta_{i}^{\ast(n)}$ and $r_{i}$ respectively and let $\mathcal{K}(0)^{\ast}= \Ker r_{i} ,\mathcal{K}(n)^{\ast}=\theta_{i}^{\ast(n)}\mathcal{K}(0)^{\ast}$, we can introduce linear operators $\tilde{\phi}_{i}^{\ast}$ and $\tilde{\epsilon}_{i}^{\ast}$ in a similar way. 
 
 Assume $i$ is a source in $\Omega$, let  $\mathcal{B}^{\ast}_{i,n}=\{[L]|t_{i}^{\ast}(L)=n\}$ for $n \in \mathbb{N}$, then $\pi_{i,n}^{\ast}$ induces a bijection between the sets $\mathcal{B}_{i,0}^{\ast}$ and $\mathcal{B}_{i,n}^{\ast}$. The following result is dual to Proposition 4.6:
 \begin{proposition} We have:\\
 	(1) The set $\bigcup \limits_{N \geq n}\mathcal{B}^{\ast}_{i,n} $ form a basis of $\theta_{i}^{\ast(n)}\mathcal{K}$;\\
 	(2) For $[L_{0}] \in \mathcal{B}^{\ast}_{i,0}$, we have $\theta_{i}^{\ast(n)}[L_{0}]-\pi_{i,n}^{\ast}([L_{0}]) \in \theta_{i}^{\ast(n+1)}\mathcal{K}$.
 \end{proposition}

 Define $T_{i,0}^{\ast}=\{x \in \mathcal{K}(0)^{\ast} |(x,x)\in 1 +v^{-1}\mathbb{Z}[v,v^{-1}]  \}$ and $T_{i,n}^{
 \ast}=\theta_{i}^{\ast(n)}T^{\ast}_{i,0}$ for $n>0$. Let $\mathcal{B}_{i}^{\ast}(n)= \mathcal{B} \cap (T^{\ast}_{i,n}+v^{-1} \mathcal{L})$, then $\mathcal{B}= \bigcup\limits_{N \geq 0} \mathcal{B}^{\ast}_{i}(N)$. The following result is dual to Proposition 4.7:
\begin{proposition} We have:\\
	(1) For any $n\in \mathbb{N}$, $\mathcal{B}^{\ast}_{i,n}=\mathcal{B}^{\ast}_{i}(n)$.\\
	(2) Let $[L_{0}] \in \mathcal{B}^{\ast}_{i,0}$ and $[L]=\pi^{\ast}_{i,n}([L_{0}])$, then $$\tilde{\phi}^{\ast}_{i}([L])= \pi_{i,n+1}^{\ast}([L_{0}])~ \mod~ v^{-1} \mathcal{L}$$ and 
		\begin{center}
		$\tilde{\epsilon}^{\ast}_{i}([L])= \left\{
		\begin{aligned}
			&	\pi^{\ast}_{i,n-1}([L_{0}])~ \mod ~ v^{-1} \mathcal{L}  & n>0 \\
			&	0~ \mod ~ v^{-1} \mathcal{L} & n=0
		\end{aligned}
		\right. $
	\end{center} 
\end{proposition}

\subsection{The graph $\mathcal{G}_{1}$}

For each $i \in I$, we fix orientations $\Omega_{i}$ and $\Omega^{i}$ such that $i$ is a sink in $\Omega_{i}$ but a source in $\Omega^{i}$.  \\

For a given $L \in \mathcal{P}_{\Omega}$  and $i \in I$, we consider $L'=\mathcal{F}_{\Omega,\Omega_{i}}(L)$. We assume $p= t_{i}(L')$, then there exists a unique $L''$ such that $\pi_{i,p}(L'')=L'$. Let $K'=\pi_{i,p+1}(L'')$, then $K= \mathcal{F}_{\Omega_{i},\Omega}(K') \in \mathcal{P}_{\Omega}$ is a well-defined simple perverse sheaf, which does not depend on the choice of $\Omega_{i}$. We associate an arrow $[L] \xrightarrow{i_{+}} [K]$ and  adapt the notation $[K]=i_{+}([L])$ or $K=i_{+}(L)$. \\

Dually, if we consider $L'=\mathcal{F}_{\Omega,\Omega^{i}}(L)$ and assume $p= t_{i}^{\ast}(L')$, let $K'= \pi _{i,p+1}^{\ast} (\pi_{i,p}^{\ast} )^{-1}(L')$ and $K= \mathcal{F}_{\Omega^{i},\Omega}(K') \in \mathcal{P}_{\Omega}$. We associate an arrow $[L] \xrightarrow{i^{+}} [K]$ and write $[K]=i^{+}([L])$ or $K=i^{+}(L)$. 

\begin{definition}
	We define an $I \times  \mathbb{Z}_{2}$-colored graph $\mathcal{G}_{1}=(\mathcal{V}_{1},\mathcal{E}_{1})$ as follows: \\
	
	(1) The set of vertices is $\mathcal{V}_{1}= \{[L]|L \in \mathcal{P}_{\Omega} \}$. \\
	
	(2) The set of arrows is $\mathcal{E}_{1}=\{ [L] \xrightarrow{i^{+}} [K], [L] \xrightarrow{i_{+}} [K]| i \in I, K,L \in \mathcal{P}_{\Omega} \} $.\\
		 
	 We also define $i_{-}$ and $i^{-}$ to be the inverse of $i_{+}$ and $i^{+}$ respectively.
\end{definition}

 By Corollary 3.3, we can see that $\mathcal{G}_{1}=(\mathcal{V}_{1},\mathcal{E}_{1})$ is a well-defined $I \times  \mathbb{Z}_{2}$-colored graph, which is independent of the choice of $\Omega$.

 \subsection{The commutative relation between $i^{+}$ and $j_{+}$}
 
 In this and the next sections, we fix $i\neq j \in I$ and choose orientations $\Omega_{i},\Omega^{i}$ and $\Omega_{j}$ such that $i$ is a sink in $\Omega_{i}$ but a source in $\Omega^{i}$ and $j$ is a sink in $\Omega_{j}$. Given $K \in \mathcal{P}_{\mathbf{V},\Omega}$, we shorten  $t_{i} (\mathcal{F}_{\Omega,\Omega_{i}}(K) ), t_{i}^{\ast} (\mathcal{F}_{\Omega,\Omega^{i}}(K) )$ and $t_{j} (\mathcal{F}_{\Omega,\Omega_{j}}(K) )$ by $t_{i}(K),t_{i}^{\ast}(K)$ and $t_{j}(K)$ in this and the next  sections respectively. 
 
 \begin{lemma}
 	Assume $[L_{0}] \in \mathcal{B}_{i,0}^{\ast}$ and $[L]=\pi_{i,c}^{\ast}([L_{0}]) $ for some $c \geq 0$, then we have $t_{j}(L_{0})=t_{j}(L)$.
 \end{lemma}
 \begin{proof}
 	Assume $t_{j}(L_{0})=d$. By Proposition 4.7, we have $\mathcal{B}_{j,d}=\mathcal{B}_{j}(d)$ and there exists $x' \in T_{j,0}$ and $z' \in \mathcal{L}$ such that $[L_{0}]=\theta_{j}^{(d)}x'+v^{-1}z' $. Similarly, there exists $x'' \in T_{i,0}^{\ast}$ and $z'' \in \mathcal{L}$ such that $[L_{0}]=x''+v^{-1}z'' $. Then we can see that $x''- E_{j}^{(d)} \ast x' \in v^{-1} \mathcal{L}$. Notice that $[L]= (\tilde{\phi}_{i}^{\ast})^{c}[L_{0}]~ \mod~ v^{-1} \mathcal{L}$, we can see that $[L]=\theta_{i}^{\ast(c)}x''+v^{-1}z$ for some $z \in \mathcal{L}$. Hence $[L]= E_{j}^{(d)} \ast x' \ast E_{i}^{(c)} ~\mod~ v^{-1} \mathcal{L}$. Notice that $x'\ast E_{i}^{(c)}  \in T_{j,0}$, we can see that $[L] \in \mathcal{B}_{j}(d)=\mathcal{B}_{j,d}$, so $t_{j}(L)=d=t_{j}(L_{0})$.
 \end{proof}
\begin{corollary}
	If $[L']=i^{+}([L])$, then $t_{j}(L)=t_{j}(L')$.
\end{corollary}
\begin{proof}
	Assume $ t_{i}^{\ast}(L)=c$. Let $L_{0}$ be the unique simple pervese sheaf such that $t_{i}^{\ast}(L_{0})=0$ and $\pi_{i,c}^{\ast}(L_{0})=L$, then $\pi_{i,c+1}^{\ast}(L_{0})=L'$. Then $t_{j}(L)=t_{j}(L_{0})=t_{j}(L')$.
\end{proof}

\begin{lemma}
	Assume $[K] \in \mathcal{B}_{j,0} \cap \mathcal{B}^{\ast}_{i,0}$, then $\pi_{j,m}\pi_{i,n}^{\ast}([K])=\pi_{i,n}^{\ast}\pi_{j,m}([K])$ for $m,n \geq 0$.
\end{lemma}
\begin{proof}
	Take $x' \in T_{i,0}^{\ast}$ and $x'' \in T_{j,0}$ such that $[K]=x'+v^{-1}z'$ and $[K]=x''+v^{-1}z''$ with $z',z'' \in \mathcal{L}$ as in the proof of Lemma 4.11, then we can see that $K_{1}=\pi_{i,n}^{\ast}\pi_{j,m}(K)$ is the unique simple perverse sheaf such that $[K_{1}]= E_{j}^{(m)}\ast x'' \ast E_{i}^{(n)}~ \mod~ v^{-1}\mathcal{L}$. Similarly, $K_{2}=\pi_{j,m}\pi_{i,n}^{\ast}(K)$ is the unique simple perverse sheaf such that $[K_{2}]= E_{j}^{(m)}\ast x' \ast E_{i}^{(n)} \mod v^{-1}\mathcal{L}$. Notice that $x'=x''~ \mod~ v^{-1}\mathcal{L}$, we can see that $[K_{1}]=[K_{2}]~ \mod ~v^{-1}\mathcal{L}$. By Proposition 4.7, the canonical basis induces a basis of $\mathcal{L}/v^{-1}\mathcal{L}$, we can see that $[K_{1}]=[K_{2}]$.
\end{proof}
\begin{corollary}
	For any $K \in \mathcal{P}_{\mathbf{V},\Omega}$, we have $i^{+}j_{+}([K])=j_{+}i^{+}([K])$.
\end{corollary}
\begin{proof}
	Assume $t_{j}(K)=m$ and $t_{i}^{\ast}(K)=n$, after applying the Lemma 3.4 and 3.6, we can take $K_{0}=(\pi_{j,m})^{-1} (\pi_{i,n}^{\ast})^{-1}(K)$. Then $K_{0}$ is the unique simple perverse sheaf 
	such that $\pi_{j,m} \pi_{i,n}^{\ast} (K_{0})=\pi_{i,n}^{\ast}\pi_{j,m} (K_{0})=K$.  Then by Lemma 4.13, we have:
	\begin{align*}
		i^{+}j_{+}([K])=&i^{+}j_{+}\pi_{j,m} \pi_{i,n}^{\ast} ([K_{0}])= i^{+} \pi_{j,m+1}\pi_{i,n}^{\ast} ([K_{0}]) \\
		=&i^{+}\pi_{i,n}^{\ast}\pi_{j,m+1}([K_{0}])=\pi_{i,n+1}^{\ast}\pi_{j,m+1}([K_{0}])
	\end{align*}
and
    \begin{align*}
		j_{+}i^{+}([K])=&j_{+}i^{+}\pi_{i,n}^{\ast}\pi_{j,m} ([K_{0}])=j_{+} \pi_{i,n+1}^{\ast}\pi_{j,m}([K_{0}])\\
		=&j_{+}\pi_{j,m}\pi_{i,n+1}^{\ast}([K_{0}])=\pi_{j,m+1}\pi_{i,n+1}^{\ast}([K_{0}]).
	\end{align*}
Then by Lemma 4.13, we have $i^{+}j_{+}([K])=j_{+}i^{+}([K])$.
\end{proof}

\subsection{The commutative relation between $i^{+}$ and $i_{+}$}
In this section we study the commutative relation between $i^{+}$ and $i_{+}$.  
\begin{lemma}
	Assume $L_{0} \in \mathcal{P}_{\mathbf{V}',\Omega}$ and $t_{i}^{\ast}(L_{0})=0, t_{i}(L_{0})=d$. Let $L=\pi_{i,c}^{\ast}(L_{0})$ for some $c \in \mathbb{N}$, then 
	\begin{center}
		$t_{i}(L)= \left\{
		\begin{aligned}
			&	d  & c+(|\mathbf{V}'|,i) \leq d \\
			&	c+(|\mathbf{V}'|,i) & otherwise
		\end{aligned}
		\right. $
	\end{center} 
\end{lemma}
\begin{proof}
	Take $x' \in T_{i,0}$ and $z \in \mathcal{L}$ such that $[L_{0}]=E_{i}^{(d)}\ast x' +v^{-1}z$ as in the proof of Lemma 4.11, then we can see that $[L]= E_{i}^{(d)}\ast x' \ast E_{i}^{(c)}~ \mod~ v^{-1} \mathcal{L}$. We denote $E_{i}^{(d)}\ast x' \ast E_{i}^{(c)}$ by $x$. Then by Lemma 4.3, we have $x=\sum \limits_{N \geq 0} \theta_{i}^{(N)}x_{N}$ with $$x_{N}=v^{-(N-1)N/2}\Pi_{N}(x) =\sum\limits_{s \geq 0} (-1)^{s}v^{(s-1)s/2 -(N-1)N/2} E_{i}^{(s)}\ast(_{i}r)^{(s+N)}(x).$$
	
	Using the fact $_{i}r(xy)= {_{i}r}(x)y+v^{(|\mathbf{V}|,i)}x{_{i}r}(y)$ for any homogenous elements $x \in \mathcal{K}_{\mathbf{V}},y\in \mathcal{K}_{\mathbf{V'}}$, we easily get:
	\begin{equation*}
		(_{i}r)^{m}(xy)=\sum\limits_{t \geq 0}^{m}v^{(|\mathbf{V}|-ti,(m-t)i)+t(m-t)} \frac{[m]_{v}!}{[t]_{v}![m-t]_{v}!} (_{i}r)^{t}(x) \ast (_{i}r)^{m-t}(y).
	\end{equation*}	
Take $x=E_{i}^{(d)}\ast x' $ and $y=E_{i}^{(c)} $, we obtain equation (1):
\begin{equation*}
	(_{i}r)^{m}(E_{i}^{(d)}\ast x' \ast E_{i}^{(c)})=\sum\limits_{t \geq 0}^{m}v^{(|\mathbf{V}'|-ti,(m-t)i)+t(m-t)} \frac{[m]_{v}!}{[t]_{v}![m-t]_{v}!} (_{i}r)^{t}(E_{i}^{(d)}\ast x') \ast (_{i}r)^{m-t}(E_{i}^{(c)} ).
\end{equation*}

Notice that by definition of $_{i}r$:
\begin{equation*}
(_{i}r)^{(a)} E_{i}^{(b)}= \left\{
	\begin{aligned}
		& v^{(2b-1-a)a/2}E_{i}^{(b-a)}  & a \geq b \\
		&	0 & otherwise
	\end{aligned}
	\right. 
\end{equation*}
so we have equation (2)
\begin{equation*}
	(_{i}r)^{t}(E_{i}^{(d)}\ast x')= \left\{
	\begin{aligned}
		& v^{(2d-1-t)t/2} E_{i}^{(d-t)} \ast x'  & d \geq t \\
		&	0 & otherwise
	\end{aligned}
	\right. 
\end{equation*}
and equation (3)
\begin{equation*}
	(_{i}r)^{m-t}(E_{i}^{(c)} )= \left\{
	\begin{aligned}
		& v^{(2c-1-m+t)(m-t)/2} E_{i}^{(c-m+t)}  & c \geq m-t \\
		&	0 & otherwise
	\end{aligned}
	\right. 
\end{equation*}

We apply the equations (1),(2) and (3) to $x_{N}$ and assume $c+(|\mathbf{V}'|,i) \leq d$. Notice that $$\frac{[a+b]_{v}!}{[a]_{v}![b]_{v}!}\in v^{ab}\mathbb{Z}[[v^{-1}]]$$ for $a,b \in \mathbb{N}$, we can see that $x_{N}=0~ \mod~ v^{-1} \mathcal{L}$ for $N \neq d$ but $x_{d}=x' \ast E_{i}^{(c)}~ \mod~ v^{-1} \mathcal{L}$. Hence $x \in T_{i,d}$ and $t_{i}(L)=d$.
	
Otherwise, we assume that $k=c+(|\mathbf{V}'|,i)-d>0$.  Applying the equations (1),(2) and (3) to $x_{N}$, we can see that $x_{d+k}=x' \ast E_{i}^{(c-k)}~ \mod~ v^{-1}\mathcal{L}$ and $x_{N}=0~ \mod~ v^{-1}\mathcal{L}$ for $N\neq d+k$. Hence $x \in T_{i,d+k}$ and $t_{i}(L)=d+k=c+(|\mathbf{V}'|,i)$.
	
\end{proof}

By the similar argument, we have the following result dual to Lemma 4.15:

\begin{lemma}
	Assume $L_{0} \in \mathcal{P}_{\mathbf{V}',\Omega}$ and $t_{i}(L_{0})=0, t_{i}^{\ast}(L_{0})=c$. Let $L=\pi_{i,d}(L_{0})$ for some $d \in 	 \mathbb{N}$, then we have:
	\begin{center}
		$t_{i}^{\ast}(L)= \left\{
		\begin{aligned}
			&	c & d+(|\mathbf{V}'|,i) \leq c \\
			&	d+(|\mathbf{V}'|,i)  & otherwise
		\end{aligned}
		\right. $
	\end{center}
\end{lemma}

The following two lemmas give the commutative relation of $i_{+}$ and $i^{+}$ case by case:
\begin{lemma}
Assume $L_{0} \in \mathcal{P}_{\mathbf{V}',\Omega}$ and $t_{i}^{\ast}(L_{0})=0, t_{i}(L_{0})=d>0$. Take $c \in \mathbb{N}_{>0}$ such that $c+(|\mathbf{V}'|,i) \leq d$ and set $L=\pi_{i,c}^{\ast}(L_{0}) $. Let $K_{0}$ be the unique simple perverse sheaf such that $i_{+}(K_{0})=L_{0}$. Then  $t_{i}^{\ast}(K_{0})=0$, hence $\pi_{i,c}^{\ast}(K_{0})$ is well-defined. Moreover, we have $i_{+}\pi_{i,c}^{\ast}(K_{0})=L$. Or equivalently, we have $i_{+}(i^{+})^{c}([K_{0}])=(i^{+})^{c}i_{+}([K_{0}]) $. In this case, $t_{i}^{\ast}(i_{-}(L))=t_{i}^{\ast}(L)=c$.
\end{lemma}
\begin{proof}
	After applying Lemma 3.4, we can take $L_{0}' \in \mathcal{P}_{\mathbf{V''},\Omega}$ to be the unique simple perverse sheaf such that $\pi_{i,d}(L'_{0})=L_{0}$ and $t_{i}(L'_{0})=0$. We assume $t_{i}^{\ast }(L'_{0})=c'$. Then by  Lemma 4.16, $0= t_{i}^{\ast}(L_{0})=max\{d+ (|\mathbf{V}''|,i), c'\}$, hence $c'=0$ and $d+ (|\mathbf{V}''|,i) \leq 0$. Then $t_{i}^{\ast}(K_{0})=max \{d-1 + (|\mathbf{V}''|,i), c'\}=0$ and $\pi_{i,c}^{\ast}(K_{0})$ can be defined.
	
	From the proof of Lemma 4.15, we can see that $[L]= E_{i}^{(d)} \ast x' \ast E_{i}^{(c)}~ \mod~ v^{-1} \mathcal{L}$ for some $x' \in T_{i,0}$ and $[L_{0}]= E_{i}^{(d)} \ast x'~ \mod~ v^{-1} \mathcal{L} $. By definition, $[K_{0}]=E_{i}^{(d-1)} \ast x'~ \mod~ v^{-1} \mathcal{L} $. Then we have $[\pi_{i,c}^{\ast}(K_{0})]=E_{i}^{(d-1)} \ast x' \ast E_{i}^{(c)}=\tilde{\epsilon}_{i}([L])~  \mod~ v^{-1}\mathcal{L}$. Hence $i_{+}\pi_{i,c}^{\ast}(K_{0})=L$ and $i_{+}(i^{+})^{c}([K_{0}])=(i^{+})^{c}i_{+}([K_{0}]) $.
\end{proof}
\begin{lemma}
	Assume $L_{0} \in \mathcal{P}_{\mathbf{V}',\Omega}$ and $t_{i}^{\ast}(L_{0})=0, t_{i}(L_{0})=d>0$. Take $c \in \mathbb{N}_{>0}$ such that $c+(|\mathbf{V}'|,i) > d$ and set $L=\pi_{i,c}^{\ast}(L_{0}) $. Then $i_{-}([L])=i^{-}([L])$, or equivalently, $i_{-}(i^{+})^{c}([L_{0}])=(i^{+})^{c-1}([L_{0}])$ Moreover, we have  $t_{i}^{\ast}(i_{-}(L))=t_{i}^{\ast}(L)-1=c-1$.
\end{lemma}
\begin{proof}
Let $k=c+(|\mathbf{V}'|,i)-d>0 $, then	by the proof of Lemma 4.15, we can see that $[L]= E_{i}^{(d+k) }\ast x' \ast E_{i}^{(c-k)} = E_{i}^{(d)} \ast x' \ast E_{i}^{(c)} ~\mod~ v^{-1}\mathcal{L} $. Then by the definition of $i^{-},i_{-}$ and Proposition 4.7 and Proposition 4.9, we can see that $[i_{-}(L)]= E_{i}^{(d+k-1) }\ast x' \ast E_{i}^{(c-k)}~ \mod~ v^{-1}\mathcal{L}$ and $[i^{-}(L)]=x''\ast E_{i}^{c-1}=E_{i}^{(d) }\ast x' \ast E_{i}^{(c-1)} ~ \mod~ v^{-1}\mathcal{L}$. Here $x'' \in T_{i,0}^{\ast}$. If $k-1>0$, then the proof of Lemma 4.15 implies that $E_{i}^{(d+k-1) }\ast x' \ast E_{i}^{(c-k)}=E_{i}^{(d) }\ast x' \ast E_{i}^{(c-1)}~ \mod~ v^{-1}\mathcal{L}$. On the other hand, if $k-1=0$, the statement $E_{i}^{(d+k-1) }\ast x' \ast E_{i}^{(c-k)}=E_{i}^{(d) }\ast x' \ast E_{i}^{(c-1)}~ \mod~ v^{-1}\mathcal{L}$ trivially holds. Since we always have  $[i^{-}(L)]=[i_{-}(L)]~ \mod~ v^{-1}\mathcal{L}$, we can see that $i_{-}(L)=i^{-}(L)$.
\end{proof}

\section{The key lemmas for $\mathcal{M}$} 
 
 In this section, we will introduce some lemmas about $\mathcal{M}$ parallel to Lemma 3.4 and Lemma 3.6.
 \subsection{The key lemma for the left mutiplication} Let $\Lambda_{\mathbf{V},i,p}$ be the subset of $\Lambda_{\mathbf{V}}$ defined by: $$\Lambda_{\mathbf{V},i,p}=\{x\in \Lambda_{\mathbf{V}}| {\rm{codim}}_{\mathbf{V}_{i}} ( {\rm{Im}} \sum\limits_{h \in H, h''=i} x_{h}) =p\}$$ and $\Lambda_{\mathbf{V},i,  \geq p} = \bigcup \limits_{p' \geq p} \Lambda_{\mathbf{V},i,p'}$. Each $\Lambda_{\mathbf{V},i,p}$ is locally closed in $\Lambda_{\mathbf{V}}$.  Since $\bigcup \limits_{p} \Lambda_{\mathbf{V},i,p}= \Lambda_{\mathbf{V}}$, there is a unique $p$ such that $Z \cap \Lambda_{\mathbf{V},i, p}$ is dense in $Z$ for each irreducible component $Z$ of $\Lambda_{\mathbf{V}}$. In this case, we say that $Z$ generically belongs to $\Lambda_{\mathbf{V},i, p}$ and write $t_{i}(Z)=p$.

 We assume that $|\mathbf{V}'|=pi$ and denote $qr: \Lambda' \rightarrow \Lambda_{\mathbf{V}}$ by $q'$. Then $p^{-1} (\Lambda_{\mathbf{V''},i,0})= (q'^{-1})(\Lambda_{\mathbf{V},i, p})$. We denote $p^{-1} (\Lambda_{\mathbf{V''},i,0})$ by $\Lambda'_{i, p}$ and denote $(q^{-1})(\Lambda_{\mathbf{V},i, p})$ by $\Lambda''_{i,p}$. Then we have the following commutative diagram:
 
 \[
 \xymatrix{
 	\Lambda_{\mathbf{V}'',i,0}\ar[d]^{i_{1}} & \Lambda'_{i,p}\ar[d]^{i_{2}}\ar[l]_{p} \ar[r]^{r}& \Lambda''_{i,p} \ar[d]^{i_{3}}\ar[r]^{q}& \Lambda_{\mathbf{V},i, p} \ar[d]^{i_{4}}\\
 	\Lambda_{\mathbf{V}''} & \Lambda' \ar[l]_{p} \ar[r]^{r} & \Lambda''  \ar[r]^{q} & \Lambda_{\mathbf{V}}
 }
 \]
 here $i_{1},i_{2},i_{3}$ and $i_{4}$ are the natural embeddings.

 The following lemma is from \cite{MR3202708} and we give a sketch of proof here: 
 
 \begin{lemma}[\cite{MR3202708} 4.17] We have:
 	
 	(1) $q': \Lambda'_{i,p} \rightarrow \Lambda_{\mathbf{V},i, p}$ is a principle $G_{\mathbf{V''}} \times G_{\mathbf{V}'}$ -bundle.
 	
 	(2) $p: \Lambda'_{i, p} \rightarrow \Lambda_{\mathbf{V}'',i,0}$ is a smooth map whose fibers are connected of dimension $(\sum \limits_{i \in I} \upsilon_{i}^{2} )- p(\upsilon'',i )$
 \end{lemma}

\begin{proof}
	By definition, $r:\Lambda'_{i,p} \rightarrow \Lambda''_{i,p}$ is a principle bundle. Notice that when $x \in \Lambda_{\mathbf{V},i, p}$, the $x$-stable subspace with dimension $|\mathbf{W}|$ must be ${\rm{Im}}(\bigoplus\limits_{h \in H, h''=i} x_{h}  )$, we get the proof of (1).
	
	The fiber of $p$ at $x''$ is the set consisting of $(x, \widetilde{\mathbf{W}},\rho_{1},\rho_{2})$ such that $(\widetilde{\mathbf{W}},\rho_{1},\rho_{2})$ can be chosen arbitarily and $x$ is an extension of $x''$. $(\widetilde{\mathbf{W}},\rho_{1},\rho_{2})$ is given by a point in $G_{\mathbf{V}}/U$, here $U$ is the unipotent radical of the parabolic subgroup of type $(|\mathbf{V}'|,|\mathbf{V}''|)$. To give an extension of $x''$ is equivalent to give $(y_{h})_{h \in H, h'=i}$ such that the following composition vanishes
	\begin{center}
		$\mathbf{V}'_{i} \xrightarrow{\bigoplus\limits_{h \in H, h'=i} y_{h}} \bigoplus \limits_{i-j} \mathbf{V}''_{j} \xrightarrow{ \sum \limits_{h \in H, h'=i} \epsilon(h)x'_{\bar{h}}} \mathbf{V}''_{i}$
	\end{center}
	By assumption, the rank of the second linear map in the composition is equal to $dim \mathbf{V}''_{i}$. By linear algebra and direct calculation, we get the proof.
	
\end{proof}
 
 \begin{corollary}
 	$q'p^{-1}$ induces a bijection $\eta_{i,p}:Irr \Lambda_{\mathbf{V}'',i,0} \rightarrow  Irr \Lambda_{\mathbf{V},i, p} $ from the set of irreducible components of $\Lambda_{\mathbf{V}'',i,0}$ to the set of irreducible components $\Lambda_{\mathbf{V},i, p}$. The map $\eta_{i,p}$ also induces a bijection ( still denoted by $\eta_{i,p}$) from $\{Z \in Irr\Lambda_{\mathbf{V}''}| t_{i}(Z)=0 \}$ to $\{ Z \in Irr \Lambda_{\mathbf{V}}|t_{i}(Z)=p \} $ defined by: $$\eta_{i,p}( \bar{Z})= \overline{\eta_{i,p}(Z)}$$ for $Z \in Irr \Lambda_{\mathbf{V}'',i,0}$. Here $\bar{Z}$ and $\overline{\eta_{i,p}(Z)}$ are the closure of $Z$ and $\eta_{i,p}(Z)$ respectively.
 \end{corollary}

 \begin{definition}
 	We say a constructible function $f$ generically takes value $c$ on an irreducible component $Z$ if $f^{-1}(c) \cap Z$ is dense in $Z$. In this case, we adapt the notation $\rho_{Z}(f)=c$. Given an irreducible component $Z$, if $f$ generically takes value $1$ on $Z$ but generically takes value $0$ on the other irreducible components, we say $f$ has property $\mathcal{P}(Z)$.
 \end{definition}

If a constructible function $f\in \mathcal{M}$ has property $\mathcal{P}(Z)$, we say $f$ is a Lusztig's function in $\mathcal{M}$ of the irreducible component $Z$. Lusztig has proved the following lemma which implies the existence of such functions and we sketch the proof. (See details in \cite{MR1758244})
 \begin{lemma}[\cite{MR1758244}]
 	For each irreducible component $Z$ of $\Lambda_{\mathbf{V}}$, there exists an $f \in \mathcal{M}_{\mathbf{V}}$  with property $\mathcal{P}(Z)$.  
 \end{lemma}

\begin{proof}
	 We prove by inducton on $|\mathbf{V}|=\upsilon$. Asuume that for any $Z' \in Irr \Lambda_{\mathbf{V}'}$ with $|\mathbf{V}'| < \upsilon$, we have construct some $f_{Z'} \in \mathcal{M}$ which has property $\mathcal{P}(Z')$. For a given $Z \in Irr \Lambda_{\mathbf{V}}$, by \cite{MR1458969} Corollary 1.9, there exists $i \in I$ such that $t_{i}(Z)=p >0$. Then $Z_{1}=Z \cap \Lambda_{\mathbf{V},i,p}$ is an irreducible componnet of $\Lambda_{\mathbf{V},i, p}$. Let $Z_{2}$ be the irreducible component of $\Lambda_{\mathbf{V}'',i,0}$ corresponding to $Z_{1}$ given by Corollary 5.2, then the closure $Z''$ of $Z_{2}$ is an irreducible component of $\Lambda_{\mathbf{V}''}$ with $t_{i}(Z'')=0$. By induction hyphothesis, there exists $f_{Z''} \in \mathcal{M}$ which has property $\mathcal{P}(Z'')$. Then $\tilde{f}=\mathbf{1}_{i^{p}} \ast f_{Z''} \in \mathcal{M}$. Notice that $\tilde{f}(x)= 0$ for $x \in \Lambda_{\mathbf{V}} \backslash \Lambda_{\mathbf{V},i,  \geq p}$ and the restriction of $\tilde{f}$ on $\Lambda_{\mathbf{V},i,p}$ is uniquely determined by the restriction of $f$ on $\Lambda_{\mathbf{V}'',i,0}$. Now we argue by descending induction on $p$. If $p= \dim \mathbf{V}_{i}$, then $\tilde{f}=\mathbf{1}_{i^{p}} \ast f_{Z''}$ has property $\mathcal{P}(Z)$ by construction. If for all $t_{i}(Z'')=p'>p$, we can construct $f_{Z''} \in \mathcal{M}$ which has property $\mathcal{P}(Z'')$. Then the constructible function $\tilde{f}- \sum \limits_{Z'' \in Irr \Lambda_{\mathbf{V}},t_{i}(Z'')>p } \rho_{Z''}(\tilde{f}) f_{Z''}$ will have property $\mathcal{P}(Z)$.
\end{proof}

\begin{remark}
	Since there maybe exist more than one $i\in I$ such that $t_{i}(Z)>0$ for a given irreducible component $Z$, the construction of $f$ depends on the choices of $f_{Z''}$ and  the choice of a sequence of irreducible components $\underline{Z}=(Z_{0},Z_{1},Z_{2},\cdots,Z_{s}), Z_{l} \in Irr\Lambda_{\mathbf{V}_{l}}$ for $0\leq l\leq s$ and a sequence $\underline{i}=(i_{1},i_{2},\cdots,i_{s}), i_{l} \in I$ for $1 \leq l \leq s$ satisfying the following conditions:
	
	(1) $Z_{0}=Z$
	
	(2) $|\mathbf{V}_{s}|= qj$ for some $q>0,j \in I$ and $Z_{s}$ is the unique irreducible component of $\Lambda_{\mathbf{V}_{s}}$.
	
	(3) For any $1 \leq l \leq s$, $t_{i_{l}}(Z_{l-1})=p_{l}>0$ and $Z_{l}$ is the corresponding irreducible component of $Z_{l-1}$ given by Corollary 5.2. More precisly, $\eta_{i_{l},p_{l}}(Z_{l})=Z_{l-1}$.
\end{remark}
	
If a pair $\underline{p}=(\underline{Z},\underline{i})$ of sequences satisfies the three conditions above, we say $\underline{p}$ is a left-addmissible path of  $Z$. Given a left-addmissible path $\underline{p}$, we can construct a constructible function $f$ having property $\mathcal{P}(Z)$. We denote this constructible function (still not unique) by $f_{\underline{p}}$.

From the proof above, we also get the following result: 
\begin{lemma}
	If $Z'$ is an irreducible component of $\Lambda_{\mathbf{V}''}$ with $t_{i}(Z')=0$ and $Z=\eta_{i,p}(Z')$, and we choose a constructible function $f_{Z'}$ having property  $\mathcal{P}(Z')$, then $$\mathbf{1}_{i^{p}}\ast f_{Z'}=f_{Z} +\sum \limits_{Z'' \in Irr \Lambda_{\mathbf{V}},t_{i}(Z'')>p } c_{Z''}f_{Z''},$$ where $f_{Z}, f_{Z''}\in \mathcal{M}_{V} $ are some constructible functions having property $\mathcal{P}(Z),\mathcal{P}(Z'')$ respectively and $c_{Z''} \in \mathbb{Z}$.
\end{lemma}

We call the lemma above the key lemma of Lusztig's functions for left mutiplication.

 \subsection{ The key lemma for the right multiplication}
 Let $\Lambda_{\mathbf{V},i}^{p}$ be the subset of $\Lambda_{\mathbf{V}}$ defined by  $$\Lambda_{\mathbf{V},i}^{p}=\{x \in \Lambda_{\mathbf{V}} | \dim (\Ker \bigoplus \limits_{h \in H, h'=i} x_{h} )=p \}$$ and $\Lambda_{\mathbf{V},i}^{\geq p} = \bigcup \limits_{p' \geq p}  \Lambda_{\mathbf{V},i}^{p'}$. Since $ \bigcup\limits_{p} \Lambda_{\mathbf{V},i}^{p} =\Lambda_{\mathbf{V}}$, for each irreducible component $Z$, there is a unique $p$ such that $\Lambda_{\mathbf{V},i}^{p} \cap Z$ is dense in $Z$ and we write $t_{i}^{\ast}(Z)=p$.
 
 If we asuume $|\mathbf{V''}|=pi$, then $p^{-1}(\Lambda_{\mathbf{V}',i}^{0})=(q')^{-1}(\Lambda_{\mathbf{V},i}^{p})$ and we denote this space by $\Lambda_{i}^{'{p}}$ and let $\Lambda_{i}^{''{p}}=r(\Lambda_{i}^{'{p}})$. Then we have the following commutative diagram:
 
 \[
 \xymatrix{
 	\Lambda_{\mathbf{V}',i}^{0}\ar[d]^{j_{1}} & \Lambda^{'{p}}_{i}\ar[d]^{j_{2}}\ar[l]_{p} \ar[r]^{r}& \Lambda^{''{p}}_{i} \ar[d]^{j_{3}}\ar[r]^{q}& \Lambda_{\mathbf{V},i}^{p} \ar[d]^{j_{4}}\\
 	\Lambda_{\mathbf{V}'} & \Lambda' \ar[l]_{p} \ar[r]^{r} & \Lambda''  \ar[r]^{q} & \Lambda_{\mathbf{V}}
 }
 \]
 here $j_{1},j_{2},j_{3}$ and $j_{4}$ are the natural embeddings. Then by the same argument, we get the following results dual to those in Section 4.1. 
 \begin{lemma}We have:
 	
 	(1) $q':\Lambda_{i}^{'{p}} \rightarrow 
 	\Lambda_{\mathbf{V},i}^{p}$ is a principle $G_{\mathbf{V'}} \times G_{\mathbf{V''}}$-bundle.
 	
 	(2) $p:\Lambda_{i}^{'{p}} \rightarrow \Lambda_{\mathbf{V}',i}^{0}$ is a smooth map whose fibers are connected of dimension $(\sum \limits_{i \in I} \upsilon_{i}^{2}) - p(\upsilon',i ) +p(\upsilon'_{i}-p)$.
 \end{lemma}

The map $q'p^{-1}$ also induces a bijection of irreducible components parallel to Corollary 5.2.
 \begin{corollary}
 	We have a bijection $\eta_{i,p}^{\ast}: Irr \Lambda_{\mathbf{V}',i}^{0} \rightarrow Irr \Lambda_{\mathbf{V},i}^{p}$ between the sets of irreducible components induced by $q^{'}p^{-1}$. The map $\eta_{i,p}^{\ast}$ also induces a bijection from $\{Z \in \Lambda_{\mathbf{V}'}| t_{i}^{\ast}(Z)=0 \}$ to $\{Z \in Irr \Lambda_{\mathbf{V}}|t_{i}^{\ast}(Z)=p\}$ defined by: $$\eta_{i,p}^{\ast}( \bar{Z})= \overline{\eta_{i,p}^{\ast}(Z)}$$ for $Z \in Irr \Lambda_{\mathbf{V}',i}^{0}$. Here $\bar{Z}$ and $\overline{\eta_{i,p}^{\ast}(Z)}$ are the closure of $Z$ and $\eta_{i,p}^{\ast}(Z)$ respectively.
 \end{corollary}

 We also has the following key lemma of Lusztig's functions for right multiplication:
 \begin{lemma}
 	If $Z'$ is an irreducible component of $\Lambda_{\mathbf{V}'}$ with $t_{i}^{\ast}(Z')=0$ and $Z=\eta_{i,p}^{\ast}(Z')$ , and we choose a constructible function $f_{Z'}$ having property  $\mathcal{P}(Z')$, then $$ f_{Z'} \ast \mathbf{1}_{i^{p}}=f_{Z} +\sum \limits_{Z'' \in Irr \Lambda_{\mathbf{V}},t_{i}^{\ast}(Z'')>p } c'_{Z''}f_{Z''},$$ where $f_{Z}, f_{Z''}\in \mathcal{M}_{V} $ are some constructible functions having property $\mathcal{P}(Z),\mathcal{P}(Z'')$ respectively and $c'_{Z''} \in \mathbb{Z}$.
 \end{lemma}

\begin{definition}
 Assume that  $\underline{Z}=(Z_{0},Z_{1},Z_{2},\cdots,Z_{s}), Z_{l} \in Irr\Lambda_{\mathbf{V}_{l}}$ for $0\leq l\leq s$ is a sequence of irreducible components and  $\underline{i}=(i_{1},i_{2},\cdots,i_{s}), i_{l} \in I$ for $1 \leq l \leq s$ is a sequence of vertex. We say $\underline{p}=(\underline{Z},\underline{i})$ is a right-admissible path of $Z$ if the following three conditions holds:

(1) $Z_{0}=Z$

(2) $|\mathbf{V}_{s}|= qj$ for some $q>0,j \in I$ and $Z_{s}$ is the unique irreducible component of $\Lambda_{\mathbf{V}_{s}}$.

(3) For any $1 \leq l \leq s$, $t_{i_{l}}^{\ast}(Z_{l-1})=p_{l}>0$ and $Z_{l}$ is the corresponding irreducible component of $Z_{l-1}$ given by corollary 5.8. More precisely, $\eta_{i_{l},p_{l}}^{\ast}(Z_{l})=Z_{l-1}$.
\end{definition}

Given a right-admissible path $\underline{p}$ of $Z$, we can construct a constructible function $f$ having property $\mathcal{P}(Z)$ by right mutiplication as in Lemma 5.4. We denote this constructible function by $f_{\underline{p}}$.
 
 \subsection{The key lemmas for the coproduct}
 Fix a decomposition $\mathbf{T} \oplus \mathbf{W} =\mathbf{V}$ of graded vector space, we can define $Res^{\mathbf{V}}_{\mathbf{T},\mathbf{W}}: \tilde{M}(\Lambda_{\mathbf{V}}) \rightarrow \tilde{M}(\Lambda_{\mathbf{T}}) \times \tilde{M}(\Lambda_{\mathbf{W}})$ by:
 \begin{center}
 	$Res^{\mathbf{V}}_{\mathbf{T},\mathbf{W}}(f)(x',x'')=f(x'\oplus x'')$
 \end{center}
 
 Geiss, Leclerc and Schroer has proved in \cite{MR2144987} that $$Res^{\mathbf{V}}_{\mathbf{T},\mathbf{W}}(\pi_{\underline{v}}) _{!}(\mathbf{1}_{\tilde{\mathcal{F}}_{\underline{v}}}) =\sum\limits_{\underline{v'},\underline{v}''}(\pi_{\underline{v}'}) _{!}(\mathbf{1}_{\tilde{\mathcal{F}}_{\underline{v}'}}) \otimes (\pi_{\underline{v}''}) _{!}(\mathbf{1}_{\tilde{\mathcal{F}}_{\underline{v}''}}) $$ here $\underline{v'},\underline{v}''$ run over all flag types such that $\underline{v}'$ is a flag type of $\mathbf{T}$, $\underline{v}''$ is a flag type of $\mathbf{W}$ and $\underline{v}' + \underline{v}'' =\underline{v}$. Hence $Res^{\mathbf{V}}_{\mathbf{T},\mathbf{W}}:\mathcal{M}_{\mathbf{V}} \rightarrow \mathcal{M}_{\mathbf{T}}  \otimes \mathcal{M}_{\mathbf{W}}$. Consider all graded space $\mathbf{V}$, $Res$ gives a coproduct $\Delta: \mathcal{M} \rightarrow  \mathcal{M} \otimes \mathcal{M}$.

 Now we assume $|\mathbf{T}|=pi$ and consider the following diagram:
 \begin{center}
 	$\Lambda_{\mathbf{W}} \cong \Lambda_{\mathbf{T}} \times \Lambda_{\mathbf{W}} \xleftarrow{ \kappa} F \xrightarrow{\iota} \Lambda_{\mathbf{V}}$
 \end{center}
 here $F$ is the closed subset of $\Lambda_{\mathbf{V}}$ consisting of $x$ such that $\mathbf{W}$ is $x$-stable. $\iota$ is the natural embbeding and $\kappa(x)=x|_{\mathbf{W}}$. $\kappa$ admits a section $i: \Lambda_{\mathbf{W}} \rightarrow F$, sending $x'' \in \Lambda_{\mathbf{W}}$ to the direct sum of $x''$ and the semisimple module supported on $i \in I$. Then by definition, we can see that $Res^{\mathbf{V}}_{\mathbf{T},\mathbf{W}}(f) =i^{\ast} \iota^{\ast}(f)$.
 
 Notice that $\kappa^{-1}(\Lambda_{\mathbf{W},i,0}) = \iota^{-1} (\Lambda_{\mathbf{V},i,p})$, we denote this set by $F_{i,p}$. We also denote $q^{-1}(\Lambda_{\mathbf{V},i,p})$ by $\Lambda''_{i,p}$. Then we have the following proposition:
 \begin{proposition} We have
 	
 	(1) Each fiber of $\kappa : F_{i,p} \rightarrow \Lambda_{\mathbf{W},i,0}$ is an affine space of dimension $p(\sum \limits_{j \in I,j-i} \upsilon''_{j} -\upsilon''_{i}) $.  Here we adapt the notation $j-i$ if  there exists  $h \in H$ connecting $i$ and $j$.
 	
 	(2) $q: \Lambda''_{i,p} \rightarrow \Lambda_{\mathbf{V},i,p}$ is an isomorphism.
 	
 	(3) Let $\mathbf{G}$ be the variety consisting of $\{\tilde{\mathbf{W}} \subset \mathbf{V}| |\mathbf{W}|=|\tilde{\mathbf{W}}| \}$. Consider the projection $\rho: \Lambda''_{i,p} \rightarrow \mathbf{G}, \rho(x, \tilde{\mathbf{W}})= \tilde{\mathbf{W}}$, then $F_{i,p}$ is naturally identified with the fiber of $\rho$ at $\mathbf{W}$. There is a natural embedding $j: F_{i,p} \rightarrow \Lambda''_{i,p}$ such that $qj=\iota$.
 \end{proposition}

 \[
\xymatrix{
	\Lambda_{\mathbf{W},i,0}\ar[d]^{i_{1}} & F_{i,p}\ar[d]^{i_{2}}\ar[l]_{\kappa} \ar[r]^{\iota}& \Lambda_{\mathbf{V},i,p} \ar[d]^{i_{3}}\\
	\Lambda_{\mathbf{W}} & F \ar[l]_{\kappa} \ar[r]^{\iota}  & \Lambda_{\mathbf{V}}
}
\]
 
 \begin{proof}
 	(1) A fiber at $x' \in \Lambda_{\mathbf{W},i,0}$ can be identified with the closed subset consisting of $ y\in \Hom(\mathbf{T}_{i}, \bigoplus\limits_{i-j}\mathbf{W}_{j} )$ such that the composition
 	\begin{center}
 		$\mathbf{T}_{i} \xrightarrow{\bigoplus \limits_{h \in H, h'=i} y_{h}} \bigoplus\limits_{i-j} \mathbf{W}_{j} \xrightarrow{ \sum \limits_{h \in H, h'=i} \epsilon(h)x'_{\bar{h}}} \mathbf{W}_{i}$
 	\end{center} 
 	is zero. Since $x' \in \Lambda_{\mathbf{W},i,0}$, the second map is surjective. Hence by linear algebra, we can see that the fiber is an affine space with dimension $p(\sum \limits_{j \in I,j-i} \upsilon''_{j} -\upsilon''_{i})$.
 	
 	(2) Notice that when $x \in \Lambda_{\mathbf{V},i, p}$, the $x$-stable subspace with dimension $|\mathbf{W}|$ must be ${\rm{Im}}(\sum\limits_{h \in H, h''=i} x_{h}  )$, then (2) is trivial.
 	
 	(3) Check by definition.
 \end{proof}
 
 Then we can see that 
 \begin{corollary}
 	The map $\kappa \iota^{-1} $ induces a bijection $\epsilon_{i,p}: Irr \Lambda_{\mathbf{V},i,p} \rightarrow Irr \Lambda_{\mathbf{W},i,0}$. And the map $\epsilon_{i,p}$ also induces a bijection from  $\{Z \in Irr\Lambda_{\mathbf{V}}|t_{i}(Z)=p\}$ to $\{Z \in Irr\Lambda_{\mathbf{W}}|t_{i}(Z)=0\}$, which is the inverse of $\eta_{i,p}$.
 \end{corollary}
\begin{proof}
	 By Proposition 5.11, we can see that $\epsilon_{i,p}: Irr \Lambda_{\mathbf{V},i,p} \rightarrow Irr \Lambda_{\mathbf{W},i,0}$ is a bijection. Notice that both $\eta_{i,p}(Z')$ and $\epsilon_{i,p}^{-1}(Z')$ are exactly the unique irreducible component containing those $x'$ which is isomorphic to a direct sum of some $x \in Z' \cap \Lambda_{\mathbf{W},i,0}$ and $p$ simple representations at $i$, we can see that $\epsilon_{i,p}$ is the inverse of $\eta_{i,p}$.
\end{proof}
 
 As a corollary, we have the following key lemma of Lusztig's functions for coproduct:
 \begin{lemma}
 	With the noation above, assume that $Z$ is an irreducible component of $\Lambda_{\mathbf{V}}$ with $t_{i}(Z)=p$ and $Z'=\epsilon_{i,p}(Z)$. Let $\underline{p}=(\underline{Z},\underline{i})$ be a left-addmissible path of $Z$ such that $i_{1}=i$ and $Z_{1}=Z'$, then $$Res^{\mathbf{V}}_{\mathbf{T},\mathbf{W}}(f_{\underline{p}})=f_{Z'}+\sum\limits_{Z'' \in Irr \Lambda_{\mathbf{W}},t_{i}(Z'')>0 } d_{Z''}f_{Z''},$$ where $f_{Z'}, f_{Z''}\in \mathcal{M}_{V} $ are some constructible functions having property $\mathcal{P}(Z'),\mathcal{P}(Z'')$ respectively and $d_{Z''} \in \mathbb{Z}$.
 \end{lemma}
 \[
\xymatrix{
	\Lambda_{\mathbf{W},i,0}\ar[d]^{i_{1}} \ar[r]^{i} & F_{i,p}\ar[d]^{i_{2}}\ar[r]^{\iota} & \Lambda_{\mathbf{V},i,p} \ar[d]^{i_{3}}\\
	\Lambda_{\mathbf{W}} \ar[r]^{i} & F  \ar[r]^{\iota}  & \Lambda_{\mathbf{V}}
}
\]

 \begin{proof}
 	Notice that if $(Z_{0},Z_{1},\cdots,Z_{s})$ and $(i_{1},i_{2},\cdots,i_{s})$ form a left-addmissible path of $Z$, then $(Z_{1},Z_{2},\cdots,Z_{s})$ and $(i_{2},i_{3},\cdots,i_{s})$ form a left-admissible path of $Z'$. We denote this left-addmissible path by $\underline{p}'$. Notice that $Res^{\mathbf{V}}_{\mathbf{T},\mathbf{W}}(f_{\underline{p}})|_{\Lambda_{\mathbf{W},i,0}}$ is determined by $f_{\underline{p}}|_{\Lambda_{\mathbf{V},i, p}}$. From the construction of $f_{\underline{p}}$, we can see that $Res^{\mathbf{V}}_{\mathbf{T},\mathbf{W}}(f_{\underline{p}})|_{\Lambda_{\mathbf{W},i,0}}=f_{\underline{p}'}|_{\Lambda_{\mathbf{W},i,0}}$ genericly takes value $1$ on $Z'$ and genericly takes value $0$ on the other $Z''$ with $t_{i}(Z'')=0$.  Take $d_{Z''}= \rho_{Z''} (Res^{\mathbf{V}}_{\mathbf{T},\mathbf{W}}(f_{Z}) ) $ for those $Z''$ with $t_{i}(Z'')>0$ and choose a $f_{Z''} \in \mathcal{M}$ having property $\mathcal{P}(Z'')$ for each $Z''$ with $t_{i}(Z'')>0$. Let $f_{Z'}=Res^{\mathbf{V}}_{\mathbf{T},\mathbf{W}}(f_{Z}) -\sum\limits_{Z'' \in Irr \Lambda_{\mathbf{W}},t_{i}(Z'')>0 } d_{Z''}f_{Z''} \in \mathcal{M}$, then $f_{Z'}\in \mathcal{M}$ has property $\mathcal{P}(Z')$ and we get the proof.
 \end{proof}
 
 Dually,  we assume $|\mathbf{W}|=pi$ and consider the following diagram:
 \begin{center}
 	$\Lambda_{\mathbf{T}} \cong \Lambda_{\mathbf{T}} \times \Lambda_{\mathbf{W}} \xleftarrow{\kappa} F \xrightarrow{\iota} \Lambda_{\mathbf{V}}$
 \end{center}
 Notice that $\kappa^{-1}(\Lambda_{\mathbf{T},i}^{0})= \iota^{-1}(\Lambda_{\mathbf{V},i}^{p})$ and we denote this space by $F_{i}^{p}$, then we have the folliwng result dual to Proposition 5.11:
 \begin{proposition} 
 	
 	(1) Each fiber of $\kappa:F_{i}^{p} \rightarrow \Lambda_{\mathbf{T},i}^{0}$ is an affine space of dimension $p(\sum\limits_{j \in I,j-i} v_{j}'- v_{i}')$.
 	
 	(2) $q: \Lambda_{i}^{''{p}} \rightarrow \Lambda_{\mathbf{V},i}^{p}$ is an isomorphism.
 	
 	(3) Consider $\rho: \Lambda_{i}^{''{p}} \rightarrow \mathbf{G}, \rho(x,\widetilde{\mathbf{W}})= \widetilde{\mathbf{W}}$, then $F_{i}^{p}$ is naturally identifed with the fiber of $\rho$ at $\mathbf{W}$ and there is a natural embedding $j:F_{i}^{p} \rightarrow \Lambda_{i}^{''{p}}$ such that $qj=\iota$.
 \end{proposition}

 With the above proposition, we can easily obtain the following corollary dual to Corollary 5.12:
 \begin{corollary}
 	There is a natural bijection $\epsilon_{i,p}^{\ast}: Irr \Lambda_{\mathbf{V},i}^{p} \rightarrow Irr \Lambda_{\mathbf{T},i}^{0}$ indeuced by	$\kappa \iota^{-1}$. And the map  $\epsilon_{i,p}^{\ast}$ also induces a bijection from $\{Z \in Irr\Lambda_{\mathbf{V}}|t_{i}^{\ast}(Z)=p\}$ to  $\{Z \in  Irr\Lambda_{\mathbf{T}}| t_{i}^{\ast}(Z)=0\}$, which is the inverse of $\eta_{i,p}^{\ast}$.
 \end{corollary}

 The following lemma is the key lemma of Lusztig's functions for coproduct, which is dual to Lemma 5.13:
 \begin{lemma}
 	With the notations above, assume that $Z$ is an irreducible component of $\Lambda_{\mathbf{V}}$ with $t_{i}^{\ast}(Z)=p$ and $Z'=\epsilon_{i,p}^{\ast}(Z)$. Let $\underline{p}=(\underline{Z},\underline{i})$ be a right-addmissible path of $Z$ such that $Z_{1}=Z'$ and $i_{1}=i$, then $$Res^{\mathbf{V}}_{\mathbf{T},\mathbf{W}}(f_{\underline{p}})=f_{Z'}+\sum\limits_{Z'' \in Irr \Lambda_{\mathbf{W}},t_{i}^{\ast}(Z'')>0 } d'_{Z''}f_{Z''},$$ where $f_{Z'}, f_{Z''}\in \mathcal{M}_{V} $ are some constructible functions having property $\mathcal{P}(Z'),\mathcal{P}(Z'')$ respectively and $d'_{Z''} \in \mathbb{Z}$.
 \end{lemma}

 \section{An $I \times \mathbb{Z}_{2}$-colored graph $\mathcal{G}_{2}$ of $\mathcal{M}$}
 \subsection{The graph $\mathcal{G}_{2}$}
 
 \begin{definition}
 	Given two constructible functions $f,g $, if on each irreducible componnet $Z$, $f$ and $g$ generically take the same value, we say $f$ generically equals to $g$. It defines an equivalent relation. Given $f\in \mathcal{M}$, we denote the equivalent class of $f $ by $[f].$
 \end{definition}
  Given an irreducible component $Z\in Irr\Lambda_{\mathbf{V}'}$, we let $[f_{Z}]$ be the equivalent class consisting of those $f_{Z} \in \mathcal{M}$ with property $\mathcal{P}(Z)$. Then by Lemma 5.6, $\eta_{i,p}$ induces a bijection (still denoted by $\eta_{i,p}$) between the set of the equivalent classes $\{[f] |f$ has property $\mathcal{P}(Z')$ for some $t_{i}(Z')=0  \}$ and the set of the equivalent classes $\{[f] |f$ has property $\mathcal{P}(Z)$ for some $t_{i}(Z)=p  \}$ as the following: Let $f \in [f]$ with property $\mathcal{P}(Z')$ with $t_{i}(Z')=0$, then $$1_{i^{p}}\ast f=f_{Z} +\sum \limits_{Z'' \in Irr \Lambda_{\mathbf{V}},t_{i}(Z'')>p } c_{Z''}f_{Z''}$$ for some $f_{Z}$ with property $\mathcal{P}(Z)$. Even though $f_{Z}$ is not unique, the equivalent class  $[f_{Z}]$ is unique and consists of elements having property  $\mathcal{P}(\eta_{i,p}(Z'))$ .  
  
  Similarly, $\eta^{\ast}_{i,p}$ induces a bijection (still denoted by $\eta^{\ast}_{i,p}$) from the set of the equivalent classes $\{[f] |f$ has property $\mathcal{P}(Z')$ for some $t_{i}^{\ast}(Z')=0  \}$ to the set of the equivalent classes $\{[f] |f$ has property $\mathcal{P}(Z)$ for some $t_{i}^{\ast}(Z)=p  \}$.\\
 
 Given an equivalent class $[f]$ such that $f \in \mathcal{M}$ has property $\mathcal{P}(Z)$ for $t_{i}(Z)=n$, we define $[g]=i_{+}([f])=\eta_{i,n+1} \eta_{i,n}^{-1}([f])$ and associate an arrow $[f]\xrightarrow{i_{+}}[g]$. Similarly, if $f$ has property $\mathcal{P}(Z)$ for $t_{i}^{\ast}(Z)=n$, we define $[g]=i^{+}([f])=\eta^{\ast}_{i,n+1} (\eta^{\ast}_{i,n})^{-1}([f])$ and associate an arrow $[f]\xrightarrow{i^{+}}[g]$. \\ 
 
Dually, if $[f]$ is an equivalent class such that $f \in \mathcal{M}$ has property $\mathcal{P}(Z)$ for $t_{i}(Z)=n>0$, we define $i_{-}([f])=\eta_{i,n-1} \eta_{i,n}^{-1}([f])$. If $[f]$ is an equivalent class such that $f \in \mathcal{M}$ such that $f \in \mathcal{M}$ has property $\mathcal{P}(Z)$ for $t_{i}^{\ast}(Z)=n>0$, we define $i^{-}([f])=\eta^{\ast}_{i,n-1} (\eta^{\ast}_{i,n})^{-1}([f])$. Then by construction, we can see that $i_{-},i^{-}$ are the inverse of $i_{+},i^{+}$ respectively.

 \begin{definition}
 	We define a $\mathbb{Z}_{2} \times I$-colored graph $\mathcal{G}_{2}=(\mathcal{V}_{2},\mathcal{E}_{2})$ as the following:\\
 	
 	(1) The set of vertices is $\mathcal{V}_{2}=\{[f]|f \in \mathcal{M}_{\mathbf{V}}$ has property $\mathcal{P}(Z)$ for some $\mathbf{V}$ and irreducible component $Z \subset \Lambda_{\mathbf{V}}\}$.\\
 
 	(2) The set of arrows is $\mathcal{E}_{2}=\{[f]\xrightarrow{i_{+}}[g],[f]\xrightarrow{i^{+}}[g]|i\in I,[f],[g]\in \mathcal{V}_{2} \}$
  
 \end{definition}
\begin{remark}
	Notice that $\mathcal{V}_{2} $ is naturally bijective to the set $\bigcup\limits_{\mathbf{V}}Irr \Lambda_{\mathbf{V}}$. If $[f] \in \mathcal{V}_{2} $ such that $f \in [f]$ has proerty $\mathcal{P}(Z)$, we adapt the notation $t_{i}([f])=t_{i}(Z)$ and $t_{i}^{\ast}([f])=t_{i}^{\ast}(Z)$.
\end{remark}

 \subsection{The commutative relation between $i^{+}$ and $j_{+}$}
 
 In this section, we prove the lemmas parallel to those in Section 4.3
 
 \begin{lemma}
 	Assume $[f_{0}] \in \mathcal{V}_{2}$ such that $t_{i}^{\ast}([f_{0}])=0$ and $[f]=\eta_{i,c}^{\ast}([f_{0}]) $, then we have $t_{j}([f_{0}])=t_{j}([f])$.
 \end{lemma}
\begin{proof}
	Assume $t_{j}([f_{0}])=d$ and $f_{0}$ has property $\mathcal{P}(Z')$, then $t_{j}(Z')=d$ and $t_{i}^{\ast}(Z')=0$.
	Consider the diagram for $|\mathbf{V}''|=ci$.
 \[
\xymatrix{
	\Lambda_{\mathbf{V}',i}^{0}\ar[d]^{j_{1}} & \Lambda^{'{c}}_{i}\ar[d]^{j_{2}}\ar[l]_{p} \ar[r]^{r}& \Lambda^{''{c}}_{i} \ar[d]^{j_{3}}\ar[r]^{q}& \Lambda_{\mathbf{V},i}^{c} \ar[d]^{j_{4}}\\
	\Lambda_{\mathbf{V}'} & \Lambda' \ar[l]_{p} \ar[r]^{r} & \Lambda''  \ar[r]^{q} & \Lambda_{\mathbf{V}}
}
\]
and the composition 
\begin{center}
		$\mathbf{V}'_{i} \xrightarrow{\bigoplus\limits_{h \in H, h'=i} x'_{h}} \bigoplus \limits_{i-j} \mathbf{V}'_{j} \xrightarrow{ \sum
			 \limits_{h \in H, h'=i} \epsilon(h)y_{\bar{h}}} \mathbf{V}''_{i}$
\end{center}

Notice that  for $x' \in \Lambda_{\mathbf{V}',i}^{0} \cap \Lambda_{\mathbf{V}',j,d}$ and $x \in q'p^{-1}(x')$, the number:
\begin{center}
	$ {\rm{codim}}_{\mathbf{V}_{j}} ( {\rm{Im}} \sum\limits_{h \in H, h''=j} x_{h})={\rm{codim}}_{\mathbf{V}_{j}} ( {\rm{Im}} \sum\limits_{h \in H, h''=j} x'_{h})=d $
\end{center} does not depned on the choice of those $y_{\bar{h}}$, we can see that $t_{j}([f_{0}])=t_{j}([f])$.
\end{proof}

By similar argument as in the proof of Corollary 4.12, we have:
\begin{corollary}
	If $[f]\in \mathcal{V}_{2}$ and $[g]=i^{+}([f])$, then $t_{j}([g])=t_{j}([f])$.
\end{corollary}

\begin{lemma}
	Assume $[f_{0}] \in \mathcal{V}_{2}$ such that $t_{i}^{\ast}([f_{0}])=t_{j}([f_{0}]) =0$, then  $$\eta_{j,d}\eta_{i,c}^{\ast}([f_{0}])=\eta_{i,c}^{\ast} \eta_{j,d}([f_{0}])$$ for any $c,d \in \mathbb{N}$.
\end{lemma}

\begin{proof}
	By Lemma 6.4, we can see that $t_{j}(\eta_{i,c}^{\ast}([f_{0}]))=0$, hence $\eta_{j,d}\eta_{i,c}^{\ast}([f_{0}])$ is well-defined. By the similar argument, $\eta_{i,c}^{\ast} \eta_{j,d}([f_{0}])$ is also well-defined. Let $[f_{1}]=\eta_{j,d}\eta_{i,c}^{\ast}([f_{0}]), [f_{2}]= \eta_{i,c}^{\ast} \eta_{j,d}([f_{0}])$, then we have $t_{j}([f_{1}])=d=t_{j}([f_{2}]) $ and $t_{i}^\ast([f_{1}])=c=t_{i}^{\ast}([f_{2}]) $. 
	
	Given any $x \in\Lambda_{\mathbf{V},i}^{c} \cap \Lambda_{\mathbf{V},j,d}$, there exist unique $x$-stable subspaces $\mathbf{W},\mathbf{T}$ such that $|\mathbf{W}|+dj=|\mathbf{V}|, |\mathbf{T}|=ci$ respectively. Since $\mathbf{T} \subset \mathbf{W}$, we can define $x |_{ \mathbf{W}/ \mathbf{T}}$ . It determines an orbit $\mathcal{O}(x) \subset \Lambda_{\mathbf{V}''}$ with $|\mathbf{V''}|+dj+ci=|\mathbf{V}|$.   We define $\tilde{f}:\Lambda_{\mathbf{V},i}^{c} \cap \Lambda_{\mathbf{V},j,d} $ by $\tilde{f}(x) =f_{0}(x')$ if $x' \in \mathcal{O}(x)$. Then by the construction of $f_{1}$ in Lemma 5.4, Lemma 5.6 and Lemma 5.9, we can see that $f_{1}|_{\Lambda_{\mathbf{V},i}^{c} \cap \Lambda_{\mathbf{V},j,d}  } $ generically equals to $g|_{{\Lambda_{\mathbf{V},i}^{c} \cap \Lambda_{\mathbf{V},j,d} }}$ with  $g=\mathbf{1}_{j^{(d)}}\ast f_{0} \ast \mathbf{1}_{i^{(c)}}$. By the similar argument, we can see that $f_{2}|_{\Lambda_{\mathbf{V},i}^{c} \cap \Lambda_{\mathbf{V},j,d}  } $ generically equals to $g|_{{\Lambda_{\mathbf{V},i}^{c} \cap \Lambda_{\mathbf{V},j,d} }}$, too. Hence $$\eta_{j,d}\eta_{i,c}^{\ast}([f_{0}])=\eta_{i,c}^{\ast} \eta_{j,d}([f_{0}]).$$
\end{proof}
By the similar argument as in Corollary 4.14, we have the following corollary:
\begin{corollary}
	Given $[f] \in \mathcal{V}_{2}$, we have $j_{+}i^{+}([f])=i^{+}j_{+}([f]).$
\end{corollary}

  \subsection{The commutative relation between $i^{+}$ and $i_{+}$}
 In this subsection, we prove the lemmas parallel to those in Section 4.4.
\begin{lemma}
	Assume $[f_{0}] \in \mathcal{V}_{2}$ with $f_{0} \in \mathcal{M}_{\mathbf{V}'}$ and $t_{i}^{\ast}([f_{0}])=0, t_{i}([f_{0}])=d$. Let $[f]=\eta_{i,c}^{\ast}([f_{0}])$, then 
	
		\begin{center}
		$t_{i}([f])= \left\{
		\begin{aligned}
			&	d  & c+(|\mathbf{V}'|,i) \leq d \\
			&	c+(|\mathbf{V}'|,i) & otherwise
		\end{aligned}
		\right. $
	\end{center} 
\end{lemma}
\begin{proof}
	
	Assume that $f_{0}$ has property $\mathcal{P}(Z')$ and $f$ has property $\mathcal{P}(Z)$. Consider the diagram for $|\mathbf{V}''|=ci$.
\[
\xymatrix{
	\Lambda_{\mathbf{V}',i}^{0}\ar[d]^{j_{1}} & \Lambda^{'{c}}_{i}\ar[d]^{j_{2}}\ar[l]_{p} \ar[r]^{r}& \Lambda^{''{c}}_{i} \ar[d]^{j_{3}}\ar[r]^{q}& \Lambda_{\mathbf{V},i}^{c} \ar[d]^{j_{4}}\\
	\Lambda_{\mathbf{V}'} & \Lambda' \ar[l]_{p} \ar[r]^{r} & \Lambda''  \ar[r]^{q} & \Lambda_{\mathbf{V}}
}
\]
and the composition
\begin{center}
	$\mathbf{V}'_{i} \xrightarrow{\bigoplus\limits_{h \in H, h'=i} x'_{h}} \bigoplus \limits_{i-j} \mathbf{V}'_{j} \xrightarrow{ \left(\begin{matrix}
			\sum \limits_{h \in H, h'=i} \epsilon(h)x'_{\bar{h}} \\
			\sum \limits_{h \in H, h'=i} \epsilon(h)y_{\bar{h}}
		\end{matrix}\right)  }   \mathbf{V}'_{i}\oplus \mathbf{V}''_{i}$
\end{center}

The fiber of $p$ at $x'$ is the set consisting of $(x, \widetilde{\mathbf{W}},\rho_{1},\rho_{2})$ such that $(\widetilde{\mathbf{W}},\rho_{1},\rho_{2})$ can be chosen arbitarily and $x$ is an extension of $x'$. $(\widetilde{\mathbf{W}},\rho_{1},\rho_{2})$ is given by a point in $G_{\mathbf{V}}/U$, here $U$ is the unipotent radical of the parabolic subgroup of type $(|\mathbf{V}'|,|\mathbf{V}''|)$. To give an extension of $x'$ is equivalent to give $(y_{\bar{h}})_{h \in H, h'=i}$ such that the above composition vanishes.

For any $x' \in \Lambda_{\mathbf{V}',i}^{0} \cap \Lambda_{\mathbf{V}',i,d}$, $\bigoplus\limits_{h \in H, h'=i} x'_{h}$ is injective and $\dim ({\rm{Im}} \sum \limits_{h \in H, h'=i} \epsilon(h)x'_{\bar{h}} )= \upsilon_{i}'-d $.  Fix  bases of $\mathbf{V}' $ and $\mathbf{V''}$, we denote the matrix of $\bigoplus\limits_{h \in H, h'=i} x'_{h}$ by $X'$, $\sum \limits_{h \in H, h'=i} \epsilon(h)x'_{\bar{h}}$ by $\bar{X}'$ and $\sum \limits_{h \in H, h'=i} \epsilon(h)y_{\bar{h}}$ by $\bar{Y}$ respectively. Then each raw of $\bar{X}',\bar{Y} $ is contained in $\mathbf{K}=\{z | zX' =0 \}$. Since $\bigoplus\limits_{h \in H, h'=i} x'_{h}$ is injective , ${\rm{rank}}(X')= \upsilon'_{i}$, we can see that $$\dim(\mathbf{K})=\sum\limits_{i-j}\upsilon'_{j}-\upsilon'_{i}= -(|\mathbf{V}'|,i)+\upsilon'_{i}. $$ Hence we have $${\rm{rank}}(\left(\begin{matrix}
	\bar{X}' \\
	\bar{Y}
\end{matrix}\right)) \leq -(|\mathbf{V}'|,i)+\upsilon'_{i}. $$

If $c+(|\mathbf{V}'|,i) \leq d$, then $\upsilon_{i}-d \leq \sum\limits_{i-j}\upsilon'_{j}-\upsilon'_{i}$. Let $\mathcal{O}_{1}=f_{0}^{-1}(1) \cap Z' \cap \Lambda_{\mathbf{V}',i}^{0} \cap \Lambda_{\mathbf{V}',i,d}$, then $\mathcal{O}_{1}$ is dense in $Z'$. Let $\mathcal{O}_{2}$ be the set consisting of $x$ such that $x$ is an extension of $x' \in \mathcal{O}_{1}$ and ${\rm{rank}}(\left(\begin{matrix}
	\bar{X}' \\
	\bar{Y}
\end{matrix}\right))={\rm{rank}}(\bar{X}')+{\rm{rank}}(\bar{Y})$. Notice that on each fiber $p^{-1}(x')$, the condition that ${\rm{rank}}(\left(\begin{matrix}
\bar{X}' \\
\bar{Y}
\end{matrix}\right))={\rm{rank}}(\bar{X}')+{\rm{rank}}(\bar{Y})$ is an open condition, we can see that  $\mathcal{O}_{2}$ is dense in $Z$. On the other hand, $f_{0}\ast \mathbf{1}_{i^{(c)}}|_{\mathcal{O}_{2}}=1$ by the construction of $\mathcal{O}_{2}$ and $\mathcal{O}_{2} \subset \Lambda_{\mathbf{V},i,d}$.  Hence $t_{i}([f])=t_{i}(Z)=d$.

If $c+(|\mathbf{V}'|,i) > d$, then on each fiber $p^{-1}(x')$,  the condition that ${\rm{rank}}(\left(\begin{matrix}
	\bar{X}' \\
	\bar{Y}
\end{matrix}\right))=\dim(\mathbf{K})$ becomes an open condition. If we let $\mathcal{O}_{3}$ be the set consisting of $x$ such that $x$ is an extension of $x' \in \mathcal{O}_{1}$ and ${\rm{rank}}(\left(\begin{matrix}
\bar{X}' \\
\bar{Y}
\end{matrix}\right))=\dim(\mathbf{K})$, we can see that $f_{0}\ast \mathbf{1}_{i^{(c)}}|_{\mathcal{O}_{3}}=1$ and $\mathcal{O}_{3} \subset \Lambda_{\mathbf{V},i,c+(|\mathbf{V}'|,i)}$. Hence $t_{i}([f])=c+(|\mathbf{V}'|,i)$.

\end{proof}

Replace $t_{i},t_{i}^{\ast}$ by $t_{i}^{\ast},t_{i}$ respectively, we have the following dual result:
\begin{lemma}
	Assume $[f_{0}] \in \mathcal{V}_{2}$ with $f_{0} \in \mathcal{M}_{\mathbf{V}'}$ and $t_{i}([f_{0}])=0, t_{i}^{\ast}([f_{0}])=c$. Let $[f]=\eta_{i,d}([f_{0}])$, then
			\begin{center}
		$t_{i}^{\ast}([f])= \left\{
		\begin{aligned}
			&	c  & d+(|\mathbf{V}'|,i) \leq c\\
			&	d+(|\mathbf{V}'|,i)  & otherwise
		\end{aligned}
		\right. $
	\end{center} 
\end{lemma}

The following lemma is parallel to Lemma 4.17:
\begin{lemma}
	Assume $[f_{0}] \in \mathcal{V}_{2}$ with $f_{0} \in \mathcal{M}_{\mathbf{V}'}$ and $t_{i}^{\ast}([f_{0}])=0, t_{i}(f_{0})=d>0$. Take $c \in \mathbb{N}_{>0}$ such that $c+(|\mathbf{V}'|,i) \leq d$ and set $[f]=\eta_{i,c}^{\ast}([f_{0}]) $. Let $[g_{0}]$ be the unique equivalent class such that $i_{+}([g_{0}])=[f_{0}]$. Then  $t_{i}^{\ast}([g_{0}])=0$, hence $\eta_{i,c}^{\ast}([g_{0}])$ is well-defined. Moreover,  we have $$i_{+}\eta_{i,c}^{\ast}([g_{0}])=[f].$$ Or equivalently, we have $$i_{+}(i^{+})^{c}([g_{0}])=(i^{+})^{c}i_{+}(g_{0}). $$ In this case, $$t_{i}^{\ast}(i_{-}([f]))=t_{i}^{\ast}([f])=c.$$
\end{lemma}
\begin{proof}
	We assume that $f_{0}$ has property $\mathcal{P}(Z')$ and $f$ has property $\mathcal{P}(Z)$. By the similar argument in Lemma 4.17, we can prove that $t_{i}^{\ast}([g_{0}])=0$, hence $[g]=\eta_{i,c}^{\ast}([g_{0}])$ is well-defined.

	Let $[h_{0}]$ be the unique equivalent class such that $\eta_{i,d}([h_{0}])=[f_{0}]$ with $h_{0} \in \mathcal{M}_{\tilde{\mathbf{V}}} $ having property $\mathcal{P}(Z_{0})$. Then by Lemma 6.9, we can see that $t_{i}^{\ast}([h_{0}])=0$, otherwise we have $t_{i}^{\ast}([f_{0}])>0$. Hence $[h]=\eta_{i,c}^{\ast}([h_{0}])$ is well-defined. Moreover, with the assumption $c+(|\mathbf{V}'|,i) \leq d$, we can see that $t_{i}([h])=0$ for the similar reason, hence $\eta_{i,d}([h]) $ is defined. 
	
	We claim that $\eta_{i,d}([h])=\eta_{i,c}^{\ast}([f_{0}])=[f]$. We assume that $h$ has property $\tilde{Z}$. Then from the proof of Lemma 6.8 and Lemma 6.9, we can see that  $Z$ contains a dense subset $\mathcal{O}_{1} \subset \Lambda_{\mathbf{V},i,d} \cap \Lambda_{\mathbf{V},i}^{c} \cap Z$ such that $\mathbf{1}_{i^{(d)}} \ast h_{0} \ast \mathbf{1}_{i^{(c)}}|_{\mathcal{O}_{1}}=1 $.  Similarly, $\tilde{Z}$ contains a dense subset $\mathcal{O}_{2} \subset \Lambda_{\mathbf{V},i,d} \cap \Lambda_{\mathbf{V},i}^{c} \cap \tilde{Z}$ such that $\mathbf{1}_{i^{(d)}} \ast h_{0} \ast \mathbf{1}_{i^{(c)}}|_{\mathcal{O}_{2}}=1 $. Notice that there exists a unique irreducible componet  of $\Lambda_{\mathbf{V},i,d} \cap \Lambda_{\mathbf{V},i}^{c}$ such that $\mathbf{1}_{i^{(d)}} \ast h_{0} \ast \mathbf{1}_{i^{(c)}} $ generically takes value $1$ on it. We can see that $Z=\tilde{Z}$ and  $$\eta_{i,d}([h])=\eta_{i,c}^{\ast}([f_{0}])=[f].$$
	
	Replace $[f],[f_{0}]$ by $[g],[g_{0}]$ respectively, for the same reason we can see that $$\eta_{i,d-1}([h])=\eta_{i,c}^{\ast}([g_{0}])=[g],$$ hence we have $$i_{+}([g])=i_{+}\eta_{i,d-1}([h]) =\eta_{i,d}([h])=[f].$$

\end{proof}

The following lemma is parallel to Lemma 4.18:
\begin{lemma}
	Assume $[f_{0}] \in \mathcal{V}_{2}$ with $f_{0} \in \mathcal{M}_{\mathbf{V}'}$ and $t_{i}^{\ast}([f_{0}])=0, t_{i}([f_{0}])=d$. Take $c \in \mathbb{N}_{>0}$ such that $c+(|\mathbf{V}'|,i) > d$ and set $[f]=\eta_{i,c}^{\ast}([f_{0}]) $. Then $$i_{-}([f])=i^{-}([f]).$$ Moreover, we have  $$t_{i}^{\ast}(i_{-}([f]))=t_{i}^{\ast}([f])-1=c-1.$$
\end{lemma}
\begin{proof}
	Let $k=c+(|\mathbf{V}'|,i)-d>0$.  Let $[h_{0}]$ be the unique equivalent class such that $\eta_{i,d}([h_{0}])=[f_{0}]$ with $h_{0} \in \mathcal{M}_{\tilde{\mathbf{V}}} $ having property $\mathcal{P}(Z_{0})$ and set $[h]=\eta_{i,c-k}^{\ast}([h_{0}])$. We claim that $\eta_{i,d+k}([h])=[f]$.
	
	Indeed, if we assume $f$ has proerty $\mathcal{P}(Z_{1})$ and $f'\in \eta_{i,d+k}([h])$ has property $\mathcal{P}(Z_{2})$, then by the proof of Lemma 6.8, we can see that $Z_{1}$ contains a dense subset $\mathcal{O}_{1}$ consisting of $x$ such that $x$ is an extension of $x'$ given by the following composition
\begin{center}
	$\mathbf{V}'_{i} \xrightarrow{\bigoplus\limits_{h \in H, h'=i} x'_{h}} \bigoplus \limits_{i-j} \mathbf{V}'_{j} \xrightarrow{ \left(\begin{matrix}
			\sum \limits_{h \in H, h'=i} \epsilon(h)x'_{\bar{h}} \\
			\sum \limits_{h \in H, h'=i} \epsilon(h)y_{\bar{h}}
		\end{matrix}\right)  }   \mathbf{V}'_{i}\oplus \mathbf{V}''_{i}$
\end{center}
	with ${\rm{rank}}(\left(\begin{matrix}
		\bar{X}' \\
		\bar{Y}
	\end{matrix}\right))=\sum\limits_{i-j} \upsilon_{j}-\upsilon_{i}+c$. 

Similarly, $Z_{2}$ contains a dense subset $\mathcal{O}_{2}$ consisting of $x$ such that $x$ is an extension of $x''$ given by the following composition
	\begin{center}
	$\mathbf{V}'_{i} \oplus \mathbf{V}''_{i} \xrightarrow{(\bigoplus\limits_{h \in H, h'=i} y_{h}, \bigoplus\limits_{h \in H, h'=i} x''_{h} )} \bigoplus \limits_{i-j} \mathbf{V}''_{j} \xrightarrow{  \sum \limits_{h \in H, h'=i} \epsilon(h)x''_{\bar{h}}}  \mathbf{V}''_{i}$
\end{center}
with ${\rm{rank}}((Y,X''))=\upsilon_{i}-c$. Here $(Y,X'')$ is the matrix of $(\bigoplus\limits_{h \in H, h'=i} y_{h}, \bigoplus\limits_{h \in H, h'=i} x''_{h} ) $ under  fixed bases.

Consider the composition
$\mathbf{V}_{i}  \xrightarrow{\bigoplus\limits_{h \in H, h'=i} x_{h}} \bigoplus \limits_{i-j} \mathbf{V}_{j} \xrightarrow{  \sum \limits_{h \in H, h'=i} \epsilon(h)x_{\bar{h}}}  \mathbf{V}_{i}$ and let $X,\bar{X}$ be the matrix of $\bigoplus\limits_{h \in H, h'=i} x_{h}, \sum\limits_{h \in H, h'=i}\epsilon(h) x_{\bar{h}}  $ under fixed bases respectively. Notice that for a fixed $X_{0}$ with ${\rm{rank}}(X_{0})=\upsilon_{i}-c $, then any $\bar{X}$ such that $\bar{X}X_{0}=0$ has rank $\leq \sum\limits_{i-j} \upsilon_{j}-\upsilon_{i}+c$, so $\{\bar{X}|\bar{X}X_{0}=0, {\rm{rank}}(\bar{X})= \sum\limits_{i-j} \upsilon_{j}-\upsilon_{i}+c \}$ is a dense subset of $\{\bar{X}| \bar{X}X_{0}=0 \}$. Dually, for fixed $\bar{X}_{0}$ with rank $\sum\limits_{i-j} \upsilon_{j}-\upsilon_{i}+c$, $ \{X| \bar{X}_{0}X=0,{\rm{rank}}(X)=\sum\limits_{i-j} \upsilon_{j}-\upsilon_{i}+c\}$ is a dense subset of $\{X| \bar{X}_{0}X=0\}$. We can see that $\mathcal{O}_{1} \cap \mathcal{O}_{2}$ is noempty and dense in $\mathcal{O}_{1},\mathcal{O}_{2}$, so $Z_{1}=Z_{2}$ and $\eta_{i,d+k}([h])=[f]$.
	
	By Lemma 6.10, we have $$\eta_{i,d}\eta_{i,c-k}^{\ast}([h_{0}])=\eta_{i,c-k}^{\ast} \eta_{i,d}([h_{0}]).$$ We denote this equivalent class by $[g]$. By the claim, we can see that $[f]=(i^{+})^{k}([g])= (i_{+})^{k}([g])$. Replace $k$ by $k-1$, we also have $(i^{+})^{k-1}([g])= (i_{+})^{k-1}([g])$, hence $i_{-}([f])=i^{-}([f])$. and $t_{i}^{\ast}(i_{-}([f]))=t_{i}^{\ast}([f])-1=c-1$.
\end{proof}
\section{The isomorphism of the graphs}

In this section, we will construct an isomorphism between the colored graph $\mathcal{G}_{1}$ and $\mathcal{G}_{2}$ by induction. We will denote $t_{i}(\mathcal{F}_{\Omega,\Omega_{i}}(L))$ and $t_{i}^{\ast}(\mathcal{F}_{\Omega,\Omega^{i}}(L))$  by $t_{i}(L)$ and $t_{i}^{\ast}(L)$ respectively  in this section, if there is no ambiguity.

Let $L_{0}$ be the unique simple perverse sheaf on $\mathbf{E}_{\mathbf{V}_{0}}$ with $\mathbf{V}_{0}=0$. By Lemma 7.2 in \cite{MR1088333}, for any $L \in \mathcal{V}_{1}$, there exists a path  $[L_{0}] \xrightarrow{i_{1+}} [L_{1}] \xrightarrow{i_{2+}} \cdots \xrightarrow{i_{l+}} [L_{l}]=[L]$ in $\mathcal{G}_{1}$ and we write $h(L)=l$.  Notice that $h(L)= \sum\limits_{i \in I}\dim \mathbf{V}_{i} $ for $L \in \mathcal{P}_{\mathbf{V}}$, it does not depend on the choice of path.

Similarly, let $\mathbf{1}_{0}$ be the constant function on $\Lambda_{\mathbf{V}_{0}}$ with $\mathbf{V}_{0}=0$. By Corollary 1.6 in \cite{MR1758244}, for any $[f] \in \mathcal{V}_{2}$, there exists a path  $[\mathbf{1}_{0}] \xrightarrow{i_{1+}} [f_{1}] \xrightarrow{i_{2+}} \cdots \xrightarrow{i_{l+}} [f_{l}]=[f]$ in $\mathcal{G}_{2}$ and we write $h([f])=l$. Notice that $h([f])= \sum\limits_{i \in I}\dim \mathbf{V}_{i}$ for $f \in \mathcal{M}_{\mathbf{V}}$, it does not depend on the choice of path. 

Let $\mathcal{V}_{1}^{\leq m}$ be the subset of $\mathcal{V}_{1}$ consisting of those $[L]$ with $h(L) \leq m$ and let $\mathcal{E}_{1}^{\leq m}$ be the subset of $\mathcal{E}_{1}$ consisting of arrows between vertex in $\mathcal{V}_{1}^{\leq m}$:  $$\mathcal{V}_{1}^{\leq m}=\{[L] \in \mathcal{V}_{1}| h(L) \leq m\}$$
$$\mathcal{E}_{1}^{\leq m}=\{  [L] \xrightarrow{i^{+}} [K], [L] \xrightarrow{i_{+}} [K]| i \in I, [K],[L] \in \mathcal{V}^{\leq m}_{1}   \}$$
Then we can see that $\mathcal{G}_{1}^{\leq m}= (\mathcal{V}_{1}^{\leq m}, \mathcal{E}_{1}^{\leq m})$ is a full subgraph of $\mathcal{G}_{1}$.  

Similarly, let $\mathcal{V}_{2}^{\leq m}$ be the subset of $\mathcal{V}_{2}$ consisting of those $[f]$ with $h([f]) \leq m$ and let $\mathcal{E}_{2}^{\leq m}$ be the subset of $\mathcal{E}_{2}$ consisting of arrows between vertices in $\mathcal{V}_{2}^{\leq m}$:
$$\mathcal{V}_{2}^{\leq m}=\{[f] \in \mathcal{V}_{2}| h([f]) \leq m\}$$
$$\mathcal{E}_{2}^{\leq m}=\{  [f] \xrightarrow{i^{+}} [g], [f] \xrightarrow{i_{+}} [g]| i \in I, [f],[g] \in \mathcal{V}^{\leq m}_{2}   \}$$
 Then $\mathcal{G}_{2}^{\leq m}= (\mathcal{V}_{2}^{\leq m}, \mathcal{E}_{2}^{\leq m})$ is a full subgraph of $\mathcal{G}_{2}$.

For any fixed $L \in \mathcal{V}_{1}$, we choose a path  $[L_{0}] \xrightarrow{i_{1+}} [L_{1}] \xrightarrow{i_{2+}} \cdots \xrightarrow{i_{l+}} [L_{l}]=[L]$ in $\mathcal{G}_{1}$ and set $\Phi([L])=i_{l+}i_{(l-1)+}\cdots i_{1+}([\mathbf{1}_{0}])$.
\begin{theorem}
	  $\Phi:\mathcal{G}_{1} \rightarrow \mathcal{G}_{2}$ is well-defined and is an isomorphism of $I \times \mathbb{Z}_{2}$- colored graphs. More precisely, $\Phi([L])$ does not depend on the choice of the path $[L_{0}] \xrightarrow{i_{1+}} [L_{1}] \xrightarrow{i_{2+}} \cdots \xrightarrow{i_{l+}} [L_{l}]=[L]$ for any $[L] \in \mathcal{V}_{1}$. $\Phi: \mathcal{V}_{1} \rightarrow \mathcal{V}_{2}$ is bijective and $\Phi$ commutes with $i_{+}$ and $i^{+}$ for any $i \in I$. Moreover, for any $i \in I$ and $L \in \mathcal{V}_{1}$, we have $t_{i}(L)=t_{i}(\Phi([L])),t_{i}^{\ast}(L)=t_{i}^{\ast}(\Phi([L]))$.
\end{theorem}

\begin{proof}
	We only need to show the restriction $\Phi: \mathcal{G}_{1}^{\leq m} \rightarrow \mathcal{G}_{2} ^{\leq m}$ is an isomorphism of colored graphs for any $m$. When $m=1$, the theorem trivially holds.
	Now we assume that $\Phi([L])$ is well-defined for all $L$ with $h(L) \leq n-1$ and $\Phi: \mathcal{G}_{1}^{\leq n-1} \rightarrow \mathcal{G}_{2} ^{\leq n-1}$ is an isomorphism of $\mathbb{Z}_{2}\times I$-colored graphs which preserves $t_{i}$ and $t_{i}^{\ast}$. We need to show the statement also holds for $n$ under the inductive assumption. We divide our proof into four parts: Firstly, we prove that $\Phi([L])$ does not depend on the choice of path for $h(L)=n$, then $\Phi: \mathcal{V}_{1}^{\leq n} \rightarrow \mathcal{V}_{2}^{\leq n}$ is well-defined. Next we prove that $\Phi: \mathcal{V}_{1}^{\leq n} \rightarrow \mathcal{V}_{2}^{\leq n}$ is bijective. Then we prove that $\Phi$ commutes with arrows in $\mathcal{E}_{1}^{\leq n}$ and $\mathcal{E}_{2}^{\leq n}$. Finally, we prove that $t_{i}(L)=t_{i}(\Phi([L]))$ and $t_{i}^{\ast}(L)=t_{i}^{\ast}(\Phi([L]))$ for any $L$ with $h(L)=n$.\\

	\textbf{Step one}: We prove that $\Phi([L])$ is well-defined for $h(L)=n$. Fix a simple perverse sheaf $L$ with $h(L)=n$ and choose different paths $[L_{0}] \xrightarrow{i_{1+}} [L_{1}] \xrightarrow{i_{2+}}  \cdots \xrightarrow{i_{n+}} [L_{n}]=[L]$ and $[L_{0}] \xrightarrow{i'_{1+}} [L'_{1}] \xrightarrow{i'_{2+}}  \cdots \xrightarrow{i'_{n+}} [L'_{n}]=[L]$, we need to show that $i_{l+}i_{(l-1)+}\cdots i_{1+}([\mathbf{1}_{0}])=i'_{l+}i'_{(l-1)+}\cdots i'_{1+}([\mathbf{1}_{0}])$. We denote the left hand side by $[f_{1}]$ and the right hand side by $[f_{2}]$. By the induction hyphothesis, $[g_{1}]=\Phi([L_{n-1}])=i_{(l-1)+}\cdots i_{1+}([\mathbf{1}_{0}]) $ and $[g_{2}]=\Phi([L'_{n-1}])=i'_{(l-1)+}\cdots i'_{1+}([\mathbf{1}_{0}]) $ are well-defined and we have $[f_{1}]=i_{n+}([g_{1}]), [f_{2}]=i'_{n+}([g_{2}])$.\\
	 
	case(1): If $i_{n}=i'_{n}=i$, notice that $[L'_{n-1}]=i_{-}([L]) =[L_{n-1}]$ and $[g_{1}]=[g_{2}]=\Phi(i_{-}([L]))$ is well-defined, then we have $[f_{1}]=[g_{1}]=i_{n+}(\Phi(i_{-}([L])))$ in this case.\\
	
	case(2):  If $i_{n} \neq i'_{n}$ and there exists $j \in I\backslash \{i_{n},i'_{n} \}$ such that $t_{j}^{\ast}([L]) > 0$. Consider $[K]=j^{-}([L])$, then $[K] \in \mathcal{V}_{1} ^{\leq n-1}$. By the induction hyphothesis, $\Phi(K)$ is well-defined. Since $t_{j}^{\ast}(L)>0$, then by Corollary 4.12, we have $t_{j}^{\ast}(L_{n-1})=t_{j}^{\ast}(i_{n-}(L))>0$. Hence we can take $K_{1}=j^{-}(L_{n-1}) \in \mathcal{V}_{1}^{\leq n-2}$. Similarly, we can take $K_{2}=j^{-}(L'_{n-1})\in \mathcal{V}_{1}^{\leq n-2}$. By the induction hyphothesis, $[h]=\Phi([K]), [h_{1}]=\Phi([K_{1}])$ and $[h_{2}]=\Phi([K_{2}])$ are well-defined and $[g_{1}]=j^{+} ([h_{1}]), [g_{2}]=j^{+}([h_{2}])$. Since $j \neq i_{n}, i'_{n}$, by Corollary 4.14 $j^{+}$ commutes with $i_{n+}$ and $i'_{n+}$ in $\mathcal{G}_{1}$, so $K=i_{n+}(K_{1})=i'_{n+}(K_{2})$. By the indcution hyphothesis, we have $[h]=i_{n+}([h_{1}])=i'_{n+}([h_{2}])$. By Corollary 6.7 $j^{+}$ also commutes with $i_{n+}$ and $i'_{n+}$ in $\mathcal{G}_{2}$, we have $$[f_{1}]=i_{n+}([g_{1}])=i_{n+}j^{+}([h_{1}])=j^{+}i_{n+}([h_{1}])=j^{+}([h])$$ $$[f_{2}]= i'_{n+}([g_{2}])=i'_{n+}j^{+}([h_{2}])=j^{+}i'_{n+}([h_{2}])=j^{+}([h]).$$ Hence in this case $[f_{1}]=[f_{2}]$.
	
	
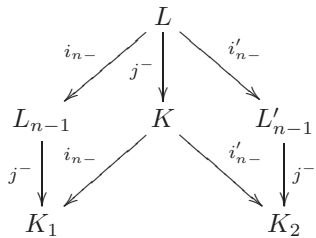
\[
	\xymatrix{
		 & L\ar[dl]_{i_{n-}} \ar[d]_{j^{-}} \ar[dr]^{i'_{n-}}& \\
		L_{n-1} \ar[d]_{j^{-}} & K  \ar[dl]_{i_{n-}} \ar[dr]^{i'_{n-}}  & L'_{n-1} \ar[d]^{j^{-}} \\
		K_{1} &  & K_{2}	
	}
	\]

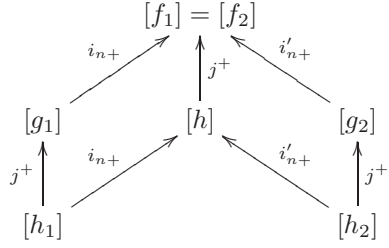
\[
	\xymatrix{
		& [f_{1}]=[f_{2}]  & \\
		[g_{1}] \ar[ur]^{i_{n+}} & [h] \ar[u]_{j^{+}}     & [g_{2}] \ar[ul]_{i'_{n+}} \\
		[h_{1}] \ar[ur]^{i_{n+}} \ar[u]^{j^{+}} &  & [h_{2}]	\ar[u]_{j^{+}} \ar[ul]_{i'_{n+}}
	}
	\]

	case(3): If $i_{n} \neq i'_{n}$ and there only exists $j = i_{n}$ or $i'_{n} \in I$ such that $t_{j}^{\ast}(L)=c > 0$ (the other $t_{j}^{\ast}(L)=0$). Without loss of generality, we can assume $j=i_{n}$. Consier $K=(j^{-})^{c}(L)\in \mathcal{V}_{1}^{\leq n-1}$, then $\Phi([K])=[h]$ is well-defined. Notice that by Lemma 4.17 and Lemma 4.18, 
			\begin{center}
		$ \left\{
		\begin{aligned}
			&(j^{-})^{c}j_{-}(L)=j_{-}(j^{-})^{c}(L)	  & if~ t_{j}^{\ast}(L_{n-1})=c  \\
			&	(j^{-})^{c-1}j_{-}(L)=(j^{-})^{c}(L) & if~t_{j}^{\ast}(L_{n-1})=c-1 
		\end{aligned}
		\right. $
	\end{center} 
	
	If the conidtions in Lemma 4.17 holds and $t_{j}^{\ast}(L_{n-1})=c$, we can take $K_{1}= j_{-}(K),K_{2}=i'_{n-}(K)$ and set $[h_{1}]= \Phi([K_{1}]),[h_{2}]=\Phi([K_{2}])$. Then $$[f_{1}]=j_{+}([g_{1}])=j_{+} \Phi ( [L_{n-1}])=j_{+}\Phi ( (j^{+})^{c}([K_{1}]))=j_{+}(j^{+})^{c}([h_{1}]).$$ Notice that in this case, the conditions in Lemma 6.10 holds, so in $\mathcal{G}_{2}$ we have $$j_{+}(j^{+})^{c}([h_{1}])=(j^{+})^{c}j_{+}([h_{1}]).$$ Hence $[f_{1}]=(j^{+})^{c}([h])$ holds. On the other hand, since $j \neq i'_{n}$, we have $j^{+}$ commutes with $i'_{n+}$. By the same argument as in case (2), we can prove $[f_{2}]=(j^{+})^{c}([h])$. Then $[f_{1}]=[f_{2}]$ in this case. 

		
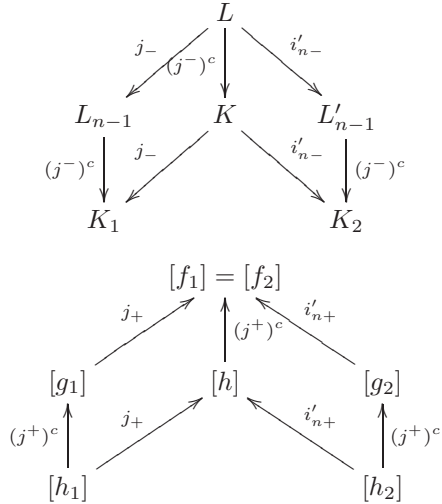
\[
	\xymatrix{
		& L\ar[dl]_{j_{-}} \ar[d]_{(j^{-})^{c}} \ar[dr]^{i'_{n-}}& \\
		L_{n-1} \ar[d]_{(j^{-})^{c}} & K  \ar[dl]_{j_{-}} \ar[dr]^{i'_{n-}}  & L'_{n-1} \ar[d]^{(j^{-})^{c}} \\
		K_{1} &  & K_{2}	
	}
	\]

	\[
	\xymatrix{
		& [f_{1}]=[f_{2}]  & \\
		[g_{1}] \ar[ur]^{j_{+}} & [h] \ar[u]_{(j^{+})^{c}}     & [g_{2}] \ar[ul]_{i'_{n+}} \\
		[h_{1}] \ar[ur]^{j_{+}} \ar[u]^{(j^{+})^{c}} &  & [h_{2}]	\ar[u]_{(j^{+})^{c}} \ar[ul]_{i'_{n+}}
	}
	\]

	If the conditions in Lemma 4.18 holds and $t_{j}^{\ast}(L_{n-1})=c-1$, we take $K_{1}=K$ and take $K_{2}$ as before. Then $(j^{+})^{c-1}(K)=L_{n-1}$. We have $$[f_{1}]=j_{+}([g_{1}])=j_{+} \Phi ( L_{n-1})=j_{+}\Phi ( (j^{+})^{c-1}(K_{1})).$$ Notice that in this case the conditions in Lemma 6.11 holds and in $\mathcal{G}_{2}$ we have $$j_{+}(j^{+})^{c-1}([h])=(j^{+})^{c}([h]).$$ Hence $[f_{1}]= (j^{+})^{c}([h])$. On the other hand, $[f_{2}]= (j^{+})^{c}([h])$ for the same reason as in the case (2). Hence $[f_{1}]=[f_{2}]$ in this case.
	
			
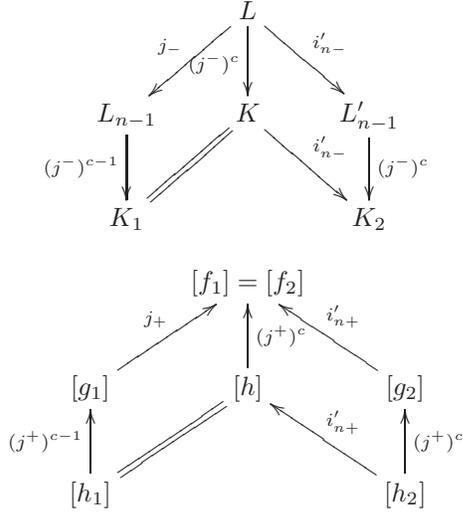
\[
	\xymatrix{
		& L\ar[dl]_{j_{-}} \ar[d]_{(j^{-})^{c}} \ar[dr]^{i'_{n-}}& \\
		L_{n-1} \ar[d]_{(j^{-})^{c-1}} & K  \ar@2{-}[dl] \ar[dr]^{i'_{n-}}  & L'_{n-1} \ar[d]^{(j^{-})^{c}} \\
		K_{1} &  & K_{2}	
	}
	\]

	\[
	\xymatrix{
		& [f_{1}]=[f_{2}]  & \\
		[g_{1}] \ar[ur]^{j_{+}} & [h] \ar[u]_{(j^{+})^{c}}     & [g_{2}] \ar[ul]_{i'_{n+}} \\
		[h_{1}] \ar@2{-}[ur] \ar[u]^{(j^{+})^{c-1}} &  & [h_{2}]	\ar[u]_{(j^{+})^{c}} \ar[ul]_{i'_{n+}}
	}
	\]
	
	Since in any cases we always have $[f_{1}]=[f_{2}]$,  $\Phi([L])$ is independent of the choice of the path $[L_{0}] \xrightarrow{i_{1+}} [L_{1}] \xrightarrow{i_{2+}}  \cdots \xrightarrow{i_{n+}} [L_{n}]=[L]$ and $\Phi:\mathcal{V}_{1}^{\leq n} \rightarrow \mathcal{V}_{2}^{\leq n}$ is well-defined.\\

	\textbf{Step two}: We prove  that $\Phi:\mathcal{V}_{1}^{\leq n} \rightarrow \mathcal{V}_{2}^{\leq n}$ is bijective. Indeed, we can construct a well-defined map $\Psi: \mathcal{G}_{2}^{\leq n} \rightarrow \mathcal{G}_{1}^{\leq n}$ in a similar way. 
	
	If $\Phi:\mathcal{V}_{1}^{\leq n} \rightarrow \mathcal{V}_{2}^{\leq n}$ is not injective, we can find simple perverse sheaves $L \neq K$ and their paths $[L_{0}] \xrightarrow{i_{1+}} [L_{1}] \xrightarrow{i_{2+}}  \cdots \xrightarrow{i_{n+}} [L_{n}]=[L]$  and $[L_{0}] \xrightarrow{j_{1+}} [K_{1}] \xrightarrow{j_{2+}}  \cdots \xrightarrow{j_{n+}} [K_{n}]=[K]$ such that $[f]=i_{n+}i_{(n-1)+} \cdots i_{1+}([\mathbf{1}_{0}])= j_{n+}j_{(n-1)+} \cdots j_{1+}([\mathbf{1}_{0}])$. Then by construction of $\Psi$, we can choose paths $[\mathbf{1}_{0}]\xrightarrow{i_{1+}} [f_{1}] \xrightarrow{i_{2+}}  \cdots \xrightarrow{i_{n+}} [f]$ and $[\mathbf{1}_{0}]\xrightarrow{j_{1+}} [g_{1}] \xrightarrow{j_{2+}}  \cdots \xrightarrow{j_{n+}} [f]$ for $[f]$. Then $\Psi([f])=[K]=[L]$, which is a contradiction. 
	
	For the same reason, $\Psi$ is injective and inverse to $\Phi$. Hence $\Phi:\mathcal{V}_{1}^{\leq n} \rightarrow \mathcal{V}_{2}^{\leq n}$ is bijective.\\
	
	 \textbf{Step three}: Now we show that $\Phi:\mathcal{G}_{1}^{\leq n} \rightarrow \mathcal{G}_{2} ^{\leq n}$ commutes with arrows $i_{+}$ and $i^{+}$ for $i \in I$. It suffices to show $\Phi$ commutes with arrows $[K]\xrightarrow{i_{+}} [L]$ and $[K]\xrightarrow{i^{+}} [L]$ with $i\in I,h(L)=n$.
	 
	 For any $[L]=i_{+}([K])$ with $h(L)=n$, we need to prove $\Phi([L])=i_{+}(\Phi([K])  )$. We choose a path $[L_{0}] \xrightarrow{j_{1+}} [K_{1}] \xrightarrow{j_{2+}}  \cdots \xrightarrow{j_{(n-1)+}} [K_{n-1}]=[K]$ of $K$, then $[L_{0}] \xrightarrow{j_{1+}} [K_{1}] \xrightarrow{j_{2+}}  \cdots \xrightarrow{j_{(n-1)+}} [K_{n-1}] \xrightarrow{i_{+}} [L]$ is a path of $L$. Since $\Phi$ does not depned on the choice of path, we can see that $\Phi([L])=i_{+}j_{(n-1)+}\cdots j_{1+} [\mathbf{1}_{0}]= i_{+} (\Phi([K])) $. Hence $\Phi$ commutes with $i_{+}, i \in I$.  
	 
	 For $[L]=i^{+}([K])$, we prove $\Phi([L])=i^{+}(\Phi([K]))$ case by case. Take $j \in I$ such that $t_{j}(L)>0$. 
	 
	 case(1): If $j \neq i$, then $i^{+}$ commutes with $j_{+}$. Let $K'=j_{-}(K)$, then $j_{+}(\Phi([K']))=\Phi([K])$. We have $$i^{+}\Phi([K])=i^{+}j_{+}(\Phi([K']))=j_{+}i^{+}(\Phi([K'])).$$ Notice that $h(K')=n-2$,  $i^{+}(\Phi ([K']))=\Phi(i^{+}([K']))$ by the induction hyphothesis, so we have $$i^{+}\Phi([K])=j_{+}\Phi (i^{+}([K']))=\Phi(j_{+}i^{+}([K']))=\Phi([L]).$$ 
	 
	 case(2): If $j=i$ and $t_{i}^{\ast}(L)=c>0$, we consider $K_{0}=(i^{-})^{c}(L)$, then $K_{0}=(i^{-})^{c-1}(K)$. It suffices to show that $\Phi([L])=(i^{+})^{c}(\Phi([K_{0}]))$. Take $K'=i_{-}(L)$.
	 
	 If the conditions in Lemma 4.17 and Lemma 6.10 holds for $i$ and $L$, then in this case $t_{i}^{\ast}(K')=c$. Let $K''=(i^{-})^{c}(K')$, then by Lemma 4.17, we can see that $i_{+}(K')=i_{+}(i^{+})^{c}(K'')$ and $K_{0}=i_{+}(K'')$. By a similar argument as in the case $i \neq j$, we have $$(i^{+})^{c}(\Phi([K_{0}]))=(i^{+})^{c}i_{+}\Phi([K''])=i_{+}(i^{+})^{c}\Phi([K''])=i_{+}\Phi([K'])=\Phi([L]).$$ 
     \[
	 \xymatrix{
	 	& L\ar[dl]_{(i^{-})^{c}} \ar[d]_{i_{-}}  \\
	 	K_{0} \ar[d]^{i_{-}}  & K'  \ar[dl]^{(i^{-})^{c}}   \\
	 	K'' &  	
	 }
	 \]
	 \[
	 \xymatrix{
	 	& \Phi([L])  \\
	 	\Phi([K_{0}]) \ar[ur]^{(i^{+})^{c}}  & \Phi([K'])  \ar[u]_{i_{+}}   \\
	 	\Phi([K'']) \ar[ur]_{(i^{+})^{c}}  \ar[u]^{i_{+}}&  	
	 }
	 \]

	 If the conditions in Lemma 4.18 and Lemma 6.11 holds for $i$ and $L$, then we have $t_{i}^{\ast}(K')=c-1$ and $K_{0}=(i^{-})^{c-1}(K')$. In this case $K'=K, i^{+}(K)=i_{+}(K)=L$, then $i^{+}\Phi([K])=i_{+}\Phi([K'])=\Phi([L])$. Hence $\Phi$ commutes with $i^{+}, i \in I$ in any case.\\
	 
	\textbf{Step four}: We prove that $t_{i}(L)=t_{i}(\Phi([L]))$ and $t_{i}^{\ast}(L)=t_{i}^{\ast}(\Phi([L]))$ for $h(L)=n$.
	
	Notice that for any $i,j \in I$ and $L \in \mathcal{P}_{\mathbf{V}}$, $t_{i}(j^{+}(L))$ is uniquely determined by $t_{i}(L)$ and the dimension vector $|\mathbf{V}|$ by Lemma 4.11 and Lemma 4.15, and  $t_{i}(j^{+}(\Phi([L])))$ is also determined by $t_{i}(\Phi([L]))$ and the dimension vector $|\mathbf{V}|$ in the same way by Lemma 6.4 and Lemma 6.8. Since $\Phi$ commutes with $j^{+}$, $\Phi$ preserves the value of $t_{i}$. Replace $j^{+}$ by $j_{+}$, we can see that $\Phi$ also preserves the value of $t_{i}^{\ast}$.
	
\end{proof}

\begin{remark}
	We can naturally identify the set $\mathcal{V}_{2}$ with the set $\bigcup \limits_{\mathbf{V}} Irr \Lambda_{\mathbf{V}}$. Then the arrows $i_{+},i_{-},i^{+},i^{-}$ in $\mathcal{G}_{2}$ coincide with the abstract crystal operator $\tilde{f}_{i},\tilde{e}_{i},\tilde{f}_{i}^{\ast},\tilde{e}_{i}^{\ast}$ on $\bigcup \limits_{\mathbf{V}} Irr \Lambda_{\mathbf{V}}$ respectively. 
	The above theorem is equivalent to the following result in \cite{MR1458969} : $\bigcup \limits_{\mathbf{V}} Irr \Lambda_{\mathbf{V}}$ has an abstract crystal structure and is isomorphic to $B(\infty)$. Our proof does not rely on the abstract crystal theory.(In particular, we do not need the embedding theorem and the criterion of $B(\infty)$).
\end{remark}

As an corollary, we have the following proposition as in \cite{MR1758244}:
\begin{proposition}
	For any irreducible component $Z$, there exists a unique $f_{Z} \in \mathcal{M}$ such that $f_{Z}$ has property $\mathcal{P}(Z)$. Moreover, the set $\{f_{Z}|Z$ is an irreducible component of $\Lambda_{\mathbf{V}}$ for some $\mathbf{V}\}$ is a $\mathbb{Q}$-basis of $\mathcal{M}$ and a $\mathbb{Z}$-basis of $_{\mathbb{Z}}\mathcal{M}$. We call $\{f_{Z}|Z \in \Lambda_{\mathbf{V}}$ for some $\Lambda_{\mathbf{V}}\}$ the semicanonical basis of $\mathcal{M}$.
\end{proposition}
\begin{proof}
	By Lemma 5.4, we can choose an $f_{Z}$ for each $Z$ and consider the $\mathbb{Q}$-vector space $\mathbf{S}=span\{f_{Z}| f_{Z}$ is chosen for $Z \}$, then it is easy to see that these $f_{Z}$ are linear independent and form a basis of $\mathbf{S}$. Let $\mathbf{S}_{\mathbf{V}}=\mathbf{S} \cap \mathcal{M}_{\mathbf{V}}$.  For any  $\mathbf{V}$, we can see that $\dim \mathbf{S}_{\mathbf{V}}=|\{f_{Z}|Z \in Irr\Lambda_{\mathbf{V}} \}|=|\{[f]\in \mathcal{V}_{2}| f \in \mathcal{M}_{\mathbf{V}} \}=|\mathcal{P}_{\mathbf{V}}|=\dim \mathbf{U}^{+}_{q}(\mathfrak{g})$. On the other hand, by Theorem 2.5, we have $\dim \mathcal{M}_{\mathbf{V}}=\dim \mathcal{M}_{\mathbf{V},\Omega}= \dim \mathbf{U}^{+}$. Hence $\mathcal{S}_{\mathbf{V}}=\mathcal{M}_{\mathbf{V}}$ and those chosen $f_{Z}$ form a $\mathbb{Q}$-basis of $\mathcal{M}$. By the construction of $f_{Z}$, it is also a $\mathbb{Z}$-basis of $_{\mathbb{Z}}\mathcal{M}$. 
	To see the uniqueness of $f_{Z}$, we assume that one can choose another $f'_{Z} \in \mathcal{P}$ having property $\mathcal{P}_{Z}$. Express $f'_{Z}$ as a linear combination of the basis $\{f_{Z}|Z \in Irr\Lambda_{\mathbf{V}}\}$ by $f'_{Z}=\sum \limits_{Z'} c_{Z'}f_{Z'}$, then the right hand side generically takes value $c_{Z'}$ on $Z'$. It implies that $c_{Z}=1$ and $c_{Z'}=0$ for the other $Z'$, hence $f'_{Z}=f_{Z}$ is unique. 
\end{proof}

\begin{remark}
	If we admit that $\{f_{Z} \vert Z \in Irr \Lambda_{\mathbf{V}}\}$ is a basis of $\mathbf{U}^{+}$, then it is a basis of canonical type in the sense of \cite{baumann2011canonical}. It has been proved in \cite{baumann2011canonical} that there is an abstract crystal basis structure arising from a basis of canonical type, which is isomorphic to $B(\infty)$.
\end{remark}

By the similar argument as Proposition 3.5, we can prove that the semicanonical basis $\{f_{Z}|Z \in \Lambda_{\mathbf{V}}$ for some $\Lambda_{\mathbf{V}}\}$ is adapted:
\begin{proposition}[\cite{MR1758244} 2.9] We have:\\
	(1)  $\{f_{Z}|t_{i}(Z) \geq d,  Z \in \Lambda_{\mathbf{V}}$ for some $\Lambda_{\mathbf{V}}\} $ form a $\mathbb{Q}$-basis of $\mathbf{1}_{i^{(d)}}\ast \mathcal{M}$ and a $\mathbb{Z}$-basis of $\mathbf{1}_{i^{(d)}}\ast {_{\mathbb{Z}}\mathcal{M}}$. \\
	(2) If $Z'$ is an irreducible component such that $t_{i}(Z')=0$ and $f_{Z}=\eta_{i,d}(f_{Z'})$, then $f_{Z}-\mathbf{1}_{i^{(d)}}\ast f_{Z'} \in \mathbf{1}_{i^{(d+1)}}\ast \mathcal{M} $.\\
	(3) $\{f_{Z}|t_{i}^{\ast}(Z) \geq d,  Z \in \Lambda_{\mathbf{V}}$ for some $\Lambda_{\mathbf{V}}\} $ form a $\mathbb{Q}$-basis of $\mathcal{M}\ast\mathbf{1}_{i^{(d)}}$ and a $\mathbb{Z}$-basis of $_{\mathbb{Z}}\mathcal{M}\ast\mathbf{1}_{i^{(d)}}$.\\
	(4) If $Z'$ is an irreducible component such that $t_{i}^{\ast}(Z')=0$ and $f_{Z}=\eta_{i,d}^{\ast}(f_{Z'})$, then $f_{Z}-f_{Z'}\ast \mathbf{1}_{i^{(d)}} \in \mathcal{M}\ast\mathbf{1}_{i^{(d+1)}} $.
\end{proposition}

\section{Comparision of the canonical basis and the semicanonical basis}

\subsection{An order in the colored graphs}
Assume $|I|=n$ and fix an order $(i_{1}\prec i_{2} \prec \cdots \prec
i_{n})$ of $I$, we define an order of $I \times \mathbb{N}^{+}$ as the following:  $(i_{n_{2}},m_{2})\prec (i_{n_{1}},m_{1})$ if and only if $i_{n_{1}}\prec i_{n_{2}}$ or $i_{n_{1}}=i_{n_{2}}$ and $m_{1}<m_{2}$.

Let $\mathcal{S}$ be the set of sequneces of $I \times \mathbb{N}^{+}$. Given $\underline{s}_{1}=((i_{n_{1}},m_{1} ), (i_{n_{2}},m_{2}),\cdots,(i_{n_{k}},m_{k}) )$ and $\underline{s}_{2}=((i_{n'_{1}},m'_{1} ), (i_{n'_{2}},m_{2}),\cdots,(i_{n'_{l}},m'_{l}) )  \in \mathcal{S}$ such that $\sum \limits_{1 \leq s \leq k}m_{s} i_{n_{s}}=\sum\limits_{1 \leq s \leq l}m'_{s}i_{n'_{s}}$, we say $\underline{s}_{2}\prec \underline{s}_{1}$ if there exists $r\in \mathbb{N}$ such that $(i_{n_{t}},m_{t})=(i_{n'_{t}},m'_{t})$ for $1 \leq t< r$ and $(i_{n'_{r}},m'_{r})\prec (i_{n_{r}},m_{r})$.  Let $\mathcal{S}_{\mathbf{V}}$ be the subset of $\mathcal{S}$ consisting of the sequences $\underline{s}=((i_{n_{1}},m_{1} ), (i_{n_{2}},m_{2}),\cdots,(i_{n_{k}},m_{k}) )$ such that $\sum \limits_{1 \leq s \leq k}m_{s} i_{n_{s}}=|\mathbf{V}|$, then $(\mathcal{S}_{\mathbf{V}},\prec)$ becomes a partially ordered set.

Now we can inductively define a map $\underline{s}^{\prec}=\underline{s}: \mathcal{V}_{1} \rightarrow \mathcal{S}$ as the following: 

(1) If $L$ is the constant sheaf $\bar{\mathbb{Q}}_{l}$ on $\mathbf{E}_{\mathbf{V},\Omega}$ with $|\mathbf{V}|=pi$, we define $\underline{s}([L])=(i,p)$. In particular, If $L$ is the constant sheaf on $\mathbf{E}_{\mathbf{V},\Omega}$ with $\mathbf{V}=0$, we define $\underline{s}([L])=\emptyset $. 

(2) For the other $L \in \mathcal{P}_{\mathbf{V}}$, there exists a unique $r \in \mathbb{N}$ such that $t_{i_{r}}(\mathcal{F}_{\Omega,\Omega_{r}}(L))=n>0$ but $t_{i_{s}}(\mathcal{F}_{\Omega,\Omega_{s}}(L))=0$ for any $r<s$. Then applying Lemma 3.4, we get a simple perverse sheaf $K$ with $t_{i_{r}}(\mathcal{F}_{\Omega,\Omega_{r}}(K))=0$, we define $\underline{s}([L])=((i_{r},n),\underline{s}([K]))$. 

By construction, we can see that $\underline{s}$ is injective.  Now we can define an order of $\mathcal{P}_{\mathbf{V}}$ as the following: $L\prec L' \in \mathcal{P}_{\mathbf{V}}$ if and only if $\underline{s}([L]) \prec \underline{s}([L'])$. 

Notice that $\mathcal{G}_{1}$ is invariant under the Fourier Deligne transform, we can choose an orientation $\Omega$ for any fixed order $\prec$ such that $(I, \prec)$ is adapted to $\Omega$. More precisely, we can choose $\Omega$ such that $i_{l}$ is a sink in the subquiver consisting of vertices $\{i_{l}\} \cup \{i_{j} \in I| i_{l} \prec i_{j}\}$ for any $l$. 

We can also define a map $\underline{s}':\mathcal{V}_{2} \rightarrow \mathcal{S}$ in a similar way, then $\mathcal{V}_{2}$ also becomes a partially ordered set. It is easy to see that $\underline{s}(L)=\underline{s}'(\Phi([L]))$ and $\Phi$ preserves the order $\prec$.

 \subsection{Monomial bases}
 
 Given $\underline{s}=((i_{n_{1}},m_{1} ), (i_{n_{2}},m_{2}),\cdots,(i_{n_{k}},m_{k}))  \in \mathcal{S}$, we difine $m_{\underline{s}}=\mathbf{1}_{i_{n_{1}}^{m_{1}}} \ast \mathbf{1}_{i_{n_{2}}^{m_{2}}} \ast \cdots \ast \mathbf{1}_{i_{n_{k}}^{m_{k}}} \in \mathcal{M}$ and define $[M_{\underline{s}}]=E_{i_{n_{1}}}^{(m_{1})}\ast E_{i_{n_{2}}}^{(m_{2})}\ast \cdots \ast E_{i_{n_{k}}}^{(m_{k})} \in \mathcal{K}$. 
 
 \begin{lemma}
 	Given $Z \in Irr \Lambda_{\mathbf{V}}$ and $f_{Z} \in \mathcal{M}$, then we have $$m_{\underline{s}'(f_{Z})}=f_{Z}+\sum \limits_{\underline{s}'(f_{Z'})\succ \underline{s}'(f_{Z}) } c_{Z,Z'}f_{Z'} .$$
 \end{lemma}

\begin{proof}
	We prove the lemma by induction on the length $k$ of the sequence $$\underline{s}'(f_{Z})=((i_{n_{1}},m_{1} ), (i_{n_{2}},m_{2}),\cdots,(i_{n_{k}},m_{k})).$$
	
	(1) When $k=1$, then $f_{Z}=\mathbf{1}_{i_{n_{1}}^{(m_{1})}}=m_{\underline{s}'(f_{Z})}$. The lemma trivially holds.
	
	(2) For $k>1$, by Lemma 3.4, we can find $\tilde{Z} \in \Lambda_{\mathbf{V}'}$ with $|\mathbf{V}|= |\mathbf{V'}|+m_{1}i_{n_{1}}$ scuh that $\mathbf{1}_{i_{n_{1}}^{m_{1}}} \ast f_{\tilde{Z}}= f_{Z}+\sum \limits_{Z'' \in Irr \Lambda_{\mathbf{V}},t_{i_{n_{1}}}(Z'')>m_{1} } c_{Z''}f_{Z''}$. By the induction hyphothesis, we have $m_{\underline{s}'(f_{\tilde{Z}})}=f_{\tilde{Z}}+\sum \limits_{\underline{s}'(f_{\tilde{Z}'})\succ \underline{s}'(f_{\tilde{Z}}) }c_{\tilde{Z},\tilde{Z}'}f_{\tilde{Z}'}$.
	
	Notice that $\underline{s}'(f_{Z})=((i_{n_{1}},m_{1}), \underline{s}'(f_{\tilde{Z}})) $, we have 
	\begin{align*}
		m_{\underline{s}'(f_{Z})}=& \mathbf{1}_{i_{n_{1}}^{m_{1}}}\ast m_{\underline{s}'(f_{\tilde{Z}})} \\
		=& \mathbf{1}_{i_{n_{1}}^{m_{1}}}\ast (f_{\tilde{Z}}+\sum \limits_{\underline{s}'(f_{\tilde{Z}'})\succ \underline{s}'(f_{\tilde{Z}}) }c_{\tilde{Z},\tilde{Z}'}f_{\tilde{Z}'} ) \\
		=& f_{Z}+\sum \limits_{Z'' \in Irr \Lambda_{\mathbf{V}},t_{i_{n_{1}}}(Z'')>m_{1} } c_{Z''}f_{Z''}+ \sum \limits_{\underline{s}'(f_{\tilde{Z}'})\succ \underline{s}'(f_{\tilde{Z}}) }c_{\tilde{Z},\tilde{Z}'}\mathbf{1}_{i_{n_{1}}^{m_{1}}}\ast f_{\tilde{Z}'} 
	\end{align*}
Notice that if $t_{i_{n_{1}}}(Z'')=m>m_{1}$, then $\underline{s}'(f_{Z''})$ starts with $(i_{n_{1}},m)$ or $(i_{r},m')$ such that $r<n_{1}$. So we have the following fact ($\sharp$):  $\underline{s}'(f_{Z''})\succ \underline{s}'(f_{Z})$ for those $Z''$ with $t_{i_{n_{1}}}(Z'')=m>m_{1}$.

Now it suffices to show that $\mathbf{1}_{i_{n_{1}}^{m_{1}}}\ast f_{\tilde{Z}'}=\sum \limits_{\underline{s}'(f_{Z'''})\succ \underline{s}'(f_{Z}) } d_{Z'''}f_{Z'''}$ for those $\tilde{Z}'$ such that $\underline{s}'(f_{\tilde{Z}'})\succ \underline{s}'(f_{\tilde{Z}})$. For those  $f_{\tilde{Z}'}$ with $t_{i_{n_{1}}}(\tilde{Z}')>0$, we have $f_{\tilde{Z}'}$ belongs to the $\mathbb{Q}$-space $\mathbf{1}_{i_{n_{1}}^{1}} \ast \mathcal{M}$ and $\mathbf{1}_{i_{n_{1}}^{m_{1}}}\ast f_{\tilde{Z}'}$ belongs to the $\mathbb{Q}$-space $\mathbf{1}_{i_{n_{1}}^{(m_{1}+1)}}\ast \mathcal{M}$. Since $\{f_{Z}|t_{i_{n_{1}}} \geq m_{1}+1   \}$ forms a basis of the $\mathbb{Q}$-space $\mathbf{1}_{i_{n_{1}}^{(m_{1}+1)}}\ast\mathcal{M}$ by Proposition 7.4,  $\mathbf{1}_{i_{n_{1}}^{m_{1}}}\ast f_{\tilde{Z}'}=\sum \limits_{t_{i_{n_{1}}}(Z''')>m_{1} } d_{Z'''}f_{Z'''}$. Since $t_{i_{n_{1}}}(Z''')=m>m_{1}$, by the fact $(\sharp)$, we have $\underline{s}'(f_{Z'''})\succ \underline{s}'(f_{Z})$.

On the other hand, for those $f_{\tilde{Z}'}$ with $t_{i_{n_{1}}}(\tilde{Z}')=0$, we have $\mathbf{1}_{i_{n_{1}}^{m_{1}}}\ast f_{\tilde{Z}'}=f_{Z'}+\sum\limits_{Z'' \in Irr \Lambda_{\mathbf{V}},t_{i_{n_{1}}}(Z'')>m_{1}}c_{Z''}f_{Z''}$. Here $Z'$ is an irreducible component with $t_{i_{n_{1}}}=m_{1}$. Since $t_{i_{n_{1}}}(Z'')>m_{1}$, we can see that $\underline{s}'(f_{Z''})\succ \underline{s}'(f_{Z})$ by $(\sharp)$. Now we prove that $\underline{s}'(f_{Z'})\succ \underline{s}'(f_{Z})$.  If $t_{i_{r}}(Z') >0$ for some $r< n_{1}$, then $\underline{s}'(f_{Z}')$ starts with $(i_{r},m')$ with $r<n_{1}$, hence $\underline{s}'(f_{Z'})\succ \underline{s}'(f_{Z})$. Otherwise, $\underline{s}'(f_{Z'})=((i_{n_{1}},m_{1}),\underline{s}'(f_{\tilde{Z}'}))$. Since $\underline{s}'(f_{\tilde{Z}'})\succ \underline{s}'(f_{\tilde{Z}})$, we have $\underline{s}'(f_{Z'})=((i_{n_{1}},m_{1}),\underline{s}'(f_{\tilde{Z}'})) \succ ((i_{n_{1}},m_{1}),\underline{s}'(f_{\tilde{Z}}))=\underline{s}'(f_{Z})$. 

In a conclusion, we have $m_{\underline{s}'(f_{Z})}=f_{Z}+\sum \limits_{\underline{s}'(f_{Z'})> \underline{s}'(f_{Z}) } c_{Z,Z'}f_{Z'} $.

\end{proof}

By basic linear algebra, we have the following result:
\begin{corollary}
	The set $\mathbf{M}'_{\mathbf{V}}= \{m_{\underline{s}'(f_{Z})}|f_{Z} \in \mathcal{V}_{2}$ for $Z\in Irr\Lambda_{\mathbf{V}} \}$ is a $\mathbb{Q}$-basis of $\mathcal{M}_{\mathbf{V}}$ and a $\mathbb{Z}$-basis of $_{\mathbb{Z}}\mathcal{M}_{\mathbf{V}}$. Moreover, the transition matrix between the semicanonical basis $\mathbf{B}_{2,\mathbf{V}}=\{f_{Z}|Z \in Irr \Lambda_{\mathbf{V}} \}$ and $\mathbf{M}'_{\mathbf{V}}$ is upper triangular (with respect to the order of $\mathcal{S}$ and $\mathcal{V}_{2}$) and the diagonal entries of this matrix are all equal to $1$.
\end{corollary}
\begin{proof}
	With the lemma above, we can easily see that $\{m_{\underline{s}'(f_{Z})}|f_{Z} \in \mathcal{V}_{2}$ for $Z\in Irr\Lambda_{\mathbf{V}} \}$ and $\{f_{Z}| f_{Z} \in \mathcal{V}_{2}$ for $Z\in Irr\Lambda_{\mathbf{V}} \}$ can be linear expressed by each other. Hence $\mathbf{M}'_{\mathbf{V}}= \{m_{\underline{s}'(f_{Z})}|f_{Z} \in \mathcal{V}_{2}$ for $Z\in Irr\Lambda_{\mathbf{V}} \}$ is a $\mathbb{Q}$-basis of $\mathcal{M}_{\mathbf{V}}$. Notice that $c_{Z,Z'}= \rho_{Z'}(m_{\underline{s}'(f_{Z}) }) \in \mathbb{Z}$, we can see that it is also a $\mathbb{Z}$-basis of $_{\mathbb{Z}}\mathcal{M}_{\mathbf{V}}$.
\end{proof}

Since $\{[L]| L \in \mathcal{P}_{\mathbf{V}} \}$ shares the similar adapted property of $\{f_{Z}| Z \in Irr\Lambda_{\mathbf{V}} \}$, we can also prove the following results in a similar way:
\begin{lemma}
	Given $L \in \mathcal{P}_{\mathbf{V}}$, then the following equation holds in $\mathcal{K}$: $$[M_{\underline{s}(L)}]=[L]+\sum \limits_{\underline{s}([L'])\succ \underline{s}([L]) } c_{L,L'}[L']. $$ 
\end{lemma}
\begin{corollary}
	The set $\mathbf{M}_{\mathbf{V}}= \{m_{\underline{s}(L)}|L \in \mathcal{V}_{1} \}$ forms a $\mathbb{Z}[v,v^{-1}]$-basis of $\mathcal{K}_{\mathbf{V}}$. Moreover, the transition matrix between the canonical basis $\mathbf{B}_{1,\mathbf{V}}=\{[L]|L \in \mathcal{P}_{\mathbf{V}} \}$ and $\mathbf{M}_{\mathbf{V}}$ is upper triangular (with respect to the order of $\mathcal{S}$ and $\mathcal{V}_{1}$) and the diagonal entries of this matrix are all equal to $1$.
\end{corollary}
\begin{remark}
	For any simple perverse sheaf $L\in \mathcal{P}_{\mathbf{V},\Omega}$, we can define a linear map $\rho_{L}:\mathcal{K} \rightarrow \mathbb{Z}[v,v^{-1}]$ as the following:
	$\rho_{L}([L'])= \sum \limits_{k \in \mathbb{Z}}dim \mathbf{Hom}_{\mathcal{Q}_{\mathbf{V}}}(L[k],L')v^{k}$ for any $L' \in \mathcal{Q}_{\mathbf{V}}$ and $\rho_{L}([L'])=0$ for $L'\in \mathcal{Q}_{\mathbf{V}'}$ with $|\mathbf{V}| \neq |\mathbf{V'}|$. Then we have $c_{L,L'}=\rho_{L'}(M_{\underline{s}(L)} ). $
\end{remark}

\subsection{ The transition matrix between the canonical basis and the semicanonical basis}
Recall that there are isomorphisms between algebras $(\sigma: \mathcal{K} \rightarrow {_{\mathbb{Z}}\mathbf{U}_{q}}(\mathfrak{g})^{+}$ and $\varkappa:{_{\mathbb{Z}}\mathcal{M}} \rightarrow { _{\mathbb{Z}}\mathbf{U}}(\mathfrak{g})^{+}$. Let $\varsigma: \mathcal{K} \rightarrow  {_{\mathbb{Z}}\mathbf{U}}(\mathfrak{g})^{+}$ be the composition of $\sigma$ and taking the classical limits. Then $\varsigma(\mathbf{B}_{1})$ is a basis of $_{\mathbb{Z}}\mathbf{U}(\mathfrak{g})^{+}$, we still denote the basis by $\mathbf{B}_{1}$. Similarly, $\varkappa(\mathbf{B}_{2})$ is a basis of $_{\mathbb{Z}}\mathbf{U}(\mathfrak{g})^{+}$, we still denote it by $\mathbf{B}_{2}$.

\begin{center}
	$\mathcal{K} \xrightarrow{\varsigma} {_{\mathbb{Z}}\mathbf{U}}(\mathfrak{g})^{+} \xleftarrow{\varkappa} {_{\mathbb{Z}}\mathcal{M}}$ 
\end{center}

\begin{theorem}
	The transition matrix between the bases $\mathbf{B}_{1}=\bigcup \limits_{\mathbf{V}}\mathbf{B}_{1,\mathbf{V}} $ and $\mathbf{B}_{2}=\bigcup \limits_{\mathbf{V}}\mathbf{B}_{2,\mathbf{V}}$ of ${_{\mathbb{Z}}\mathbf{U}}(\mathfrak{g})^{+}$ is  upper triangular (with respect to the order of $\mathcal{S}$) and the diagonal entries of this matrix are all equal to $1$. More precisely, assume that $L \in \mathcal{P}_{\mathbf{V}}$ and $f_{Z}=\Phi([L])$, then $\varsigma([L])=\varkappa(f_{Z}) +\sum\limits_{\underline{s}'(f_{Z'})\succ \underline{s}'(f_{Z}) } c_{Z'}\varkappa (f_{Z'})$ with $c_{z'} \in \mathbb{Z}$.
\end{theorem} 
\begin{proof}
	Notice that $\underline{s}(L)=\underline{s}'(\Phi(L))$, we have $\varkappa(m_{\underline{s}'(\Phi(L) )})= \varsigma ( [M_{\underline{s}(L)}]) $, hence $\varkappa(\mathbf{M}')=\varsigma(\mathbf{M})=\tilde{\mathbf{M}}$ and $\tilde{\mathbf{M}}$ is a  basis of $\mathbf{U}(\mathfrak{g})^{+}$.Let $P_{i},i=1,2$ be the transition matrix $P_{i}$ between $\tilde{\mathbf{M}}$ and $\mathbf{B}_{i},i=1,2$ respectively. By Lemma 8.3, let $P_{1}=(a_{\underline{s}(L),[L']})_{[L],[L'] \in \mathcal{V}_{1}} $,  we can see that

	\begin{center}
		$a_{\underline{s}(L),[L']}= \left\{
		\begin{aligned}
		&	\rho_{L'}([M_{\underline{s}(L)}])|_{v=1} & \underline{s}(L) \prec \underline{s}(L') \\
		&	1& L=L' \\
		&	0& otherwise
		\end{aligned}
		\right. $
	\end{center} 
 Similarly, let $P_{2}=(b_{\underline{s}'(f_{Z}),f_{Z'} })_{f_{Z},f_{Z'} \in \mathcal{V}_{2}} $, we have
 	\begin{center}
 	 $b_{\underline{s}'(f_{Z}),f_{Z'}}= \left\{
 	\begin{aligned}
 		&	\rho_{Z'}(m_{\underline{s}'(f_{Z})}) & \underline{s}'(f_{Z}) \prec \underline{s}'(f_{Z'}) \\
 		&	1& f_{Z}=f_{Z'} \\
 		&	0& otherwise
 	\end{aligned}
 	\right. $
 \end{center} 
 Since $\Phi$ preserves the order of $\mathcal{V}_{i}$, we can see that the transition matrix $P_{2}P_{1}^{-1}$ between the bases $\mathbf{B}_{1}$ and $\mathbf{B}_{2}$ is upper triangular (with respect to the order of $\mathcal{V}_{1}$ and  $\mathcal{V}_{2}$) and with diagonal entries equal to $1$.
\end{proof}

\begin{remark}
	We can see that the theorem does not depend on the choice of the order of $I$. More precisely, if we choose another order $\tilde{\prec}$ of $I$, then $\tilde{\prec}$ induces an order of $\mathcal{S}$. We can also define $\underline{s}^{\tilde{\prec}}:\mathcal{V}_{1} \rightarrow \mathcal{S}$ and $\underline{s}'^{\tilde{\prec}}:\mathcal{V}_{2} \rightarrow \mathcal{S}$ in a similar way. Then with the notation above, we still have $\varsigma([L])=\varkappa(f_{Z}) +\sum\limits_{\underline{s}'^{\tilde{\prec}}(f_{Z'})\tilde{\succ} \underline{s}'^{\tilde{\prec}}(f_{Z}) } c_{Z'}\varkappa (f_{Z'})$. Moreover, the transition matrix $P_{2}P_{1}^{-1}$ does not depend on the choice of $\prec$ (up to a permutation).
\end{remark}

\begin{definition}
	Given $L,K \in \mathcal{P}_{\mathbf{V}}$, we say $L \preceq K$ if and only if for any order $\prec$ of $I$ and the induced map $\underline{s}^{\prec}: \mathcal{V}_{1} \rightarrow \mathcal{S}$, we have $\underline{s}^{\prec}([L]) \prec \underline{s}^{\prec}([K])$. Similarly, given $f_{Z},f_{Z'} \in \mathcal{V}_{2}$, we say $f_{Z} \preceq' f_{Z'}$ if and only if for any order $\prec$ of $I$ and the induced map $\underline{s}'^{\prec}: \mathcal{V}_{2} \rightarrow \mathcal{S}$, we have $\underline{s}'^{\prec}(f_{Z}) \prec \underline{s}'^{\prec}(f_{Z'})$. Then $(\mathcal{V}_{1}, \preceq)$ and $(\mathcal{V}_{2},\preceq')$ are partially ordered sets and $L \preceq K$ if and only if $\Phi([L]) \preceq' \Phi([K])$.
\end{definition}
\begin{corollary}
		The transition matrix between the bases $\mathbf{B}_{1}$ (with the order $\preceq$) and $\mathbf{B}_{2}$ (with the order $\preceq'$) of ${_{\mathbb{Z}} \mathbf{U}}(\mathfrak{g})^{+}$ is  upper triangular  and the diagonal entries of this matrix are all equal to $1$. More precisely, assume that $L \in \mathcal{P}_{\mathbf{V}}$ and $f_{Z}=\Phi([L])$, then $\varsigma([L])=\varkappa(f_{Z}) +\sum\limits_{f_{Z} \preceq' f_{Z'} } c_{Z'}\varkappa (f_{Z'})$.
\end{corollary}

\begin{remark}
	Since both $\mathbf{B}_{1}$ and $\mathbf{B}_{2}$ are adapted, it's easy to see that they are bases of canonical type in the sense of \cite{baumann2011canonical}. Baumann in \cite{baumann2011canonical} has already proved the similar result as the following: The transition matrix between two bases of canonical type is upper triangular and the diagonal entries of this matrix are all equal to $1$.
\end{remark}

\begin{remark}
	If the graph $\mathbf{\Gamma}$ is of type $A,D,E$, by \cite{MR1360930} (or \cite{MR1227098}) the transition matrix between the PBW basis and the canonical basis is  upper triangular  with diagonal entries equal to $1$. Hence, we can see that the transition matrix between the PBW basis and the semicanonical basis is  upper triangular  with diagonal entries equal to $1$ in this case, which includes the result of \cite{MR3213879} in type $A$.
\end{remark}

With Proposition 4.8 and Proposition 7.5, we have the following result:
\begin{proposition}[\cite{MR1758244},\cite{MR1088333}]
	Fix $ \underline{d}=(d_{i})_{i \in I}$ with $d_{i}\in \mathbb{N}$.  Then the following submodules of ${_{\mathbb{Z}}\mathbf{U}}(\mathfrak{g})^{+}$ are equal:\\
	(1) $\sum\limits_{i \in I} {_{\mathbb{Z}}\mathbf{U}}(\mathfrak{g})^{+}\ast e_{i}^{(d_{i})}  $, here $e_{i}^{(d_{i})}$ is the divided power in ${_{\mathbb{Z}}\mathbf{U}}(\mathfrak{g})^{+}$.\\
	(2) The $\mathbb{Z}$-span of $\mathbf{B}_{1,\underline{d}}= \{\varsigma([L])|t_{i}^{\ast}(\mathcal{F}_{\Omega,\Omega^{i}}(L)) \geq d_{i}$ for some $i \in I\} $.\\
	(3) The $\mathbb{Z}$-span of $\mathbf{B}_{2,\underline{d}}= \{\varkappa(f_{Z})|t_{i}^{\ast}(Z) \geq d_{i}$ for some $i \in I\} $.
\end{proposition}

Given $\lambda \in \mathfrak{h}^{\ast}$ such that $\lambda(h_{i}) \leq 0$ for any $i$, ($\{h_{i},i \in I\}$ is a basis of $\mathfrak{h}$) let $d_{i}=1-\lambda(h_{i}) $, then $V(\lambda) \cong  \mathbf{U}(\mathfrak{g})^{+}/ \sum\limits_{i \in I} \mathbf{U}(\mathfrak{g})^{+}\ast e_{i}^{(d_{i})}$ admits a $\mathbf{U}(\mathfrak{g})$-module structure such that $e_{i}$ acts by left multiplication.  Let $\pi: \mathbf{U}^{+}(\mathfrak{g}) \rightarrow V(\lambda)$ be the obvious quotient map, then we have:
\begin{proposition} [\cite{MR1758244},\cite{MR1088333}]
	With the notations above,  set  $\mathbf{B}_{1}(\lambda)=\{\pi(\varsigma([L]))| [L] \in \mathbf{B}_{1}, [L] \notin \mathbf{B}_{1,\underline{d}} \}$ and $\mathbf{B}_{2}(\lambda)=\{\pi(\varkappa(f_{Z}))| \kappa(f_{Z}) \in \mathbf{B}_{2}, \kappa(f_{Z})\notin \mathbf{B}_{2,\underline{d}} \}$. Then $\mathbf{B}_{1}(\lambda)$ and $\mathbf{B}_{2}(\lambda)$ are $\mathbb{Q}$-basis of $V(\lambda)$. 
\end{proposition}

\begin{definition}
The order $\preceq$ of $\mathcal{V}_{1}$ naturally induced a  partially order of $\mathbf{B}_{1}(\lambda)$ as the following:  Given $L,K$ such that $t^{\ast}_{i} \geq d_{i}$ for some $i$, we say $\pi(\chi([L])) \preceq \pi (\chi([K]))$ if and only if $L \preceq K$. We still denote the partially order of $\mathbf{B}_{1}(\lambda)$ by $\preceq$. Similarly, we can define a partially order $\preceq'$ on $\mathbf{B}_{2}(\lambda)$.
\end{definition}

Since $\Phi$ preserves $t_{i}^{\ast}$, we can see that, for $\Phi([L])=f_{Z}$, $\varsigma([L]) \in \mathbf{B}_{1,\underline{d}}$ if and ony if $\varkappa(f_{Z}) \in \mathbf{B}_{2,\underline{d}}$. By Corollary 8.9 and Proposition 8.10, we have the following corollary:

\begin{corollary}
	The transition matrix between the bases $\mathbf{B}_{1}(\lambda)$ (with the order $\preceq$) and $\mathbf{B}_{2}(\lambda)$ (with the order $\preceq'$) of $V(\lambda)$ is  upper triangular  and the diagonal entries of this matrix are all equal to $1$. 
\end{corollary}

\bibliography{mybibfile}

\end{document}